\documentclass[11pt,twoside]{amsart}

\usepackage[latin1]  {inputenc}%
\usepackage[T1]      {fontenc }%
\usepackage          {amsmath }%

\usepackage          {amsfonts}%
\usepackage          {amssymb }%
\usepackage          {amsthm  }%
\usepackage          {a4wide  }%
\usepackage          {url     }%
\usepackage          {tikz    }%
\usepackage[bookmarks=false,pdfborder={0 0 0.05}]{hyperref}
\usepackage[all]{xy}

\usepackage{enumerate, amsmath, amsfonts, amssymb, amsthm,  wasysym, graphics, graphicx, xcolor, frcursive,comment,bbm}

\usepackage{etex}
\usepackage{MnSymbol}

\usepackage{bm}

\definecolor{darkblue}{rgb}{0.0,0,0.7} 
\definecolor{darkred}{rgb}{0.7,0,0} 

\usepackage{hyperref}
\usepackage[all]{xy}
\usepackage[T1]{fontenc}

\usepackage{array}

\usepackage{float}
\restylefloat{table}

\def\defn#1{{\sf #1}}

\newcommand{\RR}{\mathbb R}
\newcommand{\ZZ}{\mathbb Z}

\newcommand{\spanz}{\mathrm{span}_{\mathbb{Z}}}
\newcommand{\spanr}{\mathrm{span}_{\mathbb{R}}}

\DeclareMathOperator{\Red}{Red}

\DeclareMathOperator{\idop}{id}

\DeclareMathOperator{\Redd(c)}{Red_T(c)^{gen}}

\DeclareMathOperator{\GL}{GL}
\DeclareMathOperator{\TR}{tr}

\DeclareMathOperator{\Aut}{Aut}

\DeclareMathOperator{\END}{End}

\DeclareMathOperator{\Z}{\mathbb{Z}}
\DeclareMathOperator{\PP}{\mathbb{P}}
\DeclareMathOperator{\NN}{\mathbb{N}}
\DeclareMathOperator{\TT}{\mathbb{T}}

\DeclareMathOperator{\XX}{\mathbb{X}}
\DeclareMathOperator{\DD}{\mathcal{D}}
\DeclareMathOperator{\COH}{coh}
\DeclareMathOperator{\MOD}{mod}

\DeclareMathOperator{\DB}{\mathcal{D}^b}
\DeclareMathOperator{\ARS}{\Phi_{aff}}
\DeclareMathOperator{\ARSS}{\Gamma_{aff}}

\DeclareMathOperator{\FRS}{\Phi_{fin}}
\DeclareMathOperator{\PFRS}{\Phi_{fin}^+}
\DeclareMathOperator{\VFIN}{V_{fin}}
\DeclareMathOperator{\VAFF}{V_{aff}}

\DeclareMathOperator{\HOM}{Hom}
\DeclareMathOperator{\EXT}{Ext}

\DeclareMathOperator{\Thi}{Thick}
\DeclareMathOperator{\cox}{cox}
\DeclareMathOperator{\Fix}{Fix}
\DeclareMathOperator{\fin}{fin}

\newtheorem{theorem}{Theorem}[section]
\newtheorem{corollary}[theorem]{Corollary}
\newtheorem{Proposition}[theorem]{Proposition}
\newtheorem{Lemma}[theorem]{Lemma}

\theoremstyle{definition}
\newtheorem{Definition}[theorem]{Definition}
\newtheorem{propdef}[theorem]{Proposition and Definition}
\newtheorem{remark}[theorem]{Remark}
\newtheorem{hyp}[theorem]{Hypothesis}
\newtheorem{example}[theorem]{Example}

\AtBeginDocument{%
   \def\MR#1{}
}

\title[Extended Weyl groups, Hurwitz transitivity and weighted projective lines I]{Extended Weyl groups, Hurwitz transitivity and weighted projective lines I: Generalities and the tubular case}

\author[B.~Baumeister]{Barbara Baumeister}
\address{Barbara Baumeister, Universit\"at Bielefeld, Germany}
\email{b.baumeister@math.uni-bielefeld.de}
\thanks{}

\author[P.~Wegener]{Patrick Wegener}
\address{Patrick Wegener, Leibniz Universit\"at Hannover, Germany}
\email{patrick.wegener@math.uni-hannover.de}

\author[S.~Yahiatene]{Sophiane Yahiatene}
\address{Sophiane Yahiatene, Universit\"at Bielefeld, Germany}
\email{syahiate@math.uni-bielefeld.de}

\subjclass[2010]{Primary 06B15, 05E10, 20F55, 05E18}
\keywords{category of coherent sheaves, Coxeter groups, extended Weyl groups, Hurwitz action, thick subcategories}

\date{\today}

\begin{document}
\newcolumntype{C}[1]{>{\centering\arraybackslash}m{#1}}

\begin{abstract}
We start the systematic study of extended Weyl groups, and continue the combinatorial description of thick subcategories in hereditary categories started by Ingalls-Thomas \cite{IT09}, Igusa-Schiffler-Thomas \cite{IS10} and Krause \cite{Kra12}. We show that for a weighted projective line $\XX$ there exists an order preserving bijection between the thick subcategories of $\COH(\mathbb{X})$ generated by an exceptional sequence and a subposet of the interval poset of a Coxeter transformation $c$ in the Weyl group of a simply-laced extended root system if the Hurwitz action is transitive on the reduced reflection factorizations of $c$ that generate the Weyl group. By using combinatorial and group theoretical tools we show that this assumption on the transitivity of the Hurwitz action is fulfilled for a weighted projective line $\mathbb{X}$ of tubular type. 
\end{abstract}
\maketitle

\tableofcontents

\section{Introduction}\label{Sec:Intro1}
This work is the first in a series of papers that study extended and elliptic root systems as well as their Weyl groups (introduced in Sections~\ref{sec:RootSystemDerivedCategory} and \ref{sec:Elliptic}, see \cite{Sai85, ST97, SY00}). They appear naturally in different contexts as for instance in singularity theory (see \cite{BWY21}) or in the theory of hereditary categories (see \cite{STW16, GL87, Hap01}). We will apply the obtained results to get a combinatorial description of thick subcategories in hereditary categories (see Theorem~\ref{thm:MainElliptic} and Theorems 1.2 and 1.4 in \cite{BWY21}), and will also apply them to singularity theory (see \cite{BWY21}).

We begin by introducing the notion of \defn{extended Coxeter-Dynkin diagrams}. We use it to define \defn{extended spaces}, that is, certain real vector spaces equipped with a basis and a symmetric bilinear form, as well as \defn{extended root systems} (Section~\ref{sec:Notation}). An \defn{extended Weyl system} $(W,S)$ is a group $W$ that is generated by a set of involutions $S$ related to the basis, called simple reflections. We will also make use of the fact that the group $W$ is generated as well by the entire set of reflections $T:= \bigcup_{w \in W} w^{-1} S w$. 

An extended Weyl group is equipped with a bilinear form, whose signature determines the structure of the group (see Section~\ref{subsec:ExtendedSpace}). Three different types are appearing, which we call \defn{tubular}, \defn{domestic} and \defn{wild}, respectively. If the type is tubular, then the group is \defn{elliptic} (see below), and if the type is domestic then there is a subset $Q$ of the set of all reflections $T$ of $W$ such that $(W,Q)$ is an affine Coxeter system (see \cite[Remark~2.5(c)]{BWY21}). If the type is wild, then $W$ is a split extension of a Coxeter group of indefinite type by a free abelian group of finite rank.

We define an analogue of a Coxeter element, the \defn{Coxeter transformation} $c$, and determine its reflection length $\ell_T(c)$, which is the length of a shortest $T$-word for $c$ (see Section~\ref{sec:reflength}).

\begin{Proposition}\label{lem:redCoxElt}
	Let $(W,S)$ be an extended Weyl system and $c$ a Coxeter transformation in $W$.  
	Then $\ell_T(c)$ equals  the rank $\mathrm{rk}(W) = |S|$ of the extended Weyl system.
\end{Proposition}

Then we use the length function $\ell_T$ on $(W,S)$ to introduce some interesting interval posets, which are further studied in \cite{BMN25}. They also play a major role in our application of the theory to representation theory of algebras, more precisely the combinatorial description of hereditary categories.

Hereditary categories are an important tool in the representation theory of algebras as they serve as prototypes for many phenomena appearing there (see for instance \cite{HTT07, Hap01, Rei98}).  According to Happel there are up to derived equivalence two types of connected hereditary ext-finite categories with tilting object over an algebraically closed field, namely the category $\MOD(A)$ of finitely generated $A$-modules for some finite dimensional hereditary $k$-algebra $A$ and the category $\COH(\XX)$ of coherent sheaves over a weighted projective line (see \cite{GL87}) over a field $k$ (\cite[Theorem 3.1]{Hap01}). Note that $\COH(\XX)$ is derived equivalent to the category of finitely generated modules over a canonical algebra (see \cite{RI84}). One way to study the structure of a hereditary category (resp. the bounded derived category of a hereditary category) is via its poset of thick subcategories ordered by inclusion (see for instance \cite{HK16}).

Motivated by Happel's theorem \cite{Hap01}, we transfer the results for the categories of finite dimensional modules to the category $\COH(\XX)$ in a series of two papers. As in \cite{STW16} we consider an algebraically closed field $k$ of characteristic zero. For a finite dimensional hereditary $k$-algebra $A$ the thick subcategories in $\MOD(A)$ that are generated by so called \defn{exceptional sequences} (see Section~\ref{subsec:NotMisProjLine}) have been described combinatorially. This has been first done in a seminal paper by Ingalls and Thomas \cite{IT09} for the case of a quiver algebra $A$. Subsequently, it was accomplished in full generality by Igusa, Schiffler and Thomas \cite{IS10} as well as Krause \cite{Kra12} and Hubery-Krause \cite{HK16}. The category $\MOD(A)$ is naturally equipped with a root system as well as a Weyl group $W = W(A)$, which is in fact a Coxeter group. Its set of simple reflections $S$ is obtained from a \defn{complete} exceptional sequence consisting of the complete set of representatives of the isomorphism classes of the simple $A$-modules. The product of the elements in $S$ in a suitable order is a Coxeter element $c$  of $(W,S)$, which is induced by a functor of the category, the so called Auslander-Reiten functor (see for instance \cite{HK16}).

The results of Ingalls-Thomas, Igusa-Schiffler-Thomas and Hubery-Krause \cite{HK16} provide an isomorphism between the poset of the thick subcategories of $\MOD(A)$   that are generated by exceptional sequences and a combinatorial object, the poset of \defn{non-crossing partitions} NC$(W,c)$. The latter poset  consists of all the elements of $W$ that are in the interval $[1,c]$ with respect to some order relation on $W$, the \defn{absolute order} (see Definition~\ref{def:PrefixPoset}). That isomorphism yields a combinatorial description of the poset of the thick subcategories of $\MOD(A)$ that are generated by exceptional sequences.

We associate a simply-laced extended root system $\Phi$ to the category $\COH(\XX)$ following the approach of \cite{STW16} (see Section~\ref{sec:RootSystemDerivedCategory}). An exceptional sequence in the bounded derived category of $\COH(\XX)$ is used to obtain a  set of roots, from which one can derive the extended root system $\Phi$. In particular, the group generated by the reflections in the hyperplanes othogonal to the elements in $\Phi$ is an extended Weyl group $W=W_{\Phi}$. Its set of simple reflections $S$ is  the set of reflections with respect to the roots induced by a fixed complete exceptional sequence (see Section~\ref{sec:root_star_quiver}). Notice that one distinguishes three different representation types for $\COH(\XX)$: it can be of \defn{domestic}, \defn{tubular} or of \defn{wild} type. These three types correspond to the three different types of extended Weyl groups.

There is an action, the so called \defn{Hurwitz action}, of the braid group $\mathcal{B}_n$ on the set of reduced reflection factorizations of a Coxeter transformation $c$, where $n = |S|$ (see Section \ref{subsec:Hurwitz}). If $(W,S)$ is a Coxeter system, then for every factorization of a Coxeter element into reflections, the set of involved reflections generates the whole group $W$. This is not necessarily true anymore for extended Weyl groups that are elliptic (see Example~\ref{Example:Non-generating}). Therefore in the study of $\COH(\XX)$ we have to replace the set of all reduced $T$-factorizations by the set 
$$
\Redd(c):=\{ (t_{1},\ldots,t_{n})\in T^{n} \mid  c = t_{1}\cdots t_{n}~\mbox{is reduced and generating} \}
$$
consisting of all the reduced reflection factorizations of the Coxeter transformation $c$ that are generating (see Definition~\ref{def:Generating}). We also need replacing the generalized non-crossing partitions, that is the \defn{interval poset} $[\idop,c]$, by the  subposet $[\idop, c]^{\text{gen}}$ of $[\idop, c]$ that consists of all the elements $w \in [\idop, c]$ that possess a factorization which can be extended to a reduced generating factorization of $c$ (see Section~\ref{sec:IntervalPoset}). In \cite{BWY21} it is shown that for $\COH(\XX)$ of domestic or wild type every reduced reflection factorization of the Coxeter transformation in the extended Weyl group is generating.

In this paper we prove the following theorem for the category $\COH(\XX)$ for all three types under two further assumptions on the set $\Redd(c)$ (see Section~\ref{sec:MainProof}). Both will turn out to be redundant.

\begin{theorem} \label{conj:WeightProjElliptic0}
	Let $\XX$ be a weighted projective line over the algebraically closed field $k$ of characteristic zero, $W$ the corresponding extended Weyl group with set of reflections $T$ and $c \in W$ a Coxeter transformation. Furthermore we assume that
	\begin{itemize}
		\item[(a)] the Hurwitz action is transitive on the set $\Redd(c)$;
		\item[(b)] if $(t_{1},\ldots,t_{n}),(r_{1},\ldots,r_{n}) \in \Redd(c)$ such that $t_{1}\cdots t_{k}=r_{1}\cdots r_{k}$ for some $1 \leq k \leq n$, then $(t_{1},\ldots,t_{k},r_{k+1},\ldots,r_{n}) \in \Redd(c)$.
	\end{itemize}
	Then there exists an order preserving bijection between
	\begin{itemize}
		\item the poset of thick subcategories of $\COH(\XX)$ that are generated by an exceptional sequence and ordered by inclusion; and 
		\item the subposet $[\idop, c]^{\mathrm{gen}}$ of the interval poset $[\idop, c]$.
	\end{itemize}
\end{theorem}

We will show in this paper that conditions (a) and (b) are (almost) redundant if $W$ is elliptic (see below). For the remaining types we  show (a)  in \cite{BWY21}, while (b) is obvious as for these types $[\idop, c]^{\mathrm{gen}} = [\idop, c]$.

Theorem~\ref{conj:WeightProjElliptic0}  itself will be proved in Section~\ref{sec:MainProof}. The bijection in the theorem sends the thick subcategory that  is generated by an exceptional sequence $(E_1, \ldots , E_r)$ to $s_{[E_1]} \cdots s_{[E_r]}$ where $s_{[E_i]}$ is the reflection associated to the exceptional object $E_i$. This implies that there exists a bijection between the isomorphism classes of exceptional sequences of $\COH(\XX)$ and their corresponding reflection factorizations. Further this map is equivariant for the action of the braid group.

The last three sections are devoted to the introduction of elliptic  Weyl groups, their study and application to the weighted projective line of tubular type. In Section \ref{sec:Elliptic} we introduce the notion of an elliptic root system, which is due to Saito \cite{Sai85}, and we further recall the structure of an elliptic root system as well as its related  reflection group, the elliptic Weyl group (see Sections~\ref{subsec:Basis} and \ref{subsec:EllipticWeylGroup}). In particular, we remind that for each elliptic root system there is an elliptic Dynkin diagram that  determines the elliptic root system up to isomorphism (see Proposition~\ref{thm:SaitoClassification}). We show that in the tubular case the extended Coxeter-Dynkin diagram of the extended root system attached to the category $\COH(\XX)$ is an elliptic Dynkin diagram with respect to an elliptic root basis.  As a consequence we obtain that if $\XX$ is of tubular type, then the extended root system $\Phi$ attached to $\COH(\XX)$ is elliptic and simply-laced (see Corollary \ref{prop:TubEllRoot}). Therefore, we call $\Phi$ a \defn{tubular elliptic root system} and the corresponding Weyl group \defn{tubular elliptic Weyl group} (see Section~\ref{subsec:TubEllRootSystem}). Note that they do not exhaust the classes of elliptic root systems  and Weyl groups, respectively.

In the last two sections we show that conditions (a) and (b) of Theorem~\ref{conj:WeightProjElliptic0} hold for tubular elliptic Weyl groups.

\begin{theorem} \label{thm:MainElliptic}
Let $\Phi$ be a tubular elliptic root system, $\Gamma= \Gamma(\Phi)$ a tubular elliptic root basis and $c \in W$ a Coxeter transformation with respect to $\Gamma$. If $\Phi$ is not of type $D_4^{(1,1)}$, then the Hurwitz action is transitive on the set $\Redd(c)$. If $\Phi$ is of type $D_4^{(1,1)}$, then the Hurwitz action has two orbits on $\Redd(c)$.
\end{theorem}

\medskip
\noindent Notice that this result was already obtained by Kluitmann by different means and in a different context in his PhD thesis \cite{Klu87} for the tubular elliptic root systems $E_6^{(1,1)},~E_7^{(1,1)}$ and $E_8^{(1,1)}$. 

Condition (b) of Theorem \ref{conj:WeightProjElliptic0} will be shown for elliptic Weyl groups in Section~\ref{sec:rigid}. In this case we call $[\idop, c]^{\mathrm{gen}}$ \defn{rigid}. As a consequence of this as well as of Theorems ~\ref{conj:WeightProjElliptic0} and \ref{thm:MainElliptic} we obtain the combinatorial description of the set of thick subcategories of $\COH(\XX)$ for $\XX$ of tubular type, which we aimed to produce.

\begin{theorem} \label{cor:MainTheorem}
For a weighted projective line $\XX$ of tubular type, but not of type $D_4^{(1,1)}$ over an algebraically closed field $k$ of characteristic zero  there exists an order preserving bijection between
	\begin{itemize}
		\item the set of thick subcategories of $\COH(\XX)$ that are generated by an exceptional sequence in $\COH(\XX)$; and
		\item the subposet $[\idop, c]^{\mathrm{gen}}$ of the interval poset $[\idop, c]$.
	\end{itemize}
If $\XX$ is of type $D_4^{(1,1)}$, then the assertion holds for the subposet $[\idop, c]_D^{\mathrm{gen}}$ of $[\idop, c]^{\mathrm{gen}}$ instead of $[\idop, c]^{\mathrm{gen}}$ (see Definition \ref{Def:SubposetD}).
\end{theorem}

Note that as a consequence of Theorem~\ref{thm:MainElliptic} the bijection is more elaborated in type $D_4^{(1,1)}$. The bijections given in Theorem \ref{cor:MainTheorem} have applications in both directions. At the moment we are working on a nice combinatorial description of the poset $[\idop, c]^{\mathrm{gen}}$, which will also yield, via the bijection, a description of the  Schur roots in the Grothendieck group $K_0(\XX)$ (for the definition of a Schur root see Section~\ref{subsec:ExSequence}).

\bigskip
\subsection*{Acknowledgments}
The authors specially thank Henning Krause for mentioning the question of the existence of a bijection as presented in Theorem~\ref{cor:MainTheorem} to us. Further they thank him, as well as Dieter Vossieck and Lutz Hille for fruitful discussions. We also like to thank Charly Schwabe, who made us aware of a mistake in Theorem~\ref{thm:MainElliptic}.

\newpage

\section{Extended root systems and extended Weyl groups}\label{sec:Notation}

In this section we introduce notation that we use throughout this paper. In particular
we recall the definition of a generalized root system and define the extended Weyl groups, which  can be partially found in \cite{STW16} or \cite{Sai85}.

\subsection{Some basic notation}\label{subsec:notation}
Let $V$ be a finite dimensional $\RR$-space and $(-\mid -)$ a symmetric bilinear form on $V$. A vector $\alpha$ is non-isotropic, if it holds $(\alpha \mid \alpha)\neq 0$, and  the reflection of $V$ with respect to the non-isotropic $\alpha$ is 
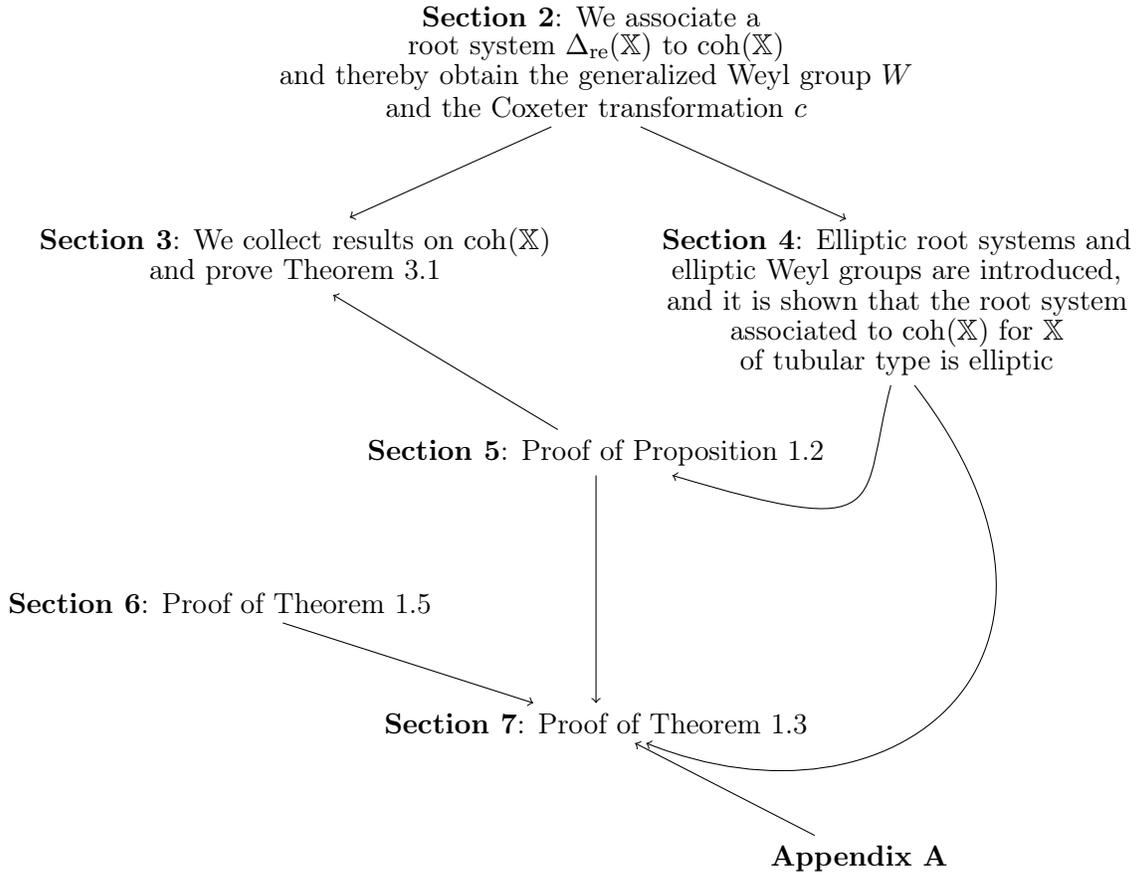
   $$s_\alpha(v):= v -  \frac{2(\alpha\mid v)}{(\alpha,\alpha)} \alpha,~\mbox{for all}~v \in V.$$
Note that in this paper we use the convention that for linear maps $f,g$ of $V$ it is $(fg)(v) = f(g(v))$ for all $v \in V$.

The following definition is a generalization of the notation of a root system in the theory of finite Coxeter systems (see \cite[Section 1.2]{Sai85}).

\begin{Definition}\label{def:root_system}
A non-empty subset $\Phi\subseteq V$ of non-isotropic vectors is called \defn{generalized root system} if the following properties are satisfied
\begin{enumerate}
\item[(a)] $\spanr(\Phi)= V$,
\item[(b)] $s_{\alpha}(\Phi) \subseteq \Phi$ for all $\alpha\in \Phi$ and
\item[(c)] $\Phi$ is \defn{crystallographic}, i.e. $\frac{2(\alpha \mid \beta)}{(\beta \mid \beta)} \in \mathbb{Z}$ for all $\alpha,\beta \in \Phi$.
\end{enumerate}

The elements in a generalized root system are \defn{roots}. The generalized root system  $\Phi$  is  \defn{reduced} if $r\alpha\in \Phi$ ($r\in \RR$) implies $r=\pm 1$ for all $\alpha\in \Phi$, and the system is \defn{irreducible} if there do not exist generalized root systems $\Phi_{1},\Phi_{2}$ such that $\Phi=\Phi_{1} \cup \Phi_{2}$ and $\Phi_{1} \bot \Phi_{2}$. A non-empty subset $\Psi$ of a generalized root system $\Phi$ is a \defn{root subsystem} if $\Psi$ is a generalized root system in $\spanr(\Psi)$. For a subset $R=\{\alpha_1, \ldots, \alpha_n \} \subseteq \Phi$ we define $\langle R \rangle_{\text{RS}}$ to be the smallest root subsystem of $\Phi$ that contains  $R$, and we call it the \defn{root subsystem generated by $R$}. For a subset $\Psi \subseteq \Phi$ we put $L(\Psi) = \spanz(\Psi)$. The crystallographic condition on  $\Phi$ implies  that $L(\Psi)$ is a lattice. The root system $\Phi$  is  \defn{simply-laced} if $(\alpha \mid \alpha) =2$ for all $\alpha \in \Phi$. Two generalized root systems $\Phi_{1}$ and $\Phi_{2}$ are  \defn{isomorphic} if there exists a linear isometry between the corresponding ambient spaces that sends $\Phi_{1}$ to $\Phi_{2}$.
\end{Definition}

Notice that in \cite[Section 2.1]{STW16} a simply laced generalized root system consists of more information, namely of a lattice, a symmetric bilinear form, a set of roots and an element called Coxeter transformation.
  
For $\Psi \subseteq \Phi$ let 
$$
W_{\Psi}: = \langle s_{\alpha} \mid \alpha \in \Psi \rangle
$$
be the reflection group on $V$ related to $\Psi$. The  following is an easy to check consequence of the crystallographic condition on the root system $\Phi$.

\begin{Lemma}\label{lem:crystallographic}
If $\Psi \subset \Phi$, then 
$w(L(\Psi)) \subseteq L(\Psi)$ holds for every $w \in W_\Psi$.
\end{Lemma}

\newpage
\subsection{Extended Coxeter-Dynkin diagrams}
We start by introducing a diagram $\Gamma$, the so-called \defn{extended Coxeter-Dynkin diagram}.
For a non-negative integer $r\in \NN_{0}$ and numbers $p_{i}\in \NN$ ($1 \leq i \leq r$) it is defined as follows:

\begin{figure}[h!]
  \centering
  \begin{tikzpicture}[scale=3.4]
    \node (A6666) at (0,-4) [circle, draw, fill=black!50, inner sep=0pt, minimum width=4pt] {};
    \node (AA6666) at (0,-3.9) [] {\tiny{$(1,p_1-1)$}};
    \node (A666) at (0.5,-4) [circle, draw, fill=black!50, inner sep=0pt, minimum width=4pt] {};
    \node (AA666) at (0.5,-3.9) [] {\tiny{$(1,p_1-2)$}};
    \node (A66) at (0.75,-4) [] {$\ldots$};
    \node (A6) at (1,-4) [circle, draw, fill=black!50, inner sep=0pt, minimum width=4pt] {};
    \node (AA6) at (1,-3.9) [] {\tiny{$(1,2)$}};
    \node (B6) at (1.5,-4) [circle, draw, fill=black!50, inner sep=0pt, minimum width=4pt]{};
    \node (BB6) at (1.45,-3.9) [] {\tiny{$(1,1)$}};
    \node (C6) at (2,-4) [circle, draw, fill=black!50, inner sep=0pt, minimum width=4pt]{};
    \node (CC6) at (2,-4.1) [] {\tiny{$1$}};
    \node (D6) at (2.5,-4) [circle, draw, fill=black!50, inner sep=0pt, minimum width=4pt]{};
    \node (DD6) at (2.55,-3.9) [] {\tiny{$(r,1)$}};
    \node (E6) at (3,-4) [circle, draw, fill=black!50, inner sep=0pt, minimum width=4pt]{};
    \node (EE6) at (3,-3.9) [] {\tiny{$(r,2)$}};
    \node (E66) at (3.25,-4) [] {$\ldots$};
    \node (F6) at (3.5,-4) [circle, draw, fill=black!50, inner sep=0pt, minimum width=4pt]{};
    \node (FF6) at (3.5,-3.9) [] {\tiny{$(r,p_r-2)$}};
    \node (G6) at (4,-4) [circle, draw, fill=black!50, inner sep=0pt, minimum width=4pt]{};
    \node (GG6) at (4,-3.9) [] {\tiny{$(r,p_r-1)$}};
    \node (I6) at (2,-3.5) [circle, draw, fill=black!50, inner sep=0pt, minimum width=4pt]{};
    \node (II6) at (2,-3.4) [] {\tiny{$1^*$}};
    \node (J6) at (2.35,-4.2) [circle, draw, fill=black!50, inner sep=0pt, minimum width=4pt]{};
    \node (JJ6) at (2.55,-4.15) []{\tiny{$(r-1,1)$}};
    \node (J66) at (2.7,-4.4) [circle, draw, fill=black!50, inner sep=0pt, minimum width=4pt]{};
    \node (JJ66) at (2.9,-4.35) []{\tiny{$(r-1,2)$}};
    \node (J66666) at (2.88,-4.5) []{$\ldots$};
    \node (J666) at (3.05,-4.6) [circle, draw, fill=black!50, inner sep=0pt, minimum width=4pt]{};
    \node (JJ666) at (3.4,-4.55) []{\tiny{$(r-1,p_{r-1}-2)$}};
    \node (J6666) at (3.4,-4.8) [circle, draw, fill=black!50, inner sep=0pt, minimum width=4pt]{};
    \node (JJ6666) at (3.75,-4.75) []{\tiny{$(r-1,p_{r-1}-1)$}};
    \node (K6) at (1.65,-4.2) [circle, draw, fill=black!50, inner sep=0pt, minimum width=4pt]{};
    \node (KK6) at (1.55,-4.15) []{\tiny{$(2,1)$}};
    \node (K66) at (1.3,-4.4) [circle, draw, fill=black!50, inner sep=0pt, minimum width=4pt]{};
    \node (KK66) at (1.2,-4.35) []{\tiny{$(2,2)$}};
    \node (K66666) at (1.12,-4.5) []{\ldots};
    \node (K666) at (0.95,-4.6) [circle, draw, fill=black!50, inner sep=0pt, minimum width=4pt]{};
    \node (KK666) at (0.75,-4.55) []{\tiny{$(2,p_2-2)$}};
    \node (K6666) at (0.6,-4.8) [circle, draw, fill=black!50, inner sep=0pt, minimum width=4pt]{};
    \node (KK6666) at (0.4,-4.75) []{\tiny{$(2,p_2-1)$}};
    \node (L6) at (2,-4.5) []{\ldots};


    \draw[-] (A666) to (A6666);
    \draw[-] (J666) to (J6666);
    \draw[-] (K666) to (K6666);
    \draw[-] (A6) to (B6);
    \draw[-] (B6) to (C6);
    \draw[-] (C6) to (D6);
    \draw[-] (D6) to (E6);
    \draw[-] (F6) to (G6);
    \draw[-] (I6) to (B6);
    \draw[-] (I6) to (D6);
    \draw[-] (I6) to (J6);
    \draw[-] (I6) to (K6);
    \draw[-] (J6) to (J66);
    \draw[-] (K6) to (K66);
    \draw[-] (C6) to (J6);
    \draw[-] (C6) to (K6);
    \draw[dashed] ([xshift=0.5]C6.north) to ([xshift=0.5]I6.south);
    \draw[dashed] ([xshift=-0.5]C6.north) to ([xshift=-0.5]I6.south);

  \end{tikzpicture} 
  \caption{Extended Coxeter-Dynkin diagram $\Gamma$ $\phantom{000000000000000}$} \label{def:GenCoxDiag}
\end{figure}

\subsection{Extended spaces, extended root systems and extended Weyl groups}\label{subsec:ExtendedSpace}
We attach analogously to the Tits representation for Coxeter groups (see \cite[Section 5.3]{Hum90}) a geometric datum to the extended Coxeter-Dynkin diagram. This will lead to the definition of the extended Weyl group and of a Coxeter transformation.

Let $Q$ be the vertex set of the diagram $\Gamma$ given in Figure \ref{def:GenCoxDiag}. Define $V$ to be the real vector space $V$ with basis $B:=\lbrace \alpha_{\nu}\mid \nu\in Q\rbrace$. We further define a symmetric bilinear form $(-\mid-)$ on $V$ by setting 
\begin{equation*}
  (\alpha_{\nu}\mid \alpha_{\omega}) =
  \begin{cases}
    2  & \nu,\omega \in Q \text{ are connected by a dotted double bound or }\nu=\omega, \\
    0  & \nu,\omega \in Q \text{ are disconnected,}\\          
    -1 & \nu,\omega \in Q \text{ are connected by a single edge},
  \end{cases}
\end{equation*}
and extending bilinearly to $V$.

We call $(V,B, (- \mid -))$ an \defn{extended space}.

A direct calculation yields the signature of the symmetric bilinear form $(-\mid -)$ in an extended space:

\begin{Lemma}\label{lem:signature}
Let $(V, B, (-\mid -))$ be an extended space.  The signature of $(-\mid-)$ is $(|B|-2,1,1)$, $(|B|-2,0,2)$ or $(|B|-1,0,1)$ where the first, the second and the third entries are the geometric dimensions of the positive, the negative and the  zero eigenvalues, respectively.
\end{Lemma}
\medskip

\begin{propdef} \label{def:basic}
Let $\mathcal{V}:=(V, B, (- \mid -))$ be an extended space.
\begin{enumerate}
\item[(a)] The group $W:=\langle s_{\alpha}\mid \alpha\in B\rangle$ is called \defn{extended Weyl group} (of $\mathcal{V}$). We say that  it is of \defn{wild},  \defn{tubular} or \defn{domestic} type if the signature of $(- \mid -)$ is $(|B|-2,1,1)$, $(|B|-2,0,2)$ or $(|B|-1,0,1)$, respectively. Further we denote by $R$ the radical of the form $(-\mid -)$. It is easy to see that the signatures that occur are exactly the following:
\begin{align*}
\text{tubular: }&(m,0,2) \text{ for } m=4,6,7,8\\
\text{domestic: }&(m,0,1) \text{ for } m\in \NN \\
\text{wild: }&(m,1,1) \text{ for } m\geq 6
\end{align*}
\item[(b)] Set $a:=\alpha_{1^{*}}-\alpha_{1}$. Then $\TR_a := s_{\alpha_{1}} s_{\alpha_{1^{*}}}$ is a ``translation'' that  maps $v \in V$ to $v - (\alpha_1, v) a$. If $W$ is of  domestic or wild type, then dim $R = 1$ and $R = \spanr(a)$. If $W$ is of tubular type, then $\spanr(a) \subset R$ and dim$_\RR(R) = 2$.
\item[(c)] The set $S:=\lbrace s_{\alpha}\mid \alpha \in B \rbrace$ is called \defn{simple system} and its elements  \defn{simple reflections}. We call $(W,S)$ an \defn{extended Weyl system} of $\mathcal{V}$. The set $T:=\bigcup_{w\in W}wSw^{-1}$ is  the \defn{set of reflections} for $(W,S)$.
\item[(d)] Let $\Phi\subseteq V$ be the minimal set that contains $B$ and is closed under the action of $W$, which is $w(\beta)\in \Phi$ for all $w\in W$ and $\beta \in \Phi$. Then $\Phi$ is a reduced and irreducible generalized root system and $\Phi=\{ w (\alpha) \mid w \in W,~\alpha \in B \}$. We call $\Phi$ \defn{extended root system} and the elements of $B$ are called \defn{simple roots}.
\item[(e)] The extended root system $\Phi$ is simply-laced.
\item[(f)] An element $c:=\left(\prod_{\alpha \in B\setminus \lbrace \alpha_{1},\alpha_{1^{*}}\rbrace}s_{\alpha}\right)\cdot s_{\alpha_{1}}s_{\alpha_{1^{*}}}$ is called \defn{Coxeter transformations} where we take the first $|B|-2$ factors in an arbitrary order.
\item[(g)] Set $\Gamma_0:= \Gamma\setminus{\{1^*\}}, B_0:= B\setminus{\{\alpha_{1^*}\}}, W_0:= W_{\Gamma_0}, S_0:= \{s_\alpha \mid \alpha \in B_0\}$ and $V_0:= \spanr(B_0)$. Then $(W_0,S_0)$ is a Coxeter system with Coxeter diagram $\Gamma_0$, generalized root system $\Phi_0$ and Tits representation $W_0 \rightarrow \GL(V_0)$.
\end{enumerate}
\end{propdef}

\begin{Lemma}\label{lem:reflectionsNonelliptic}
   Let $W$ be of domestic or wild type with related extended root system $\Phi$. Then 
   $\Phi = \Phi_0 \oplus \ZZ a$, where 
    $a =\alpha_{1^{*}}-\alpha_{1}$.
\end{Lemma}

\begin{proof}
It is $B \subset \Phi_0 \oplus \ZZ a$  and the latter set is invariant under the action of $W$. Thus $\Phi \subseteq \Phi_0 \oplus \ZZ a$. From the definition we get $\Phi_0  \subset \Phi$.  As the Coxeter diagram $\Gamma_0$ for $W_0$ is a tree, all the roots in $\Phi_0$ are conjugate. Therefore, it suffices to show $\alpha_1 + \ZZ a \subseteq \Phi$, which follows from the facts that there is $\alpha \in \Phi$ such that $(\alpha \mid \alpha_1) =1$ and the action of $\TR_a(\alpha_1)$ on $V$.
\end{proof}   

In the following we will also consider for $U$ a subspace of the radical $R$ the projection 
$p_U: V \rightarrow V/U$. Let $W_{\Phi/U}$ be the 
action of $W = W_\Phi$ on $V/U$. We denote the map
from $W$ onto $W_{\Phi/U}$ also by $p_U$.

\begin{Lemma}\label{lem:Projection}
If $W$ is of domestic or wild type, then $W_{\Phi/R} \cong W_0$.
\end{Lemma}
\begin{proof} It follows from the fact that $p_R(s_{\alpha_{1^*}}) = p_R(s_{\alpha_{1}})$.
\end{proof}  

Further this isomorphism transports $\Phi/R$ onto $\Phi_0$ and we will identify the two root systems.

\subsection{The Hurwitz action}\label{subsec:HurwitzN}
We define the \defn{Hurwitz action} for arbitrary groups to connect the braid group action on exceptional sequences with a suitable action on tuples of reflections. Let $G$ be a group and $(g_{1},\ldots,g_{n})$ an element of $G^{n}$. The Hurwitz action on $G^n$ is defined by
\begin{align*}
  \sigma_i (g_1 ,\ldots , g_n )      & = (g_1 ,\ldots , g_{i-1} , \hspace*{5pt} \phantom
  {g_i}g_{i+1}\phantom{g_i^{-1}}, \hspace*{5pt} g_{i+1}^{-1}g_ig_{i+1}, \hspace*{5pt} g_{i+2} ,
  \ldots , g_n),                                                                                      \\
  \sigma_i^{-1} (g_1 ,\ldots , g_n ) & = (g_1 ,\ldots , g_{i-1} , \hspace*{5pt} g_i g_{i+1} g_i^{-1},
  \hspace*{5pt} \phantom{g_{i+1}}g_i\phantom{g_{i+1}^{-1}}, \hspace*{5pt} g_{i+2} ,
  \ldots , g_n).
\end{align*}
for the standard generator $\sigma_{i}\in \mathcal{B}_n$, for $1\leq i \leq n-1$. In particular, for a subset $T$ of $G$ that is closed under conjugation, the Hurwitz action restricts to an action on $T^{n}$.

\begin{Lemma} \label{lem:HurwitzLattice}
Let $\Phi$ be an extended root system and $\alpha_1,\ldots,\alpha_k,\beta_1, \ldots,\beta_k \in \Phi$ such that $\sigma (s_{\alpha_1},\ldots,s_{\alpha_k})=(s_{\beta_1}, \ldots,s_{\beta_k})$ for some $\sigma \in \mathcal{B}_k$. Then $L(\{\alpha_1,\ldots,\alpha_k\})=L(\{\beta_1, \ldots,\beta_k\})$.
\end{Lemma}

\begin{proof}
It is enough to check the assertion for the case that $\sigma$ is a standard generator, say $\sigma = \sigma_i$. Then 
    $$
    \sigma_i (s_{\alpha_1},\ldots,s_{\alpha_k}) = (s_{\alpha_1},\ldots, s_{\alpha_{i-1}}, s_{\alpha_{i+1}}, s_{s_{\alpha_{i+1}}(\alpha_{i})}, s_{\alpha_{i+2}},\ldots, s_{\alpha_k}),
    $$ 
and $s_{\alpha_{i+1}}(\alpha_{i})= \alpha_i -\frac{2(\alpha_i\mid \alpha_{i+1})}{(\alpha_{i+1}\mid \alpha_{i+1})}\alpha_{i+1}$. In particular, we observe that
    $$
    L(\{\alpha_1,\ldots,\alpha_k\}) = 
    L(\{ \alpha_1,\ldots, \alpha_{i-1}, \alpha_{i+1}, s_{\alpha_{i+1}}, \alpha_{i+2},\ldots, \alpha_k\}).
    $$
\end{proof}

\subsection{The reflection length of Coxeter transformations}\label{sec:reflength}
The aim of this subsection is to show Proposition~\ref{lem:redCoxElt}. Thus we determine the reflection length of a Coxeter transformation $c$ in an extended Weyl group $W$.

Let $\mathcal{V}:=(V, B, (- \mid -))$ be an extended space and let $(W,S)$ an extended Weyl system (of $\mathcal{V}$) with set of reflections $T$ and corresponding extended root system $\Phi$. We put $n = |S|=$ dim$_\RR(V)$.

As the set of reflections $T$ generates $W$, each $x \in W$ is a product $x = t_1 \cdots t_r$ where $t_i \in T$. If this product is of minimal length, then we say that this $T$-factorization is \defn{reduced} and that $\ell_T(x) =r$. We show that $\ell_T(c) = |S|$, which implies that the factorization of  $c$  in pairwise different simple reflections is $T$-reduced. This fact will be an ingredient of the proof of Theorem~\ref{conj:WeightProjElliptic0}. In the further course we have to distinguish whether the form is tubular or not.

We first consider the cases where the signature of the form $(-\mid-)$ is $(m,0,1)$ for $m \in \mathbb{N}$ or $(m,1,1)$ for $m\geq 6$, that is $W$ is of domestic or wild type, respectively. We keep the notation as introduced in Sections \ref{subsec:notation} to \ref{subsec:HurwitzN}. 

\begin{proof}[\textbf{Proof of Proposition~\ref{lem:redCoxElt} (the domestic and wild cases)}]
Assume that the length of the Coxeter transformation is strictly smaller than $n+2$ and assume, for simplicity, that
$$
c:=\bigg(\prod_{\alpha \in B\setminus \lbrace \alpha_{1},\alpha_{1^{*}}\rbrace}s_{\alpha}\bigg)\cdot s_{\alpha_{1}}s_{\alpha_{1^{*}}} \in W,
$$
i.e. the reflections $s_{\alpha_{1}}$ and $s_{\alpha_{1}^{*}}$ are at the last positions of the factorization and the remaining reflections can be arbitrary ordered. Then by considering the map $\rho:W  \rightarrow W_{\Phi_0}$ induced by the projection
$$
p: \Phi_0 \oplus \ZZ a \rightarrow \Phi_0, \beta + ka \mapsto \beta
$$
and by parity arguments, the length of $c$ has to be $n$.

Let $c=s_{\beta_{1}+\ell_{1}a} \cdots s_{\beta_{n}+\ell_{n}a}$ be a reduced factorization with $\beta_{i}+\ell_{i}a \in \Phi_0\oplus \ZZ a$, $\beta_{i}\in \Phi_0$ and $\ell_{i}\in \ZZ$. After a possible renumbering we can assume that $c= s_{\alpha_2} \cdots s_{\alpha_{n+1}} s_{\alpha_1}s_{\alpha_{1^*}}$ and therefore
\begin{align*}
\rho(c) = \prod_{\alpha \in B\setminus \lbrace \alpha_{1},\alpha_{1^{*}}\rbrace}s_{\alpha} =s_{\alpha_2} \cdots s_{\alpha_{n+1}} = s_{\beta_{1}} \ldots s_{\beta_{n}}.
\end{align*}

\noindent Since $W_{\Phi_0}$ is a Coxeter group, the element $\rho(c)$ is a parabolic Coxeter element in $W_{\Phi_0}$ and by \cite[Theorem 1.3]{BDSW14} the Hurwitz action is transitive on the reduced factorizations of $\rho(c)$, that is, there exists $\sigma \in \mathcal{B}_{n}$ with
\begin{align} \label{Eq:BraiExcCoxWild}
\sigma (s_{\alpha_2}, \ldots, s_{\alpha_{n+1}}) = (s_{\beta_{1}}, \ldots, s_{\beta_{n}})
\end{align}
The latter yields
\begin{align*}
\sigma (s_{\alpha_2}, \ldots, s_{\alpha_{n+1}}, s_{\alpha_1}, s_{\alpha_{1^*}}) = (s_{\beta_{1}}, \ldots, s_{\beta_{n}}, s_{\alpha_1}, s_{\alpha_{1^*}}),
\end{align*}
thus
\begin{align*}
\idop = c^{-1}c
=\big(s_{\alpha_{1^*}}s_{\alpha_{1}} s_{\alpha_{n+1}} \cdots s_{\alpha_2} \big) \big(s_{\beta_{1}+\ell_{1}a} \cdots s_{\beta_{n}+\ell_{n}a} \big) = \big(s_{\alpha_{1^*}}s_{\alpha_{1}} s_{\beta_{n}} \cdots s_{\beta_1} \big) \big(s_{\beta_{1}+\ell_{1}a} \cdots s_{\beta_{n}+\ell_{n}a} \big) 
\end{align*}
and
\[s_{\alpha_{1}} s_{\alpha_{1}^{*}} = s_{\beta_{n}} \cdots s_{\beta_{1}}  s_{\beta_{1}+\ell_{1}a} \cdots s_{\beta_{n}+\ell_{n}a}.\]
It is easy to check (by induction) that
\[s_{\beta_{n}} \cdots s_{\beta_{1}}  s_{\beta_{1}+\ell_{1}a} \cdots s_{\beta_{n}+\ell_{n}a}(x)=x-\left(x~ \bigg \vert ~ \sum_{i=1}^{n}k_{i}\beta_{i}\right)a\]
for some $k_{i}\in \ZZ$ and $s_{\alpha_{1}} s_{\alpha_{1}^{*}}(x)=x-(x\mid \alpha_{1})a$. The previous yields
\[\left( x ~\bigg \vert~ \alpha_{1}-\sum_{i=1}^{n}k_{i}\beta_{i} \right)=0\]
for all $x\in \spanz(\Phi_0)$. Since the radical of $(-\mid-)$ restricted to $\spanz(\Phi_0)$ is trivial, we obtain by (\ref{Eq:BraiExcCoxWild}) and Lemma \ref{lem:HurwitzLattice} that
$$
\alpha_{1}\in \spanz (\beta_{1},\ldots,\beta_{n} ) = \spanz (\alpha \mid \alpha \in B_{0}\setminus \lbrace \alpha_1 \rbrace ),
$$
contradicting the fact that $\spanz(\Phi_0)$ is of rank $n+1$.
\end{proof}

\begin{remark}
    In the tubular case, Kluitmann has already shown in \cite[Korollar~2.3]{Klu87} that a Coxeter transformation cannot be written as a product of less than $n+2$ reflections, which completes the Proof of Proposition~\ref{lem:redCoxElt}. Nevertheless, for sake of completeness we will inculde a proof for the tubular case in Section \ref{subsec:CoxTrafo} after a more detailed study of tubular extended Weyl groups.
\end{remark}

\medskip
Of particular interest, also in view of our applications, is the set of elements lying below a Coxeter element with respect to the relation $\leq $ given in Definition~\ref{def:PrefixPoset}.

\subsection{The interval posets $[\idop, c]$ and  $[\idop, c]^{\operatorname{gen}}$} \label{sec:IntervalPoset}
Let $(W,S)$ be an extended Weyl system with set of reflections $T$ and let $\Phi$ be the associated extended root system. As $T$ generates $W$, each $x \in W$ is a product $x = t_1 \cdots t_r$ where $t_i \in T$. If this product is of minimal length, then we say that this $T$-factorization is \defn{reduced} and that $\ell_T(x) =r$. For $x  \in W$ with $\ell_T(x) =r$ we put 
$$
\Red_T(w) := \{ (t_1 , \ldots, t_r) \in T^r \mid w=t_1 \cdots t_r \}.
$$

\begin{Definition} \label{def:PrefixPoset}
$\phantom{4}$
\begin{itemize}
\item[(a)] Define a partial order on $W$ by $x \leq y ~ \mbox{if and only if} ~\ell_T(y)= \ell_T(x) + \ell_T(x^{-1}y)$ for $x,y \in W$, called \defn{absolute order}, where $\ell_T$ is the \defn{length function} on $W$ with respect to $T$.
\item[(b)]For $w \in W$ the interval $[\idop, w] = \{x \in W \mid \idop \leq x \leq w \}$ is called the \defn{interval poset} of $w$ with respect to the partial order $\leq$.
\item[(c)] Let $w$ and $(t_1,\ldots,t_{n+2}) \in \Red_T(c)$. We call $(t_1,\ldots , t_i)$ a \defn{prefix} of the factorization $(t_1,\ldots,t_{n+2})$ for $1 \leq i \leq n+2 $.
\end{itemize}
\end{Definition}

Note that $x \leq w$ if and only if there exists a reduced factorization $\left(t_1,\ldots , t_{\ell_T(x)} \right) \in \Red_T(x)$ which is a prefix of a reduced factorization of $w$. By abuse of notation we sometimes call $x$ to be a prefix of $w$. For the extended root systems we need studying generating factorizations.

\begin{Definition}\label{def:Generating}
$\phantom{4}$
\begin{itemize}
\item[(a)]Let $w \in W$, $t_{i}\in T$ for $1 \leq i \leq n+2 $ such that $w=t_{1}\cdots t_{n+2}$. We call the factorization $(t_{1},\ldots, t_{n+2})$ \defn{generating}, if $W=\langle t_{1},\ldots,t_{n+2} \rangle$.
\item[(b)] Assume that $w \in W$ admits a generating factorization.  Define the subposet $[\idop, w]^{\text{gen}}$ of $[\idop, w]$ to be the set of prefixes of reduced generating factorizations of $w$. More precisely, we have $x \in [\idop, w]^{\text{gen}}$ if and only if there exists $\left(t_1,\ldots , t_{\ell_T(x)} \right) \in \Red_T(x)$ such that $\left(t_1,\ldots , t_{\ell_T(x)} \right)$ is a prefix of a reduced generating factorization for $w$.
\end{itemize}
\end{Definition}
    
\begin{Lemma}\label{lem:gen_root_lattice}
Let $(s_{\beta_{1}},\ldots,s_{\beta_{n+2}})$ be a generating factorization with $\beta_{i}\in \Phi$ for $1\leq i\leq n+2$. Then $L(\Phi)= L(\{\beta_{1},\ldots,\beta_{n+2}\})$.
\end{Lemma}

\begin{proof}
Since the extended Coxeter-Dynkin diagram in Figure \ref{def:GenCoxDiag} has a spanning tree consisting of simple edges, it is easy to see that all simple reflections are conjugated in $W$. Therefore the set of reflections forms a single conjugacy class. In particular, for every $\beta\in \Phi$ there exists $w\in W=\langle s_{\beta_{1}},\ldots, s_{\beta_{n+2}} \rangle$ such that $s_{w(\beta_{1})}=ws_{\beta_{1}}w^{-1}=s_{\beta}$. As $\Phi$ is crystallographic and reduced by definition and by Proposition \ref{def:basic} and as $w$ has a factorization  in the reflections $s_{\beta_{1}},\ldots,s_{\beta_{n+2}}$, it follows that $\beta=\pm w(\beta_{1}) \in L(\beta_{1},\ldots,\beta_{n+2})$. Therefore Lemma \ref{lem:crystallographic} yields $L(\Phi)\subseteq \spanz(\beta_{1},\ldots,\beta_{n+2})$.
\end{proof}

\begin{remark}
In Example~\ref{Example:Non-generating} we will present an example of an extended root system $\Phi$ such that $W_{\Phi}$ contains a Coxeter transformation that admits a non-generating factorization.
\end{remark}

\section{The extended root systems attached to the categories $\COH(\XX)$ and $D^b(\COH(\mathbb{X}))$} \label{sec:RootSystemDerivedCategory}
Throughout the following sections $k$ will be an algebraically closed field of characteristic zero (see also \cite{STW16}), and $\COH(\mathbb{X})$ the category of coherent sheaves over a weighted projective line $\XX$ in the sense of \cite{GL87}. We will not introduce this category in detail, but recall some basic facts  on $\COH(\mathbb{X})$ as well as on its bounded derived category $\mathcal{D}:=D^b(\COH(\mathbb{X}))$ in the next subsection. Throughout the whole section we will quote and prove some more properties that we will  use in the proof of Theorem~\ref{conj:WeightProjElliptic0}. We refer the reader to \cite{HTT07, CK09, GL87} for the definitions and for further details.

\subsection{The categories $\COH(\mathbb{X})$ and $D^b(\COH(\mathbb{X}))$}\label{subsec:IntroCategories}
By \cite[Chapter 6, Section 10.2]{HTT07} or \cite[Theorem 6.8.1]{CK09} the category $\COH(\XX)$ satisfies the following properties, which we will use. Note that in this paper we will abbreviate $\HOM_{\COH(\XX)}$ and $\EXT^{n}_{\COH(\XX)}$ by $\HOM_{\XX}$ and $\EXT^{n}_{\XX}$, respectively.
\begin{itemize}
\item $\COH(\XX)$ is a connected abelian $k$-category, and each object in $\COH(\XX)$ is noetherian.
\item  $\COH(\XX)$ is (skeletally) small and ext-finite, that is, all morphism and extension spaces are finite dimensional $k$-vector spaces.
\item  There exists a Serre functor $\tau$, that is an equivalence $\tau:\COH(\XX)\rightarrow \COH(\XX)$ and natural isomorphisms $\EXT_{\XX}^{1}(X,Y)^{*}=\HOM_{\XX}(Y,\tau X)$ for all objects $X,Y$ of $\COH(\XX)$, where $(-)^{*}$ denotes the $k$-duality $\HOM_k(-,k)$.
\item  $\COH(\XX)$ has a tilting object.
\end{itemize}

In fact, the category $\COH(\XX)$ is characterized by these properties together with the remaining properties of  \cite[Chapter 6, Section 10.2]{HTT07}.

For every abelian $k$-category $\mathcal A$ there is the concept of the bounded derived category $D^b(\mathcal A)$. The derived category $D(\mathcal A)$ is obtained from the classes of cochain complexes in $\mathcal A$ up to homotopy by formally inverting all quasi-isomorphisms. The full subcategory of  $D(\mathcal A)$ consisting of the objects that are isomorphic to a complex $A$ such that $A^n = 0$ for almost all $n \in \ZZ$ is denoted by $D^b({\mathcal A})$.

There is a functor on $D^b(\mathcal A)$, the \defn{shift} operator: let $(A,d_A)$ be a cochain complex, that is a representative of an object of $D^b({\mathcal A})$. For $i \in \ZZ$ we denote by $A[i]$ the shifted complex with
$$A[i]^n = A^{i+n}~\mbox{and}~d_{A[i]}^n = (-1)^n d_A^{n+i}.$$
In particular, $[i]: [(A,d_A)] \mapsto [(A[i], d_{A[i]})]$ is an isomorphism on $D^b({\mathcal A})$. Note also that  $D^b(\mathcal A)$ is an ext-finite triangulated $k$-category. The concept of a bounded derived category of an abelian category $\mathcal A$ is a helpful tool also because of the following fact. 

\begin{Proposition}[{\cite[Prop. 2.1.1]{CK09}}]\label{prop:CK211}
$\EXT_{\mathcal A}^n(A,B) \cong \HOM_{D^b({\mathcal A})}(A,B[n])$
for all $A,B \in {\mathcal A}$ and $n \in \ZZ$.
\end{Proposition}

Recall that a morphism $\phi: A \rightarrow B$ between cochain complexes $(A, d_X)$ and $(B,d_Y)$ consists of morphisms $\phi^n: A^n \rightarrow B^n$ with $d^n_Y \phi^n = \phi^{n+1}d^n_X$ for all $n \in \ZZ$.

Following \cite{STW16} we will attach to the categories  $\mathcal{D}$ and $\COH(\mathbb{X})$ an extended root system (see Section~\ref{subsec:SimplyLacedGenRoot} and the definition and discussion of an extended root system in \cite{BWY21}). The Grothendieck groups $K_0(\mathcal{D})$ and $K_0(\COH(\mathbb{X}))$ will serve as replacements for the respective root lattices. Attached to them  is  a bilinear form, the \defn{Euler form}, that is  defined as follows. For objects $X,Y$ of $\COH(\mathbb{X})$ we set
\[\chi([X],[Y]):=\sum_{n\geq 0} (-1)^{n}\text{dim}_{k}\EXT^{n}_{\XX}(X,Y)\]
and for objects $X,Y$ of $\mathcal{D}$
\[\chi([X],[Y]):=\sum_{n\in \ZZ} (-1)^{n}\text{dim}_{k}\HOM_{\mathcal{D}}(X,Y[n]).\]
The next result clarifies the connection between these two groups and forms.

\begin{Lemma}[{\cite[Chapter 3.5]{CK09}}] \label{le:IsoGrothendieck}
Let $\mathcal{A}$ be an ext-finite abelian $k$-category. The inclusion $\mathcal{A} \rightarrow \DB(\mathcal{A})$ induces an isomorphism $K_0(\mathcal{A}) \cong K_0(\DB(\mathcal{A}))$. In addition, if the global dimension of $\mathcal{A}$ is finite, then the induced isomorphism is an isometry with respect to the respective Euler forms.
\end{Lemma}

\subsection{Exceptional sequences in $\COH(\mathbb{X})$}\label{subsec:NotMisProjLine}
An object $E$ in an abelian $k$-category $\mathcal{A}$ is called \defn{exceptional} if 
$$\END_{\mathcal{A}}(E)= k~\mbox{and}~\EXT_{\mathcal{A}}^{i}(E,E)=0~\mbox{for}~i>0,$$
and an object $E$ in a triangulated $k$-category $\mathcal{C}$ is \defn{exceptional} if $$\END_{\mathcal{C}}(E)=k~\mbox{and}~\HOM_{\mathcal{C}}(E,E[i])=0~\mbox{for}~i\neq 0.$$
A pair $(E,F)$ of exceptional objects in $\mathcal{A}$ (resp. in $\mathcal{C}$) is called \defn{exceptional pair} provided it holds in addition that $\EXT_{\mathcal{A}}^{i}(F,E)=0$ for $i\geq 0$ (resp. that $\HOM_{\mathcal{C}}(F,E[i])=0$ for all $i \in \ZZ$). A sequence $\mathcal{E}=(E_{1},\ldots,E_{r})$ of exceptional objects in $\mathcal{A}$ (resp. in $\mathcal{C}$) is called \defn{exceptional sequence of length $r$} if $\EXT_{\mathcal{A}}^{s}(E_{j},E_{i})=0$ (resp. $\HOM_{\mathcal{C}}(E_{j},E_{i}[s])=0$) for $i<j$ and all $s\in \ZZ$. Provided the rank of the Grothendieck group $K_0(\mathcal{A})$ (resp. $K_0(\mathcal{C})$) is finite an  exceptional sequence in $\mathcal{A}$ (resp. in $\mathcal{C}$) is called \defn{complete} if its length equals this rank.

\begin{remark}
  We are only interested in isomorphism classes of objects, thus the exceptional sequences will be considered as a sequence of isomorphism classes.
\end{remark}

We recall the concept of mutations of exceptional sequences  in $\DD$ as well as in $\COH(\mathbb{X})$ and the braid group action on these. For more details consider for example \cite[Chapter 2]{Bo89} and \cite[Chapter 3.2]{Mel04}.

We start by considering $\DD$. Given an exceptional pair $(E,F)$ in $\DD$, put 
$$
\HOM^{\bullet}(E,F):=\bigoplus_{i\in \ZZ}\HOM_{\DD}(E,F[i])[-i],
$$
where we consider this as a complex of vector spaces with trivial differential. The corresponding cochain complex has on position $i\in \ZZ$ the vector space $\HOM_{\DD}(E,F[-i])$. The objects $\mathcal{L}_{E}F$ and $\mathcal{R}_{F}E$ in $\DD$ defined by  the distinguished triangles
\begin{align*}
  \mathcal{L}_{E}F\rightarrow \HOM^{\bullet}(E,F)\otimes E \rightarrow F\rightarrow  \mathcal{L}_{E}F[1] \\
  E\rightarrow \HOM^{\bullet}(E,F)^{*}\otimes F \rightarrow  \mathcal{R}_{F}E\rightarrow E[1]
\end{align*}
are exceptional objects where $(-)^{*}$ denotes the complex with reverse grading. Set
$$
\HOM^{\bullet}(E,F)\otimes E :=\bigoplus_{i\in \ZZ} E[-i]^{\text{dim}_{k}\HOM_{\DD}(E,F[i])}.
$$
A left (resp. right) \defn{mutation} of an exceptional pair $(E,F)$ is the pair $( \mathcal{L}_{E}F,E)$ (resp. $(F,\mathcal{R}_{F}E)$). A mutation of an exceptional sequence $(E_{1},\ldots,E_{r})$ is defined as a mutation of a pair of adjacent objects, where all the other objects are fixed. Let $\mathcal{B}_r$ be the \defn{braid group} on $r$ strands, that is the group with (standard) generators $\sigma_{1},\ldots, \sigma_{r-1}$ and subject to the relations $\sigma_{i}\sigma_{j}=\sigma_{j}\sigma_{i}$ for $|i-j|\geq 2$ and $\sigma_{i}\sigma_{i+1}\sigma_{i}=\sigma_{i+1}\sigma_{i}\sigma_{i+1}$ for $i=1,\ldots,r-2$. The group $\mathcal{B}_r$ acts on an exceptional sequence $(E_{1},\ldots,E_{r})$ by
\begin{align*}
  \sigma_{i}(E_{1},\ldots,E_{r})      & =(E_{1},\ldots, E_{i-1},\hspace*{5pt} \phantom{\mathcal{L}_{E_{i}}}E_{i+1},\hspace*{5pt} \mathcal{R}_{E_{i+1}}E_{i},\hspace*{5pt} E_{i+2}\ldots,E_{r})   \\
  \sigma_{i}^{-1}(E_{1},\ldots,E_{r}) & =(E_{1},\ldots, E_{i-1},\hspace*{5pt} \mathcal{L}_{E_{i}}E_{i+1},\hspace*{5pt} \phantom{\mathcal{R}_{E_{i+1}}}E_{i},\hspace*{5pt} E_{i+2},\ldots,E_{r}).
\end{align*}
We consider the semidirect product $\ZZ^{r}\ltimes \mathcal{B}_r$ given by the group homomorphism $\mathcal{B}_r \rightarrow \mathfrak{S}_{r} \rightarrow \Aut_{\ZZ}(\ZZ^{r})$ which is induced by the map $\sigma_{i}\mapsto (i,i+1)$ and the natural action of the symmetric group $\mathfrak{S}_{r}$ on $\ZZ^{r}$.

The group $\ZZ^{r}\ltimes \mathcal{B}_r$ acts on the set of exceptional sequences  in $\DD$ by defining for a basis element $e_{i}\in \ZZ^{r}$
\[e_{i}(E_{1},\ldots,E_{r})=(E_{1},\ldots,E_{i-1},E_{i}[1],E_{i+1},\ldots,E_{r}).\]

Now we consider $\COH(\mathbb{X})$. Given an exceptional pair $(E,F)$ in $\COH(\mathbb{X})$, denote by $L_{E}F$ (resp. $R_{F}E$) the sheaf which coincides with $\mathcal{L}_{E}F$ (resp. $\mathcal{R}_{F}E$) up to shifts in the derived category, since $\COH(\XX)$ is hereditary (see \cite[Lemma 2.2.1]{CK09}). In this situation the sheaf $L_EF$ is uniquely determined by one of the following exact sequences
\begin{align*}
  0 & \rightarrow L_{E}F \rightarrow \HOM_{\mathbb{X}}(E, F) \otimes E \rightarrow F \rightarrow 0    \\
  0 & \rightarrow \HOM_{\mathbb{X}}(E, F) \otimes E \rightarrow F \rightarrow L_{E}F \rightarrow 0    \\
  0 & \rightarrow F \rightarrow L_{E}F \rightarrow \EXT_{\mathbb{X}}^1(E, F) \otimes E \rightarrow 0,
\end{align*}
where $V \otimes E$ is defined to be the sum of dim$(V)$ copies of $E$ for a finite dimensional vector space $V$. The sheaf $R_FE$ is defined similar. In particular we obtain an action of the braid group $\mathcal{B}_r$ on the set of exceptional sequences of length $r$ in $\COH(\mathbb{X})$ by
\begin{align*}
  \sigma_i (E_1 ,\ldots , E_r )      & = (E_1 ,\ldots , E_{i-1} , \hspace*{5pt} \phantom{L_{E_i}} E_{i+1},
  \hspace*{5pt} R_{E_{i+1}}E_i, \hspace*{5pt} E_{i+2} , \ldots , E_r),                                                                                        \\
  \sigma_i^{-1} (E_1 ,\ldots , E_r ) & = (E_1 ,\ldots , E_{i-1} , \hspace*{5pt} L_{E_i}E_{i+1}, \hspace*{5pt} \phantom{R_{E_{i+1}}}E_i, \hspace*{5pt} E_{i+2} ,
  \ldots , E_r).
\end{align*}
In the following we abbreviate the Grothendieck group $K_0(\COH(\mathbb{X}))$ by $K_0(\mathbb{X})$. Recall the following well known properties.

\begin{theorem}[{\cite[Theorem 4.3.1]{Mel04} and \cite[Theorem 1.1]{KM02}}]\label{thm:BraidGroupActionExc} 
  Let $\mathbb{X}$ be a weighted projective line and let $n$ be the rank of $K_0(\mathbb{X})$. Then the braid group $\mathcal{B}_n$ acts transitively on the isomorphism classes of complete exceptional sequences in $\COH(\mathbb{X})$.
\end{theorem}

\begin{corollary}[{\cite[Corollary 4.3.2]{Mel04}}]\label{cor:BraidGroupActionExc} 
  Let $\mathbb{X}$ be a weighted projective line and let $n$ be the rank of $K_0(\mathbb{X})$. Then the group $\mathcal{B}_n \ltimes \ZZ^{n}$ acts transitively on the isomorphism classes of complete exceptional sequences in $\DD$.
\end{corollary}

The next two results state that all exceptional sequences can be enlarged to complete exceptional sequences.

\begin{Lemma}[{\cite[Lemma 3.1.3]{Mel04} and \cite[Lemma 1]{CB92}}] \label{le:Enlargement_Coh}
  Every exceptional sequence in $\COH(\mathbb{X})$ or $\MOD(A)$ for a hereditary algebra $A$ can be enlarged to a complete exceptional sequence.
\end{Lemma}

\begin{corollary}\label{le:Enlargement_D}
  Every exceptional sequence in $\DD$ can be enlarged to a complete exceptional sequence.
\end{corollary}
\begin{proof}
  The exceptional sequences of $\COH(\XX)$ coincide  with those of $\DD$ up to shifts. Therefore the assertion follows directly from Lemma \ref{le:Enlargement_Coh}.
\end{proof}
The next lemma is a combination of the results \cite[Proposition 4.1]{GL87} and \cite[Lemma 8.1.2]{Mel04}.

\begin{Lemma}\label{le:Enlargement_DD}
  The category $\COH(\XX)$ of coherent sheaves over a weighted projective line $\XX$ contains a complete exceptional sequence.
\end{Lemma}

\subsection{The extended root system and the extended Weyl group attached to $\XX$ }\label{subsec:SimplyLacedGenRoot}

We consider the following quadruple
$$
  Q(\DD):=(K_0(\mathcal{D}), \chi^{s}, \Delta_{\text{re}}(\DD), c_{\DD}),
$$
which we call the \defn{root quadruple related to $\DD$}, where
\begin{itemize}
  \item $K_0(\DD)$ is the Grothendieck group of $\DD$;
  \item $\chi^{s}([X],[Y]):= \chi([X],[Y]) + \chi([Y],[X])$, where $\chi$ is the Euler form on $K_0(\DD)$ and where $[X]$ (resp. $[Y]$) denotes the class in $K_0(\DD)$ of the object $X \in \DD$ (resp. $Y\in \DD$);
  \item $\Delta_{\text{re}}(\DD):= W(B)B \subseteq K_0(\DD)$, where $(E_1, \ldots , E_n)$ is a complete exceptional sequence,
 $$B:=\{ [E_1], \ldots , [E_n] \}$$ 
 a basis of $K_0(\mathcal{D})$ and where $W(B)$ is the subgroup of \linebreak $\text{Aut}(K_0(\DD), \chi^{s})$ that is generated by the reflections $s_{[E_i]}$ (see Section \ref{subsec:notation}) on  $K_0(\DD)$ (for $1 \leq i \leq n$);
  \item the \defn{Coxeter transformation} $c_{\DD}$ is the automorphism on $K_0(\DD)$ induced by the Coxeter functor $\mathcal{C}_{\DD}:=\mathcal{S}_{\DD}[-1]$ on $\DD$ where $\mathcal{S}_{\DD}$ is the Serre functor on $\DD$.
\end{itemize}

Furthermore, we define the group
$$
  W(\DD):= W_{\Delta_{\text{re}}(\DD)} = \langle s_{\alpha} \in \text{Aut}(K_0(\DD), \chi^{s}) \mid \alpha \in \Delta_{\text{re}}(\DD) \rangle
$$
(see Section \ref{subsec:notation}). By \cite[Lemma 4.1.2]{Mel04} all complete exceptional sequences in $\mathcal{D}=D^b(\COH(\mathbb{X}))$ are full in the sense of \cite{STW16}. Therefore, because of Corollaries~\ref{cor:BraidGroupActionExc} and \ref{le:Enlargement_D}, the conditions of \cite[Proposition~2.10]{STW16} on $\mathcal{D}=D^b(\COH(\mathbb{X}))$ are satisfied, which yields that $Q(\DD)$ fulfills the conditions of \cite[Definition~2.1]{STW16}.

Both, the definition of the root quadruple as well as the definition of the group $W(\DD)$ do not depend on the choice of the complete exceptional sequence $(E_1, \ldots , E_n)$ (see \cite[Proposition 2.10]{STW16}). Notice that the Coxeter transformation can also be obtained from a complete exceptional sequence:

\begin{Lemma}[{\cite[Section 2]{STW16}}] \label{le:InfoRootSystem} 
  Let $\DD$, $B$ and $(E_1, \ldots , E_n)$ be as above. Then it holds
  \begin{itemize}
    \item[(a)] $c_{\DD} = s_{[E_1]} \cdots s_{[E_n]}$;
    \item[(b)] $\chi^{s}(\alpha, \alpha)=2$ for all $\alpha \in \Delta_{\text{re}}(\DD)$; and
    \item[(c)] $W(\DD)= W(B)$.
  \end{itemize}
\end{Lemma}

Analogously to the category $\DD$, we consider the quadruple
\[Q(\XX):=(K_{0}(\XX),\chi^{s},\Delta_{\text{re}}(\XX),c),\]
which we call the \defn{root quadruple related to $\XX$}, where
  \begin{itemize}
    \item $K_0(\XX)$ is the Grothendieck group of $\COH(\XX)$;
    \item $\chi^{s}([X],[Y]):= \chi([X],[Y]) + \chi([Y],[X])$ where $\chi$ is the Euler form on $K_0(\XX)$ and $[X]$ (resp. $[Y]$) denotes the class in $K_0(\XX)$ of an object $X \in \COH(\XX)$ (resp. $Y\in \COH(\XX)$)
        (for the definition of the Euler form see Section~\ref{subsec:IntroCategories});
    \item $\Delta_{\text{re}}(\XX):= W(B)B \subseteq K_0(\XX)$ where
          $$B:=\{ [E_1], \ldots , [E_n] \}$$
          for a complete exceptional sequence $(E_1, \ldots , E_n)$ and 
          $W(B)$ is the subgroup of \linebreak $\Aut(K_0(\XX), \chi^{s})$ generated by the reflections
          $s_{[E_i]}$ on $K_0(\XX)$;
    \item  $W(B)=W(\XX):= W_{\Delta_{\text{re}}(\XX)}$;
    \item $c$ is the Coxeter transformation of $W(\XX)$, that is, the automorphism of $K_0(\XX)$ defined as $c=s_{[E_1]}\ldots s_{[E_n]}$.
  \end{itemize}
The root quadruple $Q(\XX)$ also fulfills the conditions of \cite[Definition~2.1]{STW16}. For our purpose, there is no need to distinguish between $Q(\DD)$ and $Q(\XX)$, as the following result shows.
  
\begin{corollary}\label{cor:genrootsys}
There is an isometry  between the $\Z$-modules $K_{0}(\XX)$ and $K_0(\DD)$ that sends $\Delta_{\text{re}}(\XX)$ to $\Delta_{\text{re}}(\DD)$. The latter also induces an isomorphism from $W(\XX)$ to $W(\DD)$ which maps $c$ to $c_{\DD}$.
\end{corollary}

\begin{proof}
The exceptional sequences of $\COH(\XX)$ and $\DD$ coincide up to shifts. Every shift by $[1]$ or $[-1]$ yields a change of sign in the Grothendieck group, thus the corresponding elements in the group coincide up to sign, where $K_{0}(\XX)$ is identified with $K_{0}(\DD)$ by Lemma \ref{le:IsoGrothendieck}. Thus the corresponding Weyl groups are equal, which yields that the sets of roots $\Delta_{\text{re}}(\DD)$ and $\Delta_{\text{re}}(\XX)$ coincide. In particular, the Coxeter transformations of $W(\XX)$ and $W(D)$ are equal.
\end{proof}

Notice that our notation is differemt from that one in \cite{STW16}. In Section \ref{sec:root_star_quiver} we show that $\Delta_{\text{re}}(\XX)$ and $W(\XX)$ are indeed an extended root system and an extended Weyl group in the sense of Definition \ref{def:basic}. Therefore we call the set $\Delta_{\text{re}}(\XX)$ the \defn{extended root system related to $\XX$} and $W(\XX)$ the \defn{extended Weyl group related to $\XX$}.

\begin{remark} \label{rem:redcry}
Let $S: = \{ s_{[E_1]}, \ldots , s_{[E_n]} \}$. By definition of the Euler form we have\linebreak $\chi^s([E],[E])=2$ for each exceptional $E \in \COH(\XX)$. Since the group $W(\XX)$ leaves the form $\chi^s(-,-)$ invariant, the extended root system $\Delta_{\text{re}}(\XX)$ is simply-laced. In particular, it is reduced and crystallographic.
\end{remark}

\subsection{The reflection related to an exceptional object} \label{subsec:ExSequence}
In this section we show that an exceptional object in $\COH(\XX)$ uniquely determines a reflection, and we describe the action of such a  reflection on the extended root system related to $\XX$ as well as to $\DD$.

We start by recalling that every exceptional object in $\COH(\XX)$ is uniquely determined by its class in the Grothendieck group.

\begin{Lemma}[{\cite[Proposition 4.4.1]{Mel04}}] \label{le:ClassUniquely}
Let $E,F \in \COH(\mathbb{X})$ be exceptional and $[E]=[F]$ in $K_0(\mathbb{X})$. Then $E \cong F$.
\end{Lemma}

\begin{Lemma}[{\cite[Lemma 2.13]{STW16}}] \label{le:ExcObjectIsRoot}
  For every exceptional object $E \in \COH(\XX)$ (resp. $E \in \DD$), the class $[E] \in K_0(\XX)$ (resp. $[E] \in K_0(\DD)$) belongs to  $\Delta_{\text{re}}(\XX)$ (resp. $\Delta_{\text{re}}(\DD)$).
\end{Lemma}
Notice that the roots of the form $[E]$ for an exceptional object $E$ are called \defn{(real) Schur roots}.

\begin{Lemma} \label{le:Uniqueness}
The map sending an isomorphism class of an exceptional object $E$ from $\COH(\XX)$ to the reflection $s_{[E]}$ from $W(\XX)$ is injective.
\end{Lemma}

\begin{proof}
By Lemma \ref{le:ExcObjectIsRoot} the map is well-defined. Moreover, since $\Delta_{\text{re}}(\XX)$ is reduced,  if  $E$ and $F$ are two exceptional objects  with $s_{[E]}=s_{[F]}$, then $[E]=\pm [F]$. In this case \cite[Proposition 6.3.7 (2)]{CK09} yields that $[E]=[F]$ and Lemma \ref{le:ClassUniquely} implies $E\cong F$.
\end{proof}

\begin{Lemma} \label{le:BraidCompatibility}
Let $(E,F)$ be an exceptional pair in $\DD$, then 
$$s_{[E]}([F]) =- [\mathcal{L}_{E}F]~\mbox{and}~s_{[F]}([E])=-[\mathcal{R}_{F}E],$$
and if $(E,F)$ is an exceptional pair in $\COH(\mathbb{X})$, then 
$$ s_{[E]}([F]) = \pm  [ L_EF]~\mbox{ and }~ s_{[F]}([E]) = \pm [R_FE].$$
\end{Lemma}

\begin{proof}
Consider the exceptional pair $(E,F)$ in $\DD$ and the following distinguished triangles
\begin{align*}
    \mathcal{L}_{E}F\rightarrow \HOM^{\bullet}(E,F)\otimes E \rightarrow F\rightarrow  \mathcal{L}_{E}F[1] \\
    E\rightarrow \HOM^{\bullet}(E,F)^{*}\otimes F \rightarrow  \mathcal{R}_{F}E\rightarrow E[1].
\end{align*}
In the Grothendieck group $K_{0}$ it holds, as $\HOM(F,E[i])=0$ for all $i\in \ZZ$, that
$$[\HOM^{\bullet}(E,F)\otimes E]=\sum_{i\in \ZZ} (-1)^{i} \dim_{k}\HOM_{\DD}(E,F[i])[E]=\chi^{s}([E],[F])[E].$$
Thus the first triangle yields
\[[\mathcal{L}_{E}F]=-\big( [F]- \chi^{s}([E],[F]) [E] \big)=-s_{[E]}([F])\]
and the second yields
\[[\mathcal{R}_{F}E]=-\big( [E]- \chi^{s}([E],[F]) [F] \big)=-s_{[F]}([E]).\]
As the exceptional sequences of $\COH(\XX)$ coincide up to shifts with those of $\DD$ the second part of the lemma follows.
\end{proof}

\subsection{An extended Coxeter-Dynkin diagram for $\XX$} \label{sec:root_star_quiver}
The aim of this section is to show that for a weighted projective line $\XX$ we always find a complete exceptional sequence $\mathcal{E}$ in $\COH(\XX)$ such that the associated diagram (see Definition \ref{def:GenCoxDiagram}) is an extended Coxeter-Dynkin diagram. Our strategy is as follows: we consider the module category of a bound quiver algebra whose quiver is the one point extended star quiver. We construct a complete exceptional sequence in this category that induces an extended star diagram. As the latter module category is derived equivalent to $\COH(\XX)$  for a weighted projective line $\XX$ by \cite{STW16}, we thereby obtain the desired diagram. It will provide a nice description of the corresponding extended root system $\Delta_{\text{re}}(\XX)$.

\begin{Definition} \label{def:GenCoxDiagram}
We associate to every complete exceptional sequence $\mathcal{E}=(E_1, \ldots , E_n)$  in $\COH(\XX)$ (resp. in $\DD$) a diagram whose set of vertices is in bijection with $\{E_1, \ldots ,E_n\}$. Let $E_i$ and $E_j$ be two different exceptional objects, that is $i \neq j$. If $\chi^s([E_i] , [E_j])= -k$ with $k\in \NN$ we draw $k$ edges between the corresponding vertices, and these edges are dotted if $\chi^s([E_i] , [E_j])= k$. If $\chi^s([E_i] , [E_j])= 0$ we do not draw an edge between the corresponding vertices. As these diagrams are equal for $\COH(\XX)$ and $\DD$ according to Corollary~\ref{cor:genrootsys} we simply call it the \defn{generalized Coxeter-Dynkin diagram} associated to $(\XX, \mathcal{E})$. 
\end{Definition}

Recall the following concrete definition of a weighted projective line (see \cite{CK09}). Let  $\PP_k^1$ be the projective line over $k$, let $\lambda = (\lambda_1, \ldots, \lambda_r)$ be a (possibly empty) collection of distinct closed points of $\PP_k^1$, and let $p = (p_1, \ldots , p_r)$  be a weight sequence, that is, a sequence of positive integers. The triple $\XX = (\mathbb{P}_k^1, \lambda, p)$ is called a \defn{weighted projective line}. Without loss we may assume that $\lambda$ is normalized, that is $\lambda_1 = \infty, \lambda_2 = 0$ and $\lambda_3 = 1$. 

\begin{Definition} \label{def:bound_quiver} 
Let $r\geq 3$ be a positive integer and let $\XX = (\PP_k^1, \lambda, p)$ be a weighted projective line (see Section~\ref{sec:RootSystemDerivedCategory}). Set
$$(\lambda_{1}^{(1)},\lambda_{1}^{(2)})=(1,0)~\mbox{and}~ (\lambda_{i}^{(1)},\lambda_{i}^{(2)})=(\lambda_{i},1),~\mbox{ for}~ i=2,\ldots,r.$$

Further let $k\widetilde{\TT}_{p}$ be the quiver algebra corresponding to the one point extended star quiver given in Figure \ref{def:Quiver_one}. Then define the \defn{bound quiver algebra} to be
$$k\widetilde{\TT}_{p,\lambda}:=k\widetilde{\TT}_{p}/\mathcal{I}$$ 
where $\mathcal{I}$ is the ideal
$$\mathcal{I}:=\big( \sum_{i=1}^{r} \lambda^{(1)}_{i} f_{i,1^{*}}f_{1,i}, \sum_{i=1}^{r} \lambda^{(2)}_{i} f_{i,1^{*}}f_{1,i}  \big)$$ 
and where $f_{1,i}$ is the arrow from $1$ to $(i,1)$ and $f_{i,1^{*}}$ the arrow from $(i,1)$ to $1^{*}$.
\end{Definition}

\begin{remark}
The one point extended star quiver is a finite acyclic quiver, which yields the admissibility of the ideal $\mathcal{I}$, that is, there exists a positive integer $m$ such that $A^{m}\subseteq \mathcal{I} \subseteq A^{2}$ for $A$ the two-sided ideal generated by the arrows of the quiver.
\end{remark}

To state an explicit formula of the symmetrized Euler form attached to $k\widetilde{\TT}_{p,\lambda}$ we need the following property.

\begin{Lemma}\label{lem:gl.dim}
The global dimension of the bound quiver algebra $k\widetilde{\TT}_{p,\lambda}$ is two.
\end{Lemma}

\begin{proof}
By \cite{Au55} it is sufficient to calculate the projective dimensions of all simple modules $E_{\nu}$ corresponding to the vertices $\nu \in Q_{0}$.

By \cite[Chapter III, Lemma 1.11]{ARS97} we have that the global dimension of $k\widetilde{\TT}_{p,\lambda}$ is at least two. We immediately get $S_{1^{*}}=P_{1^{*}}$ and $S_{(i,p_{i}-1)}=P_{(i,p_{i}-1)}$ for $1\leq i \leq r$. A direct calculation yields for $1\leq i \leq r$ the following projective resolution
\[0\longrightarrow P_{1^{*}}\oplus \bigoplus_{j= 2}^{p_{i}-1}P_{(i,j)}\longrightarrow P_{(i,1)}\longrightarrow S_{(i,1)}\longrightarrow 0.\]
A projective resolution for $S_{1}$ is given by
\[0\longrightarrow P_{1^{*}}^{2}\longrightarrow \bigoplus_{i=1}^{r}P_{(i,1)}\longrightarrow P_{1}\longrightarrow S_{1}\longrightarrow 0\]
and projective resolutions for the remaining simples, i.e. for $1\leq i\leq r$ and $2\leq j \leq p_{i}-2$, are
\[0 \longrightarrow \bigoplus_{k=j+1}^{p_{i}-1}P_{(i,k)}\longrightarrow P_{(i,j)}\longrightarrow S_{(i,j)}\longrightarrow 0.\]
All stated projective resolutions have length at most two, which yields the assertion.
\end{proof}

\begin{figure}
  \centering
  \begin{tikzpicture}[scale=3.4]

    \node (A6666) at (0,-4) [circle, draw, fill=black!50, inner sep=0pt, minimum width=4pt] {};
    \node (AA6666) at (0,-3.9) [] {\tiny{$(1,p_1-1)$}};
    \node (A666) at (0.5,-4) [circle, draw, fill=black!50, inner sep=0pt, minimum width=4pt] {};
    \node (AA666) at (0.5,-3.9) [] {\tiny{$(1,p_1-2)$}};
    \node (A66) at (0.75,-4) [] {$\ldots$};
    \node (A6) at (1,-4) [circle, draw, fill=black!50, inner sep=0pt, minimum width=4pt] {};
    \node (AA6) at (1,-3.9) [] {\tiny{$(1,2)$}};
    \node (B6) at (1.5,-4) [circle, draw, fill=black!50, inner sep=0pt, minimum width=4pt]{};
    \node (BB6) at (1.45,-3.9) [] {\tiny{$(1,1)$}};
    \node (C6) at (2,-4) [circle, draw, fill=black!50, inner sep=0pt, minimum width=4pt]{};
    \node (CC6) at (2,-3.9) [] {\tiny{$1$}};
    \node (D6) at (2.5,-4) [circle, draw, fill=black!50, inner sep=0pt, minimum width=4pt]{};
    \node (DD6) at (2.55,-3.9) [] {\tiny{$(r,1)$}};
    \node (E6) at (3,-4) [circle, draw, fill=black!50, inner sep=0pt, minimum width=4pt]{};
    \node (EE6) at (3,-3.9) [] {\tiny{$(r,2)$}};
    \node (E66) at (3.25,-4) [] {$\ldots$};
    \node (F6) at (3.5,-4) [circle, draw, fill=black!50, inner sep=0pt, minimum width=4pt]{};
    \node (FF6) at (3.5,-3.9) [] {\tiny{$(r,p_r-2)$}};
    \node (G6) at (4,-4) [circle, draw, fill=black!50, inner sep=0pt, minimum width=4pt]{};
    \node (GG6) at (4,-3.9) [] {\tiny{$(r,p_r-1)$}};
    \node (I66) at (2,-3.5) [circle, fill=white, inner sep=0pt, minimum width=9pt]{};
    \node (I6) at (2,-3.5) [circle, draw, fill=black!50, inner sep=0pt, minimum width=4pt]{};
    \node (II6) at (2,-3.4) [] {\tiny{$1^*$}};
    \node (J6) at (2.35,-4.2) [circle, draw, fill=black!50, inner sep=0pt, minimum width=4pt]{};
    \node (JJ6) at (2.55,-4.15) []{\tiny{$(r-1,1)$}};
    \node (J66) at (2.7,-4.4) [circle, draw, fill=black!50, inner sep=0pt, minimum width=4pt]{};
    \node (JJ66) at (2.9,-4.35) []{\tiny{$(r-1,2)$}};
    \node (J66666) at (2.88,-4.5) []{$\ldots$};
    \node (J666) at (3.05,-4.6) [circle, draw, fill=black!50, inner sep=0pt, minimum width=4pt]{};
    \node (JJ666) at (3.4,-4.55) []{\tiny{$(r-1,p_{r-1}-2)$}};
    \node (J6666) at (3.4,-4.8) [circle, draw, fill=black!50, inner sep=0pt, minimum width=4pt]{};
    \node (JJ6666) at (3.75,-4.75) []{\tiny{$(r-1,p_{r-1}-1)$}};
    \node (K6) at (1.65,-4.2) [circle, draw, fill=black!50, inner sep=0pt, minimum width=4pt]{};
    \node (KK6) at (1.55,-4.15) []{\tiny{$(2,1)$}};
    \node (K66) at (1.3,-4.4) [circle, draw, fill=black!50, inner sep=0pt, minimum width=4pt]{};
    \node (KK66) at (1.2,-4.35) []{\tiny{$(2,2)$}};
    \node (K66666) at (1.12,-4.5) []{\ldots};
    \node (K666) at (0.95,-4.6) [circle, draw, fill=black!50, inner sep=0pt, minimum width=4pt]{};
    \node (KK666) at (0.75,-4.55) []{\tiny{$(2,p_2-2)$}};
    \node (K6666) at (0.6,-4.8) [circle, draw, fill=black!50, inner sep=0pt, minimum width=4pt]{};
    \node (KK6666) at (0.4,-4.75) []{\tiny{$(2,p_2-1)$}};
    \node (L6) at (2,-4.5) []{\ldots};


    \draw[->] (A666) to (A6666);
    \draw[->] (J666) to (J6666);
    \draw[->] (K666) to (K6666);
    \draw[<-] (A6) to (B6);
    \draw[<-] (B6) to (C6);
    \draw[->] (C6) to (D6);
    \draw[->] (D6) to (E6);
    \draw[->] (F6) to (G6);
    \draw[<-] (I6) to (B6);
    \draw[<-] (I6) to (D6);
    \draw[<-] (I66) to (J6);
    \draw[<-] (I66) to (K6);
    \draw[->] (J6) to (J66);
    \draw[->] (K6) to (K66);
    \draw[->] (C6) to (J6);
    \draw[->] (C6) to (K6);

  \end{tikzpicture}
  \caption{One point extended star quiver $\phantom{0000000000000000}$} \label{def:Quiver_one}
\end{figure}

\begin{Proposition} \label{lem:QuadraticForm}
Consider the one point extended star quiver of Figure \ref{def:Quiver_one}. Then the following holds.
  \begin{itemize}
    \item[(a)] For every $\nu \in Q_{0}$ there is (up to isomorphism) precisely one  simple $k\widetilde{\TT}_{p,\lambda}$-module, given by $E_{\nu}$.
    \item[(b)] There exists a complete exceptional sequence $\mathcal{E}'$ which contains the simple modules $E_{\nu}$.
    \item[(c)] The symmetrized Euler form $\chi^{s}$ on $K_{0}(k\widetilde{\TT}_{p,\lambda})$ satisfies
    \[\chi^{s}([X],[Y])=q\big( (X_{\nu})_{\nu\in Q_{0}}+(Y_{\omega})_{\omega\in Q_{0}}\big)-q\big((X_{\nu})_{\nu\in Q_{0}}\big)-q\big((Y_{\omega})_{\omega\in Q_{0}}\big)\]
    for all $X,Y\in $ mod$(k\tilde{\TT}_{p,\lambda})$ with corresponding dimension vectors $(X_{\nu})_{\nu\in Q_{0}}$ resp. $(Y_{\nu})_{\nu\in Q_{0}}$ where $q$ is the Tits form given by
    \[q((Z_{\nu})_{\nu\in Q_{0}})=\sum_{\nu \in Q_{0}}Z_{\nu}^{2}-\sum_{\nu \rightarrow \omega} Z_{\nu}Z_{\omega}+2Z_{1^{*}}Z_{1}\]
    for $(Z_{\nu})_{\nu\in Q_{0}}\in \mathbbm{N}^{|Q_{0}|}\setminus \bf{0}$.
    \item[(d)] There exists an exceptional sequence $\mathcal{E}=(\ldots, E_{1},E_{1^{*}})$ in mod$(k\widetilde{\TT}_{p,\lambda})$ such that the corresponding roots in the Grothendieck group induce an extended Coxeter-Dynkin diagram as illustrated in Figure \ref{def:GenCoxDiag}.
  \end{itemize}
\end{Proposition}

\begin{proof}
Denote by $E_{\nu}$ (up to isomorphism) the simple $k\widetilde{\TT}_{p,\lambda}$-modules corresponding to the vertices in $Q_{0}$ of the one point extended star quiver of Figure \ref{def:Quiver_one}. Each simple module $E_{\nu}$ is an exceptional object in the category mod$(k\widetilde{\TT}_{p,\lambda})$. The $k$-dimension of $\EXT_{k\widetilde{\TT}_{p,\lambda}}^{1}(E_{\nu},E_{\omega})$ for vertices $\nu,\omega \in Q_{0}$ is equal to the number of arrows from $\nu$ to $\omega$ and since the global dimension of mod$(k\widetilde{\TT}_{p,\lambda})$ is two (Lemma \ref{lem:gl.dim}), the $k$-dimension of $\EXT_{k\widetilde{\TT}_{p,\lambda}}^{2}(E_{\nu},E_{\omega})$ can easily be calculated by using \cite[Proposition 1]{Bo89}. This implies that the sequence
\[\mathcal{E}'=(E_{1},E_{(1,1)},E_{(1,2)},\ldots,E_{(1,p_{1}-1)},\ldots,E_{(r,1)},E_{(r,2)},\ldots,E_{(r,p_{r}-1)},E_{1^{*}})\]
is a complete exceptional sequence in mod$(k\widetilde{\TT}_{p,\lambda})$. An easy calculation yields that the Tits form on mod$(k\widetilde{\TT}_{p,\lambda})$ is given by
\[q((X_{\nu})_{\nu\in Q_{0}})=\sum_{\nu \in Q_{0}}X_{\nu}^{2}-\sum_{\nu \rightarrow \omega} X_{\nu}X_{\omega}+2X_{1^{*}}X_{1}\]
for $X\in $ mod$(k\widetilde{\TT}_{p,\lambda})$ with corresponding dimension vectors $(X_{\nu})_{\nu\in Q_{0}}$ (see \cite[Definition 2]{Bon83}). By part a) and \cite[Proposition 2.2]{Bon83} the Euler quadratic form and the Tits form coincide. The latter yields due to polarization the symmetrized Euler form
\[\chi^{s}([X],[Y])=q\big( (X_{\nu})_{\nu\in Q_{0}}+(Y_{\omega})_{\omega\in Q_{0}}\big)-q\big((X_{\nu})_{\nu\in Q_{0}}\big)-q\big((Y_{\omega})_{\omega\in Q_{0}}\big)\]
for all $X,Y\in $ mod$(k\widetilde{\TT}_{p,\lambda})$ with corresponding dimension vectors $(X_{\nu})_{\nu\in Q_{0}}$ resp. $(Y_{\nu})_{\nu\in Q_{0}}$. By knowing the symmetrized Euler form $\chi^{s}$ it is easy to see that the sequence $\mathcal{E}'$ induces a basis of the Grothendieck group of $\DD^{b}(k\widetilde{\TT}_{p,\lambda})$ such that the generalized Coxeter-Dynkin diagram is an extended Coxeter-Dynkin diagram (see Figure~\ref{def:GenCoxDiag}).

The number of objects in the complete exceptional sequence $\mathcal{E}'$ is given by \linebreak $m=\big(\sum_{i=1}^{r}p_{i}-1\big)+2$. By applying the braid $\sigma_{m-2}^{-1}\cdots \sigma_{2}^{-1}\sigma_{1}^{-1}$ to $\mathcal{E}'$ we obtain a new complete exceptional sequence $\mathcal{E}=(\ldots,E_{1},E_{1^{*}})$. Due to Lemma \ref{le:BraidCompatibility} the roots corresponding to $\mathcal{E}$ induce the same generalized Coxeter-Dynkin diagram.
\end{proof}

\begin{Proposition}\label{lem:root_system}
Let $\XX = (\PP_k^1,\lambda, p)$ be a weighted projective line and let $(K_0(\XX), \chi^{s}, \Delta_{\text{re}}(\XX), c)$ be the associated root quadruple. Then there is a complete exceptional sequence $\mathcal{E}=(\ldots,E_{1},E_{1^{*}})$  such that the following holds:
\begin{itemize}
\item[(a)] $\Delta_{\text{re}}(\XX)$ is an extended root system in $K_0(\XX)~\otimes_{\ZZ}~\RR$, and the generalized Coxeter-Dynkin diagram associated to $(\XX, \mathcal{E})$ is 
an extended Coxeter-Dynkin diagram;
\item[(b)] the Coxeter transformation is $c = \cdots s_{[E_1]} s_{[E_1^*]}$.
\end{itemize}
\end{Proposition}

\begin{proof}
Due to Corollary \ref{cor:BraidGroupActionExc}, \cite[Lemma 8.1.2]{Mel04} and Lemma \ref{le:Enlargement_DD}, the assumptions of \cite[Proposition 2.10]{STW16} are satisfied. So Corollary \ref{cor:genrootsys} yields the existence of an extended root system $\Delta_{\text{re}}(\XX)$ attached to $\COH(\XX)$. By \cite[Proposition 2.24]{STW16} the categories $\DD^{b}(\COH(\XX))$ and $\DD^{b}(k\widetilde{\TT}_{p,\lambda})$ are triangulated equivalent, thus their corresponding root quadruples are isomorphic (see Corollary~\ref{cor:genrootsys}). By Proposition \ref{lem:QuadraticForm} there exists a complete exceptional sequence $\mathcal{E}=(\ldots,E_{1},E_{1^{*}})$ and the corresponding elements in the Grothendieck group form a root basis of $\Delta_{\text{re}}(\XX)$ whose generalized Coxeter-Dynkin diagram is an extended diagram.
\end{proof}

\begin{remark}
Note that in contrast to \cite{STW16} we constructed an exceptional sequence with corresponding diagram as given in Figure \ref{def:GenCoxDiag} such that the two exceptional objects, which are connected by an edge of strength $2$, are neighboring in the exceptional sequence.
\end{remark}

\section{An order preserving bijection}\label{sec:OrderPreBij}
In this section we prove Theorem~\ref{conj:WeightProjElliptic0}. The interval posets have been introduced in Subsection~\ref{sec:IntervalPoset}. Here we present the other poset, that is, the poset of thick subcategories that are generated by exceptional sequences. Then, after the introduction of the Hurwitz action, Theorem~\ref{conj:WeightProjElliptic0} will be proven in the last subsection.

\subsection{Thick subcategories generated by exceptional sequences}\label{subsec:ThickSubcat}

\begin{Definition}\label{def:2outof3}
Let $\mathcal{A}$ be an abelian category and $\mathcal{H}$ a full subcategory of $\mathcal{A}$. The category $\mathcal{H}$ is called \defn{thick} if it is abelian and closed under extensions.

If $H$ is a class of objects in $\mathcal{A}$, we denote the smallest thick subcategory which contains $H$ by $\Thi(H)$ and call it the \defn{thick subcategory generated by $H$}.
\end{Definition}

\begin{remark}[{\cite[Theorem 3.3.1]{Di09} and \cite[Proposition 9.1]{IPT15}}] \label{rem:Thick_equiv}
  Let $\mathcal{A}$ be a hereditary abelian category. A full subcategory $\mathcal{H}$ of $\mathcal{A}$ is thick if and only if it is closed under direct summands and if it fulfils the following property. Given an exact sequence $0\rightarrow A \rightarrow B \rightarrow C \rightarrow 0$ in $\mathcal{A}$ and two of the three objects are in $\mathcal{H}$ then the third also lies in $\mathcal{H}$. The latter property is called '\textit{two out of three}' property.
\end{remark}

\begin{Lemma}\label{le:MutationInThick}
Let $(E,F)$ be an exceptional pair in $\COH(\mathbb{X})$, then $L_{E}F$ and $R_{F}E$ are objects in $\Thi(E,F)$.
\end{Lemma}

\begin{proof}
  We only consider the short exact sequence 
  \[0 \rightarrow L_{E}F \rightarrow \HOM_{\mathbb{X}}(E, F) \otimes E \rightarrow F \rightarrow 0.\]
  The argumentation for the other sequences is analogous. The object $\HOM_{\mathbb{X}}(E, F) \otimes E$ is isomorphic to the sum of $\text{dim}\big(\HOM_{\mathbb{X}}(E, F)\big)$ copies of $E$. Thick subcategories are closed under extensions, thus the object $\HOM_{\mathbb{X}}(E, F) \otimes E$ is in $\Thi(E,F)$. The 'two out of three' property yields that $L_{E}F$ is an object of $\Thi(E,F)$.
\end{proof}

For a subset $C \subseteq \COH(\XX)$ denote by 
$$
C^{\perp}=\lbrace X\in \COH(\XX) \mid \EXT_{\XX}^{i}(A,X)=0 ~\text{for all}~i \in \NN_{0} ~\text{and}~A\in C\rbrace
$$ 
the \defn{right perpendicular} of $C$ and analogously by 
$$
^{\perp}C=\lbrace X\in \COH(\XX) \mid \EXT_{\XX}^{i}(X,A)=0 ~\text{for all}~i \in \NN_{0} ~\text{and}~A\in C \rbrace
$$ 
the \defn{left perpendicular} of $C$.

The next result is well known and follows directly from Remark \ref{rem:Thick_equiv}.

\begin{Lemma}\label{lem:perpisthick}
Let $C \subseteq \COH(\mathbb{X})$. Then $C^{\perp}$ is a thick subcategory.
\end{Lemma}

\begin{proof}
By Remark \ref{rem:Thick_equiv} we need to check that $C^{\perp}$ is closed under direct summands and that the 'two out of three' property holds. Obviously $C^{\perp}$ is closed under direct summands and the 'two out of three' property follows by use of the long exact sequence induced by the covariant $\HOM_{\XX}$ functor.
\end{proof}

\begin{Lemma} \label{le:perpendicular}
Let $(E_1, \ldots, E_n)$ and $(F_1, \ldots , F_n)$ be complete exceptional sequences in $\COH(\mathbb{X})$ and $U=\Thi(E_1, \ldots , E_r)$, $V=\Thi(F_1, \ldots , F_s)$ for some $r,s \leq n$. Then $U=V$ if and only if $r=s$ and $(F_1, \ldots, F_r, E_{r+1}, \ldots, E_n)$ is a complete exceptional sequence.
\end{Lemma}

\begin{proof}
First assume that $(F_1, \ldots, F_s, E_{r+1}, \ldots, E_n)$ is a complete exceptional sequence. Then it holds $r=s$. Consider the right perpendicular category $\mathcal{H}:=(E_{r+1},\ldots,E_{n})^{\perp}$ in $\COH(\mathbb{X})$. By \cite[Theorems 3.3.2 and 3.3.3]{Mel04} the category $\mathcal{H}$ is equivalent to the category $\COH(\mathbb{X}')$ for a weighted projective line $\mathbb{X}'$ of reduced weight or to $\MOD(A)$ for a hereditary artin algebra $A$. By the same results the Grothendieck group is in both cases of rank $r$, so $(E_1,\ldots,E_r)$ and $(F_1,\ldots,F_r)$ are complete exceptional sequences of $\mathcal{H}$. By \cite[Theorem 1.1]{KM02} in the case of $\COH(\mathbb{X}')$ and by \cite[Corollary Ch.7]{RI94} in the case of $\MOD(A)$ the exceptional sequences $(E_1,\ldots,E_r)$ and $(F_1,\ldots,F_r)$ lie in the same orbit of the braid group action. By Lemma \ref{le:MutationInThick} thick subcategories are closed under left and right mutation, so $\Thi(E_1,\ldots,E_r)=\Thi(F_1,\dots,F_r)$ holds.

Assume that $U=V$. By definition we have $E_1,\ldots,E_r \in (E_{r+1},\ldots,E_{n})^{\perp}$ and Lemma \ref{lem:perpisthick} implies with $U=V$
\[F_1,\ldots,F_s\in \Thi(F_1,\ldots,F_s)=\Thi(E_1,\ldots,E_r)\subseteq (E_{r+1},\ldots,E_n)^{\perp}.\]
Therefore $(F_1,\ldots,F_s,E_{r+1},\ldots,E_n)$ is an exceptional sequence. By Lemma \cite[Lemma 3.25]{Sun17} we need to check that $\Thi(F_1,\ldots,F_s,E_{r+1},\ldots,E_n)=\COH(\XX)$. The latter follows by
\begin{align*}
\Thi(F_1,\ldots,F_s,E_{r+1},\ldots,E_n)&=\Thi\left(\Thi(F_1,\ldots,F_s),\Thi(E_{r+1},\ldots,E_n)\right)\\
&=\Thi\left(\Thi(E_1,\ldots,E_r),\Thi(E_{r+1},\ldots,E_n)\right)\\
&=\Thi(E_1,\ldots,E_n)\\
&=\COH(\XX).
\end{align*}
\end{proof}

\begin{remark} \label{rem:Koehler}
Note that a thick subcategory in $\COH(\XX)$ is not necessarily generated by an exceptional sequence (see for instance \cite[Proposition 6.13]{Kra12}).
\end{remark}

\subsection{The Hurwitz action on exceptional sequences}\label{subsec:Hurwitz}
We will connect the braid group action on exceptional sequences with the Hurwitz action on tuples of reflections. Observe the following nice result regarding the Hurwitz action (as introduced in Section \ref{subsec:HurwitzN}).
\begin{Lemma} \label{lem:BraidConjugation}
Let $G$ be a group, $T$ be a subset of $G$ that is closed under conjugation and $t_1, \ldots , t_m, t \in T$. Then for all $x \in \langle t_1, \ldots , t_m \rangle$ there exists a braid $\sigma \in \mathcal{B}_{m+2}$ such that
$$
\sigma(t_1, \ldots , t_m, t,t) = (t_1, \ldots , t_m, t^x,t^x).
$$
\end{Lemma}

\begin{proof}
Put $\sigma:= (\sigma_i \sigma_{i+1} \cdots \sigma_{m+1})(\sigma_{m+1} \cdots \sigma_{i+1} \sigma_i)$. A straightforward calculation shows that
$$\sigma(t_1, \ldots , t_m, t,t) = (t_1, \ldots , t_m, t^{t_i},t^{t_i}).$$
\end{proof}

Here is the aforementioned connection.

\begin{Proposition} \label{thm:BraidGroupAction}
Let $\mathbb{X}$ be a weighted projective line and let $n$ be the rank of $K_0(\mathbb{X})$. If there exist exceptional sequences $(E_1, \ldots, E_n), (F_1, \ldots , F_n)$ of $\COH(\XX)$ and a braid $\sigma \in \mathcal{B}_n$ such that $\sigma(E_1, \ldots , E_n)=(F_1, \ldots , F_n)$, then also $\sigma(s_{[E_1]}, \ldots , s_{[E_n]})= (s_{[F_1]}, \ldots , s_{[F_n]})$.
\end{Proposition}

\begin{proof}
The fact that both braid group actions are compatible is due to Lemma \ref{le:BraidCompatibility}.
\end{proof}

We can use the definition of a generating factorization to draw the following conclusion from Proposition~\ref{thm:BraidGroupAction}.

\begin{corollary} \label{cor:RefExcSeq}
Let $W$ be an extended Weyl group of rank $n$ and let $c \in W$ be a Coxeter transformation. Assume that the braid group $\mathcal{B}_n$ acts transitively on the set of sequences $(s_1, \ldots , s_n)$ of reflections in $W$ such that $c= s_1 \cdots s_n$ is a generating factorization of $c$. If $c= t_1 \cdots t_n$ is a reduced generating factorization of $c$ then there exists up to isomorphism a unique complete exceptional sequence $(F_1, \ldots , F_n)$ in $\COH(\mathbb{X})$ such that $t_i = s_{[F_i]}$, for $1\leq i \leq n$.
\end{corollary}

\begin{proof}
By Corollary \ref{cor:genrootsys} there exists a complete exceptional sequence \linebreak $(E_1, \ldots , E_n)$ such that $\spanz([E_{1}],\ldots,[E_{n}])=\spanz(\Delta_{\text{re}}(\XX))$ and $c= s_{[E_1]} \cdots s_{[E_n]}$, and by assumption  there exists $\sigma \in \mathcal{B}_n$ such that $(t_1, \ldots , t_n) = \sigma (s_{[E_1]}, \ldots , s_{[E_n]})$. Let $(F_1, \ldots , F_n)$ be the complete exceptional sequence such that $(F_1, \ldots , F_n) = \sigma (E_1, \ldots , E_n)$. It follows that $s_{[F_i]} = t_i$  by Proposition \ref{thm:BraidGroupAction}. The uniqueness follows from Lemma \ref{le:Uniqueness}.
\end{proof}

\subsection{Proof of the first main theorem, Theorem~\ref{conj:WeightProjElliptic0}} \label{sec:MainProof}
Now we will prove Theorem \ref{conj:WeightProjElliptic0}. Notice that Proposition~\ref{lem:redCoxElt} implies that the reflection length of a Coxeter transformation  equals the rank of the Grothendieck group $K_0(\XX)$.

\begin{proof}[\textbf{Proof of Theorem \ref{conj:WeightProjElliptic0}}]
Let $(E_1, \ldots , E_{n+2})$ be a complete exceptional sequence in $\COH(\XX)$. By Lemma \ref{le:InfoRootSystem} and Corollary \ref{cor:genrootsys} we have $c=s_{[E_1]} \cdots s_{[E_{n+2}]}$. For $U=\Thi(E_1, \ldots , E_r)$, where $r \leq n+2$,  put $\cox(U):=s_{[E_1]} \cdots s_{[E_r]} \in W$, where $W:= W(\XX)$ the extended Weyl group attached to $\XX$.
  
Let us first point out that $\cox(-)$ is well-defined. Therefore choose another complete exceptional sequence $(F_1, \ldots, F_{n+2})$ in $\COH(\mathbb{X})$ such that $U= \Thi(F_1, \ldots, F_s)$ for some $s \leq n+2$. Then $(F_1, \ldots, F_s, E_{r+1}, \ldots , E_{n+2})$ is a complete exeptional sequence by Lemma \ref{le:perpendicular}, and $r=s$. By Theorem \ref{thm:BraidGroupActionExc} and Proposition~\ref{thm:BraidGroupAction} we obtain
  $$
    s_{[F_1]} \cdots s_{[F_{n+2}]} = s_{[F_1]} \cdots s_{[F_r]}s_{[E_{r+1}]} \cdots s_{[E_{n+2}]} = s_{[E_1]} \cdots s_{[E_r]}s_{[E_{r+1}]} \cdots s_{[E_{n+2}]} = c.
  $$
Thus we get $s_{[F_1]} \cdots s_{[F_s]} = s_{[E_1]} \cdots s_{[E_r]}$.
  
By definition of $\cox(U)$ we have $\cox(U) \leq c$. Since the root quadruple is independent of the choice of the complete exceptional sequence it holds $\spanz([E_1], \ldots, [E_{n+2}])=\spanz(\Delta_{\text{re}}(\XX))$ and $[E_{i}]\in \Delta_{\text{re}}(\XX)$ for $1\leq i \leq n+2$. Hence $\cox(U)\in [\idop, c]^{\text{gen}}$.

Next we show that the map induced by $\cox(-)$ is injective. Let $(E_1, \ldots , E_{n+2})$ and\linebreak $(F_1, \ldots , F_{n+2})$ be complete exceptional sequences such that $U=\Thi(E_1, \ldots , E_r)$ and $V=\Thi(F_1, \ldots , F_s)$ for some $r,s \leq n+2$ and such that $\cox(U)= \cox(V)$. That is, $s_{[E_1]} \cdots s_{[E_r]}= s_{[F_1]} \cdots s_{[F_s]}$ and since $c$ is of length $n+2$ we obtain $r=s$. In particular
$$
    c=s_{[F_1]} \cdots s_{[F_r]}s_{[E_{r+1}]} \cdots s_{[E_{n+2}]}.
$$
By assumption (b) and Corollary \ref{cor:RefExcSeq} we obtain that $(F_1, \ldots , F_r, E_{r+1}, \ldots , E_{n+2})$ is a complete exceptional sequence. By Lemma \ref{le:perpendicular} we get $U=V$.

The surjectivity of $\cox(-)$ follows directly from Corollary \ref{cor:RefExcSeq}.

It remains to show that $\cox(-)$ is order preserving. Therefore let $V \subseteq U$ be thick subcategories with $U=\Thi(E_1, \ldots E_r)$ and $V=\Thi(F_1, \ldots , F_s)$. As in the proof of the Lemma \ref{le:perpendicular} the subcategory $U=(E_{r+1},\ldots,E_{n+2})^{\perp}$ is by \cite[Theorems 2.4.2 and 2.4.3]{Mel04} equivalent either to a module category $\MOD(A)$ for some hereditary algebra $A$ or to a sheaf category $\COH(\mathbb{X}')$ for a weighted projective line $\mathbb{X}'$ and $(F_1, \ldots , F_s)$ is an exceptional sequence in $U$. The rank of the Grothendieck group $K_0(U)$ implies $s \leq r$. By Lemma \ref{le:Enlargement_Coh} there exist exceptional objects $F'_{s+1}, \ldots , F'_r$ such that $(F_1, \ldots , F_s,F'_{s+1}, \ldots , F'_r)$ is a complete exceptional sequence in $U$. Therefore by Theorem \ref{thm:BraidGroupActionExc} $(E_1, \ldots E_r)$ and $(F_1, \ldots , F_s,F'_{s+1}, \ldots , F'_r)$ lie in the same orbit of the braid group action and $s_{[F_1]}\dots s_{[F_s]} s_{[F'_{s+1}]} \dots  s_{[F'_r]} = s_{[E_1]}\dots s_{[E_r]}$ by Proposition \ref{thm:BraidGroupAction}. Since $c=s_{[E_1]}\dots s_{[E_{n+2}]}$ is reduced, we have $s_{[F_1]}\dots s_{[F_s]} \leq s_{[E_1]}\dots s_{[E_r]}$.

Conversely, let $u \leq v \in [\idop, c]^{\text{gen}}$. Then there exists a reduced generating reflection factorization $c = t_1 \cdots t_{n+2}$ as well as $1 \leq i \leq j \leq n+2$ such that $u=t_1 \cdots t_i$ and $v = t_1 \cdots t_j$. By assumptions (a) and (b) this does not depend on the chosen factorization. By Corollary \ref{cor:RefExcSeq} there exists a unique complete exceptional sequence $(F_1, \ldots, F_n)$ in $\COH(\mathbb{X})$ such that $t_i = s_{[F_i]}$. In particular, we have $\Thi(F_1, \ldots , F_i) \subseteq \Thi(F_1, \ldots , F_j)$, as desired.
\end{proof}

By a theorem of Br\"uning \cite[Theorem 1.1]{Br07}, for a heriditary abelian category $\mathcal{A}$, there is a one-to-one correspondence between the thick subcategories in $\mathcal{D}^b(\mathcal{A})$ and the thick subcategories in $\mathcal{A}$. Therefore, the statement of Theorem~\ref{conj:WeightProjElliptic0} holds as well if we replace the category $\COH(\XX)$ therein by the bounded derived category $\DD^{b}(\COH(\XX))$.

\section{Elliptic root systems and elliptic Weyl groups} \label{sec:Elliptic}
In this section we introduce the notion of an elliptic root system and an elliptic Weyl group due to Saito. We recall some of their properties (see \cite{Sai85}), sometimes we add a proof if it is difficult to deduce from \cite{Sai85}. Moreover, we present some new results as for instance Lemma~\ref{lem:GenLatticeGroup}. Further we show that every root quadruple and every extended Weyl group for a category $\COH(\mathbb{X})$ of coherent sheaves over a weighted projective line $\XX$ of tubular type are an elliptic root system and an elliptic Weyl group, respectively (see Corollary \ref{prop:TubEllRoot}).

\subsection{Elliptic root systems and elliptic root basis} \label{subsec:Basis}

\begin{Definition} \label{def:EllipticRootSystem}
Let $V$ be a finite dimensional vector space over $\RR$ and $(-\mid -)$ a positive semidefinite symmetric bilinear form with $2$-dimensional radical. A reduced generalized root system $\Phi$ in $V$ is called \defn{elliptic root system} with respect to $(- \mid -)$. The \defn{rank} of $\Phi$ is the dimension of $V$. We call $W:= W_\Phi$ the \defn{elliptic Weyl group (of type $\Phi$)}. Since $\spanr(\Phi)=V$, we have that $L(\Phi)$ is a full lattice in $V$; it is  called the \defn{root lattice} of $\Phi$.
 
We simply call a root subsystem $\Phi^\prime \subseteq \Phi$ of $\Phi$ a \defn{subsystem} if $\Phi'$ is itself an elliptic root system. 

Note that the name elliptic root system was introduced in \cite{ST97} while first Saito used the name extended affine root system \cite{Sai85}.
\end{Definition}

From now on we assume that $\Phi$ is a simply-laced elliptic root system. The \defn{radical} $R$ of the bilinear form $(- \mid -)$ on $V$ and of the root lattice $L(\Phi)$ are, respectively, as $L(\Phi)$ is a full lattice in $V$,
    \begin{align*}
      R:= R_{(- \mid -)}  := & \{ x \in V \mid (x \mid y)= 0 ~\text{for all } y \in V \}
      ~\text{and }\\
    L(\Phi) \cap R = & \{ x \in L(\Phi) \mid (x \mid y)= 0 ~\text{for all } y \in L(\Phi)\}.
    \end{align*}
\medskip
Let $U$ be a one dimensional subspace of $R$ such that $L(\Phi) \cap U \neq \{0\}$, and denote by $p_R$ and $p_U$ the canonical epimorphisms $V \rightarrow V/R$ and $V \rightarrow V/U$ respectively. Let $(- \mid -)_R$ be the induced bilinear form on $V/R$, that is $(p_R(x) \mid p_R(y))_R = (x \mid y)$ for all $x,y \in V$. Define $(- \mid -)_U$ analogously.

\begin{Lemma}[{\cite[(3.1) Lemma 1, (3.2) Assertion i)]{Sai85}}] \label{le:EllipticFiniteAffine}
$\phantom{4}$
  \begin{itemize}
    \item[(a)] $\FRS := p_{R}(\Phi)$ is a finite root system with respect to $(- \mid -)_{R}$;
    \item[(b)] $\ARS :=p_{U}(\Phi)$ is an affine root system with respect to $(- \mid -)_{U}$;
    \item[(c)] $L(\Phi) \cap U$ and $ L(\Phi) \cap R$ are full lattices in $U$ and $R$, respectively.
  \end{itemize}
\end{Lemma}

From now on assume that $\Phi$ is irreducible. Then, as $U \subseteq R$ and  as $R$ is the radical of the form $(- \mid -)$, the root systems $\ARS$ and $\FRS$ are irreducible. Our next aim is to describe $\Phi$ (see Proposition~\ref{RootSystem}). Let $\{\beta_0, \ldots , \beta_n\}$ be a simple system of $\ARS$ such that $\beta_1, \ldots , \beta_n$ span a finite root system. By \cite[Theorem 5.6]{Kac83} there exist unique $m_1, \ldots, m_n \in \NN$ such that 
$$\beta_0 + \sum_{i = 1}^n m_i \beta_i~\mbox{}$$
spans the lattice $L(\ARS) \cap R/U$. Let $\alpha_i$ be a preimage of $\beta_i$ in $\Phi$ for $0 \leq i \leq n$ and set
$$b = \alpha_0 + \sum_{i = 1}^n m_i \alpha_i.$$
Then $b$ is in $L(\Phi) \cap R$ and
$$\spanz(b +U) = L(\ARS) \cap R/U.$$
Let $a$ in $U$ such that $L(\Phi) \cap U = \ZZ a$.

\begin{Lemma}\label{RadicalLattice}
  We have $L(\Phi) \cap R = \ZZ a \oplus \ZZ b$.
\end{Lemma}

\begin{proof} 
By the choice of $a$ and $b$,  the right hand side is a  direct sum. Further $a,b \in L(\Phi) \cap R$ by the choice of $a$ and $b$. Thus $\ZZ a \oplus \ZZ b \subseteq L(\Phi) \cap R$.

Now let $\alpha  \in L(\Phi) \cap R$. Then, as $U \subseteq R$ yields $(L(\Phi) + U) \cap R =  (L(\Phi) \cap R) +U$ by the Dedekind identity, and as $(L(\Phi) +U)/U = L(\ARS)$, it follows that 
$$p_U(\alpha) \in [(L(\Phi) +U)/U] \cap R/U = L(\ARS) \cap R/U. $$

Therefore $p_U(\alpha) = z b +U$  for some $z \in \ZZ$. As $\alpha, b \in L(\Phi)$ it follows $\alpha - z b \in L(\Phi) \cap U = \ZZ a$, which yields the assertion.
\end{proof}

Set 
$$\VAFF := \sum_{i = 0}^n \RR \alpha_i~\mbox{and}~\VFIN := \sum_{i = 1}^n \RR \alpha_i.$$
Then $\alpha_0$ is in $\VAFF \cap \Phi$, and $(- \mid -)$ restricted to $\VFIN$ is positive definite, which yields that $\VFIN \cap \Phi$ is a root system in $\VFIN$, which  is isomorphic to $\FRS$. In the following we identify $V_x \cap \Phi$ with $\Phi_x$ and $V_x \cap L(\Phi)$ with $L(\Phi_x)$ for $x \in \{\text{fin}, \text{aff} \}$. Now we are ready to state and proof the following description of the elliptic root system.

\begin{Proposition}\label{RootSystem}
Let $\Phi$ be an irreducible simply-laced elliptic root system with respect to $(- \mid -)$ of rank at least four. Then there are $a,b \in R= R_{(- \mid -)}$ such that $L(\Phi) \cap R = \ZZ a \oplus \ZZ b$ and
$$\Phi = \FRS \oplus \ZZ a \oplus \ZZ b.$$
\end{Proposition}

\begin{proof}
Let $a, b \in R$ and $U$ be chosen as above. Then $L(\Phi) \cap R = \ZZ a \oplus \ZZ b$ by Lemma~\ref{RadicalLattice}. Next we show $\Phi = \FRS \oplus \ZZ a \oplus \ZZ b$. Let $\alpha \in \Phi$. Then $p_U(\alpha) \in \ARS$, so
$$p_U(\alpha) = \sum_{i = 1}^n z_i \beta_i + z p_U(b)~\mbox{for some}~z_i,z \in \ZZ,$$
and $\alpha = \sum_{i = 1}^n z_i \alpha_i + z b + ra$ for some $r \in \RR$. We get $ra = \alpha - \sum_{i = 1}^n z_i \alpha_i - z b \in L(\Phi) \cap U = \ZZ a$. Thus $r \in \ZZ$. Further, as $p_R(\alpha) = \sum_{i=1}^n z_i p_R(\alpha_i)$ is in $\FRS$, we get
$$\sum_{i=1}^n z_i \alpha_i \in \FRS$$
and one inclusion holds. For the other inclusion consider for $\alpha \in  \Phi$ the set
$$K_R(\alpha):= \{x \in R~|~\alpha + x \in \Phi\},$$
that is studied in \cite[1.16]{Sai85}. Saito shows for every $\alpha \in \Phi$ that

\begin{itemize}
    \item[(a)] $K_R(\alpha)$ contains a full lattice of $R$ (the  $L$ in \cite{Sai85} is $ \VFIN$ here);
    \item[(b)] $K_R(\varphi (\alpha)) = K_R(\alpha)$ for every $\varphi \in \Aut (\Phi)$;
    \item[(c)]  $\cup_{\alpha \in \FRS} K_R(\alpha)$ generates the lattice $L(\Phi) \cap R$.
\end{itemize}

Since $\Phi$ is simply-laced, $\FRS$ is simply-laced as well, and, as $\FRS$ is finite, the related Dynkin diagram is a tree. Therefore $W_{\FRS}$ is transitive on $\FRS$. We conclude from (b) and (c) that $K_R(\beta) = \cup_{\alpha \in \FRS} K_R(\alpha)$ generates $L(\Phi) \cap R$ for every $\beta \in \Phi$.

Now we show that $K_R(\beta)$ is itself a lattice, that is a subgroup of $L(\Phi)$. Let $x, y \in K_R(\beta)$, and, as $\FRS$ is irreducible and of rank at least $2$, we may choose $\alpha_1, \alpha_2 \in \FRS$  with $(\alpha_1 \mid \alpha_2) = -1$. Then $s_{\alpha_2 +y}(\alpha_1 + x) = \alpha_1 +x + \alpha_2 +y \in \Phi$ which implies $x+y \in K_R(\alpha_1 + \alpha_2) =  K_R(\beta)$. Further $s_{\alpha_2}(-\alpha_1 - x) = -\alpha_1 - x - \alpha_2 = -(\alpha_1 + \alpha_2) - x \in \Phi$ and therefore $- x \in K_R(\beta)$ as well. Thus $K_R(\beta)$ is a lattice. Now (c) and Lemma~\ref{RadicalLattice} yield $K_R(\beta) = L(\Phi) \cap R = \ZZ a \oplus \ZZ b$, which yields the other inclusion.
\end{proof}

\begin{remark}
If the finite root system $\FRS$  is of type $A_{1}$, that is the corresponding simply-laced elliptic root system is of rank three, then there exist exactly two simply-laced elliptic root systems as described in \cite[Proposition 4.2 and Table 4.5]{BA97}. Therefore, the conclusion of Proposition~\ref{RootSystem} does not hold if the rank of $\Phi$ is less than four.
\end{remark}

Set $\widetilde{\alpha} := -\alpha_0 + b$. Then $\widetilde{\alpha} = \sum_{i = 1}^n m_i \alpha_i$ is the highest root in $\FRS$ with respect to the simple system $\{\alpha_1, \ldots , \alpha_n\}$. The Dynkin diagram of $\ARS$ is one of the diagrams $X_n^{(1)}$ given in Table Aff 1 of \cite{Kac83} where $X_n$ is one of the simply-laced types $A_n$ ($n \geq 2$), $D_n$ ($n \geq 4$) or $E_n$ ($n \in \{ 6,7,8 \}$).
\medskip

In contrast to finite root systems, the Weyl group $W= W_\Phi$ of an elliptic root system does not act anywhere properly on the ambient vector space. Hence there is no analogous of a Weyl chamber. Nevertheless Saito introduced  the notion of a basis for $\Phi$ and classified the irreducible elliptic root systems in terms of so called elliptic Dynkin diagrams in \cite{Sai85}.

Let $\Gamma_{\text{aff}} := \{\alpha_0, \ldots , \alpha_n\}$ be the set of simple roots of $\Phi \cap \VAFF = \ARS$ as chosen before Lemma~\ref{RadicalLattice}.

\begin{remark}
By \cite[Chapter 3]{Sai85} the set $\Gamma_{\text{aff}}$ is unique up to isomorphism of $\Phi$.
\end{remark}

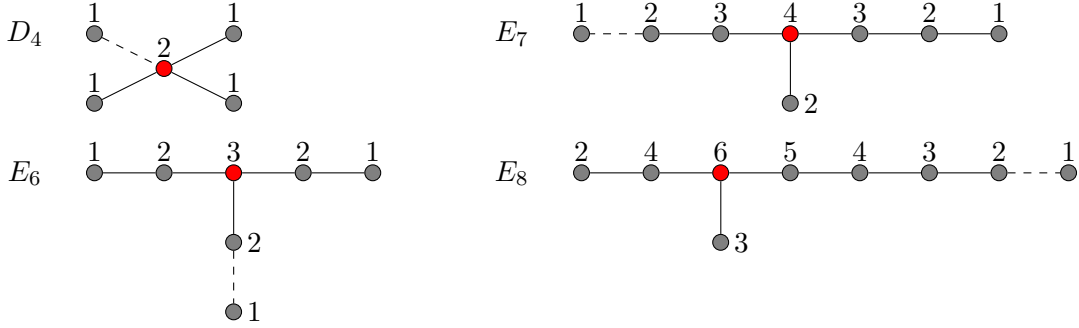
\begin{figure}
  \centering
  \begin{tikzpicture}[scale=1.85]

    \node (03) at (0.5,-1.5) [] {$D_4$};
    \node (AA33) at (1,-1.35) [] {$1$};
    \node (AA3) at (1,-1.5) [circle, draw, fill=black!50, inner sep=0pt, minimum width=6pt] {};
    \node (A33) at (1,-1.85) [] {$1$};
    \node (A3) at (1,-2) [circle, draw, fill=black!50, inner sep=0pt, minimum width=6pt] {};
    \node (B33) at (1.5,-1.6) [] {$2$};
    \node (B3) at (1.5,-1.75) [circle, draw, fill=red!150, inner sep=0pt, minimum width=6pt]{};

    \node (E33) at (2,-1.35) [] {$1$};
    \node (E3) at (2,-1.5) [circle, draw, fill=black!50, inner sep=0pt, minimum width=6pt]{};
    \node (F33) at (2,-1.85) [] {$1$};
    \node (F3) at (2,-2) [circle, draw, fill=black!50, inner sep=0pt, minimum width=6pt]{};

    \node (04) at (0.5,-2.5) [] {$E_6$};
    \node (A44) at (1,-2.35) [] {$1$};
    \node (A4) at (1,-2.5) [circle, draw, fill=black!50, inner sep=0pt, minimum width=6pt] {};
    \node (B44) at (1.5,-2.35) [] {$2$};
    \node (B4) at (1.5,-2.5) [circle, draw, fill=black!50, inner sep=0pt, minimum width=6pt]{};
    \node (C44) at (2,-2.35) [] {$3$};
    \node (C4) at (2,-2.5) [circle, draw, fill=red!150, inner sep=0pt, minimum width=6pt]{};
    \node (D44) at (2.5,-2.35) [] {$2$};
    \node (D4) at (2.5,-2.5) [circle, draw, fill=black!50, inner sep=0pt, minimum width=6pt]{};
    \node (E44) at (3,-2.35) [] {$1$};
    \node (E4) at (3,-2.5) [circle, draw, fill=black!50, inner sep=0pt, minimum width=6pt]{};
    \node (F44) at (2.15,-3) [] {$2$};
    \node (F4) at (2,-3) [circle, draw, fill=black!50, inner sep=0pt, minimum width=6pt]{};
    \node (FF44) at (2.15,-3.5) [] {$1$};
    \node (FF4) at (2,-3.5) [circle, draw, fill=black!50, inner sep=0pt, minimum width=6pt]{};

    \node (05) at (4,-1.5) [] {$E_7$};
    \node (AA55) at (4.5,-1.35) [] {$1$};
    \node (AA5) at (4.5,-1.5) [circle, draw, fill=black!50, inner sep=0pt, minimum width=6pt] {};
    \node (A55) at (5,-1.35) [] {$2$};
    \node (A5) at (5,-1.5) [circle, draw, fill=black!50, inner sep=0pt, minimum width=6pt] {};
    \node (B55) at (5.5,-1.35) [] {$3$};
    \node (B5) at (5.5,-1.5) [circle, draw, fill=black!50, inner sep=0pt, minimum width=6pt]{};
    \node (C55) at (6,-1.35) [] {$4$};
    \node (C5) at (6,-1.5) [circle, draw, fill=red!150, inner sep=0pt, minimum width=6pt]{};
    \node (D55) at (6.5,-1.35) [] {$3$};
    \node (D5) at (6.5,-1.5) [circle, draw, fill=black!50, inner sep=0pt, minimum width=6pt]{};
    \node (E55) at (7,-1.35) [] {$2$};
    \node (E5) at (7,-1.5) [circle, draw, fill=black!50, inner sep=0pt, minimum width=6pt]{};
    \node (F55) at (7.5,-1.35) [] {$1$};
    \node (F5) at (7.5,-1.5) [circle, draw, fill=black!50, inner sep=0pt, minimum width=6pt]{};
    \node (G5) at (6,-2) [circle, draw, fill=black!50, inner sep=0pt, minimum width=6pt]{};
    \node (G55) at (6.15,-2) [] {$2$};

    \node (06) at (4,-2.5) [] {$E_8$};
    \node (A6) at (4.5,-2.5) [circle, draw, fill=black!50, inner sep=0pt, minimum width=6pt] {};
    \node (A66) at (4.5,-2.35) [] {$2$};
    \node (B6) at (5,-2.5) [circle, draw, fill=black!50, inner sep=0pt, minimum width=6pt]{};
    \node (B66) at (5,-2.35) [] {$4$};
    \node (C6) at (5.5,-2.5) [circle, draw, fill=red!150, inner sep=0pt, minimum width=6pt]{};
    \node (C66) at (5.5,-2.35) [] {$6$};
    \node (D6) at (6,-2.5) [circle, draw, fill=black!50, inner sep=0pt, minimum width=6pt]{};
    \node (D66) at (6,-2.35) [] {$5$};
    \node (E6) at (6.5,-2.5) [circle, draw, fill=black!50, inner sep=0pt, minimum width=6pt]{};
    \node (E66) at (6.5,-2.35) [] {$4$};
    \node (F6) at (7,-2.5) [circle, draw, fill=black!50, inner sep=0pt, minimum width=6pt]{};
    \node (F66) at (7,-2.35) [] {$3$};
    \node (G6) at (5.5,-3) [circle, draw, fill=black!50, inner sep=0pt, minimum width=6pt]{};
    \node (G66) at (5.65,-3) [] {$3$};
    \node (H6) at (7.5,-2.5) [circle, draw, fill=black!50, inner sep=0pt, minimum width=6pt]{};
    \node (H66) at (7.5,-2.35) [] {$2$};
    \node (HH6) at (8,-2.5) [circle, draw, fill=black!50, inner sep=0pt, minimum width=6pt]{};
    \node (HH66) at (8,-2.35) [] {$1$};


    \draw[dashed] (AA3) to (B3);
    \draw[-] (A3) to (B3);
    \draw[-] (B3) to (E3);
    \draw[-] (B3) to (F3);

    \draw[-] (A4) to (B4);
    \draw[-] (B4) to (C4);
    \draw[-] (C4) to (D4);
    \draw[-] (D4) to (E4);
    \draw[-] (C4) to (F4);
    \draw[dashed] (F4) to (FF4);

    \draw[dashed] (A5) to (AA5);
    \draw[-] (A5) to (B5);
    \draw[-] (B5) to (C5);
    \draw[-] (C5) to (D5);
    \draw[-] (D5) to (E5);
    \draw[-] (E5) to (F5);
    \draw[-] (C5) to (G5);

    \draw[-] (A6) to (B6);
    \draw[-] (B6) to (C6);
    \draw[-] (C6) to (D6);
    \draw[-] (D6) to (E6);
    \draw[-] (E6) to (F6);
    \draw[-] (C6) to (G6);
    \draw[-] (F6) to (H6);
    \draw[dashed] (H6) to (HH6);

  \end{tikzpicture}
  \caption{(Extended) Dynkin diagrams. Each vertex is labeled by the corresponding coefficient $m_t$ in the linear combination of the highest root in the simple roots.} \label{fig:Dynkin}
\end{figure}

Notice that Proposition~\ref{RootSystem} also yields the splitting
$$
  \VAFF \cap L(\Phi) = \left( \bigoplus_{i = 1}^n \ZZ \alpha_i \right) \oplus \ZZ b = L(\FRS) \oplus \ZZ = b = L(\ARS).
$$
In the remainder of the paper we are interested in special irreducible elliptic root systems, the tubular ones:

\begin{Definition}\label{def:tubular}
An irreducible elliptic root system $\Phi$ is called \defn{tubular} if $\FRS$ is of type $D_4, E_6, E_7$ or $E_8$, respectively. The corresponding elliptic Weyl group $W=W_{\Phi}$ is called \defn{tubular} elliptic Weyl group.
\end{Definition}
The term tubular will be justified in Section \ref{subsec:TubEllRootSystem}.

\begin{remark}\label{rem:root_coeff}
For the types $D_4, E_6,E_7$ and $E_8$, we list in Figure \ref{fig:Dynkin} the non-negative integers $m_i$ such that $\widetilde{\alpha} = - \alpha_0 + b = \sum_{i= 1}^n m_i \alpha_i$ in the corresponding (extended) Dynkin diagrams next to  the vertices  representing $\alpha_i$ (for $1 \leq i \leq n$).
\end{remark}

Now we define a root basis for an elliptic root system. If the root system is tubular, then this root basis is a basis of the underlying vector space $V$. Notice that our definition agrees with that one in \cite[Section 8]{Sai85} by Proposition~\ref{RootSystem}.

\begin{Definition}\label{DefinitionRootBasis}
Let $\Phi$ be an irreducible simply-laced elliptic root system, $U$ a $1$-dimensional subspace of the radical $R$, $a \in U$ such that $L(\Phi) \cap U = \ZZ a$, and $\Gamma_{\text{aff}} = \{\alpha_0, \ldots , \alpha_n\}$. For $\alpha \in  \Gamma_{\text{aff}}$ put $\alpha^*  = \alpha + a$, and define 
$$m_{\text{max}}  := \max \{ m_{i} \mid 1 \leq i \leq n \}~\mbox{and}~ \Gamma_{\max}=\{ \alpha_i^* \mid m_{i}= m_{\max} \}.$$
Then the set $\Gamma:= \Gamma(\Phi):= \Gamma_{\text{aff}} \cup \Gamma_{\max}$ is called \defn{elliptic root basis} for $\Phi$.
\end{Definition}

Note that the name elliptic root basis was first used in \cite{SY00}.

\begin{Proposition}[{\cite[Section 9]{Sai85}}]
  Let $\Gamma = \Gamma(\Phi)$ be an elliptic root basis of an irreducible simply-laced elliptic root system $\Phi$. Then the following holds.
  \begin{enumerate}
    \item[(a)] $W= W_{\Gamma}$;
    \item[(b)] $\Phi = W_{\Gamma}(\Gamma)$.
  \end{enumerate}
\end{Proposition}

\subsection{The elliptic Dynkin diagrams}
Let $\Phi$ be an irreducible simply-laced elliptic root system with elliptic root basis $\Gamma$. The \defn{elliptic Dynkin diagram} $\Lambda$ with respect to $\Gamma$ is defined as the undirected graph with vertex set in bijection with $\Gamma$. Let $\alpha, \beta \in \Gamma$ be different roots. If $(\alpha \mid \beta)= 0$ there is no edge between the vertices corresponding to $\alpha$ and $\beta$. If $(\alpha \mid \beta)= \pm 1$ there is a single edge between the corresponding vertices. This edge is dotted if $(\alpha \mid \beta)=1$. If $(\alpha \mid \beta)= \pm 2$ there is a double edge between the corresponding vertices. This edge is again dotted if $(\alpha \mid \beta)=2$. Note that $(\alpha \mid \beta) \in \{ 0, \pm 1, \pm 2 \}$ for all $\alpha, \beta \in \Gamma$ if $\Phi$ is simply-laced.

\begin{Proposition}[{\cite[(9.6) Theorem]{Sai85}}] \label{thm:SaitoClassification}
 Let $(\Phi,\Gamma)$ be an irreducible simply-laced elliptic root system with elliptic root basis $\Gamma$. The elliptic Dynkin diagram for $(\Phi, \Gamma)$ is uniquely determined by the isomorphism class of $\Phi$. Conversely, the elliptic Dynkin diagram for $\Phi$ uniquely determines the isomorphism class of $(\Phi, \Gamma)$.
\end{Proposition}

If $X$ is the elliptic Dynkin diagram of the irreducible simply-laced elliptic root system $(\Phi,\Gamma)$, then
we say also that $(\Phi,\Gamma)$ or simply $\Phi$ is \defn{of type $X$}.
\smallskip

\begin{figure}
  \centering
  \begin{tikzpicture}[scale=1.75]

    \node (03) at (0.5,-1.75) [] {$D_4$};
    \node (A33) at (1,-1.6) [] {$1$};
    \node (A3) at (1,-1.75) [circle, draw, fill=black!50, inner sep=0pt, minimum width=4pt] {};
    \node (B33) at (1.5,-1.6) [] {$2$};
    \node (B3) at (1.5,-1.75) [circle, draw, fill=black!50, inner sep=0pt, minimum width=4pt]{};
    \node (E33) at (2,-1.35) [] {$3$};
    \node (E3) at (2,-1.5) [circle, draw, fill=black!50, inner sep=0pt, minimum width=4pt]{};
    \node (F33) at (2,-1.85) [] {$4$};
    \node (F3) at (2,-2) [circle, draw, fill=black!50, inner sep=0pt, minimum width=4pt]{};

    \node (04) at (0.5,-2.5) [] {$E_6$};
    \node (A44) at (1,-2.35) [] {$1$};
    \node (A4) at (1,-2.5) [circle, draw, fill=black!50, inner sep=0pt, minimum width=4pt] {};
    \node (B44) at (1.5,-2.35) [] {$3$};
    \node (B4) at (1.5,-2.5) [circle, draw, fill=black!50, inner sep=0pt, minimum width=4pt]{};
    \node (C44) at (2,-2.35) [] {$4$};
    \node (C4) at (2,-2.5) [circle, draw, fill=black!50, inner sep=0pt, minimum width=4pt]{};
    \node (D44) at (2.5,-2.35) [] {$5$};
    \node (D4) at (2.5,-2.5) [circle, draw, fill=black!50, inner sep=0pt, minimum width=4pt]{};
    \node (E44) at (3,-2.35) [] {$6$};
    \node (E4) at (3,-2.5) [circle, draw, fill=black!50, inner sep=0pt, minimum width=4pt]{};
    \node (F44) at (2.15,-3) [] {$2$};
    \node (F4) at (2,-3) [circle, draw, fill=black!50, inner sep=0pt, minimum width=4pt]{};

    \node (05) at (4.5,-1.5) [] {$E_7$};
    \node (A55) at (5,-1.35) [] {$1$};
    \node (A5) at (5,-1.5) [circle, draw, fill=black!50, inner sep=0pt, minimum width=4pt] {};
    \node (B55) at (5.5,-1.35) [] {$3$};
    \node (B5) at (5.5,-1.5) [circle, draw, fill=black!50, inner sep=0pt, minimum width=4pt]{};
    \node (C55) at (6,-1.35) [] {$4$};
    \node (C5) at (6,-1.5) [circle, draw, fill=black!50, inner sep=0pt, minimum width=4pt]{};
    \node (D55) at (6.5,-1.35) [] {$5$};
    \node (D5) at (6.5,-1.5) [circle, draw, fill=black!50, inner sep=0pt, minimum width=4pt]{};
    \node (E55) at (7,-1.35) [] {$6$};
    \node (E5) at (7,-1.5) [circle, draw, fill=black!50, inner sep=0pt, minimum width=4pt]{};
    \node (F55) at (7.5,-1.35) [] {$7$};
    \node (F5) at (7.5,-1.5) [circle, draw, fill=black!50, inner sep=0pt, minimum width=4pt]{};
    \node (G5) at (6,-2) [circle, draw, fill=black!50, inner sep=0pt, minimum width=4pt]{};
    \node (G55) at (6.15,-2) [] {$2$};

    \node (06) at (4.5,-2.5) [] {$E_8$};
    \node (A6) at (5,-2.5) [circle, draw, fill=black!50, inner sep=0pt, minimum width=4pt] {};
    \node (A66) at (5,-2.35) [] {$1$};
    \node (B6) at (5.5,-2.5) [circle, draw, fill=black!50, inner sep=0pt, minimum width=4pt]{};
    \node (B66) at (5.5,-2.35) [] {$3$};
    \node (C6) at (6,-2.5) [circle, draw, fill=black!50, inner sep=0pt, minimum width=4pt]{};
    \node (C66) at (6,-2.35) [] {$4$};
    \node (D6) at (6.5,-2.5) [circle, draw, fill=black!50, inner sep=0pt, minimum width=4pt]{};
    \node (D66) at (6.5,-2.35) [] {$5$};
    \node (E6) at (7,-2.5) [circle, draw, fill=black!50, inner sep=0pt, minimum width=4pt]{};
    \node (E66) at (7,-2.35) [] {$6$};
    \node (F6) at (7.5,-2.5) [circle, draw, fill=black!50, inner sep=0pt, minimum width=4pt]{};
    \node (F66) at (7.5,-2.35) [] {$7$};
    \node (G6) at (6,-3) [circle, draw, fill=black!50, inner sep=0pt, minimum width=4pt]{};
    \node (G66) at (6.15,-3) [] {$2$};
    \node (H6) at (8,-2.5) [circle, draw, fill=black!50, inner sep=0pt, minimum width=4pt]{};
    \node (H66) at (8,-2.35) [] {$8$};


    \draw[-] (A3) to (B3);
    \draw[-] (B3) to (E3);
    \draw[-] (B3) to (F3);

    \draw[-] (A4) to (B4);
    \draw[-] (B4) to (C4);
    \draw[-] (C4) to (D4);
    \draw[-] (D4) to (E4);
    \draw[-] (C4) to (F4);

    \draw[-] (A5) to (B5);
    \draw[-] (B5) to (C5);
    \draw[-] (C5) to (D5);
    \draw[-] (D5) to (E5);
    \draw[-] (E5) to (F5);
    \draw[-] (C5) to (G5);

    \draw[-] (A6) to (B6);
    \draw[-] (B6) to (C6);
    \draw[-] (C6) to (D6);
    \draw[-] (D6) to (E6);
    \draw[-] (E6) to (F6);
    \draw[-] (C6) to (G6);
    \draw[-] (F6) to (H6);

  \end{tikzpicture}
  \caption{Dynkin diagrams: Numbering $\phantom{000000000000000}$} \label{fig:DynkinBourbaki}
\end{figure}

\begin{example}\label{ex:rootbasis}
Let $\Phi$ be a tubular elliptic root system.  Then $\FRS$ is a finite root system of type $X_n$ where the diagram $X_n$ is given as in Figure~\ref{fig:EllipticDynkin}. According to \cite{Sai85} the elliptic Dynkin diagram is one of the diagrams $X_n^{(1,1)}$ given in Figure~\ref{fig:EllipticDynkin}. The elliptic Dynkin diagrams can as well be realized as extended Dynkin diagrams (see Figure \ref{def:GenCoxDiag}) by choosing the following parameters.
$$\begin{array}{|l|l|l|l|l|}
	\hline
	    & r & (p_1,\ldots,p_r) \\ \hline
	   D_4^{(1,1)} & 4 & (2,2,2,2) \\ \hline
    E_6^{(1,1)} & 3 & (3,3,3) \\ \hline
    E_7^{(1,1)} & 3 & (4,4,2) \\ \hline
    E_8^{(1,1)} & 3 & (6,3,2) \\ 
	\hline
\end{array}.$$

By Proposition~\ref{thm:SaitoClassification} an elliptic root basis $\Gamma$ for $\Phi$ can explicitly be defined. 

As before we choose $\{ \alpha_1, \ldots, \alpha_n \}$ as a basis of $\FRS$ where we number the simple roots as given in Figure~\ref{fig:DynkinBourbaki}. Recall that $\alpha_0=-\widetilde{\alpha} +b$. Then
  $$
   \Gamma(D_4^{(1,1)}):= \{\alpha_0, \alpha_1, \ldots , \alpha_4, \alpha_2^* \}
  $$
  is an elliptic root basis for $D_4^{(1,1)}$ and
  $$
   \Gamma(E_n^{(1,1)}):=  \{\alpha_0, \alpha_1, \ldots , \alpha_n, \alpha_4^* \}
  $$
is an elliptic root basis for $E_n^{(1,1)}$ (for $n=6,7,8$).
  
In particular, this shows  that every tubular elliptic root system and every tubular elliptic Weyl group is  also an extended root system and an extended Weyl group, respectively.
\end{example}

\begin{figure}
  \centering
  \begin{tikzpicture}[scale=1.9]

    \node (03) at (0.5,0.5) [] {$D_4^{(1,1)}$};
    \node (G3) at (2,1) [circle, draw, fill=black!50, inner sep=0pt, minimum width=4pt]{};
    \node (A3) at (1.5,0.5) [circle, draw, fill=black!50, inner sep=0pt, minimum width=4pt]{};
    \node (B3) at (2.5,0.5) [circle, draw, fill=black!50, inner sep=0pt, minimum width=4pt]{};
    \node (C3) at (2,0.5) [circle, draw, fill=black!50, inner sep=0pt, minimum width=4pt]{};
    \node (E3) at (1.6, 0.1) [circle, draw, fill=black!50, inner sep=0pt, minimum width=4pt]{};
    \node (F3) at (2.4, 0.1) [circle, draw, fill=black!50, inner sep=0pt, minimum width=4pt]{};

    \node (04) at (0.5,-1) [] {$E_6^{(1,1)}$};
    \node (A4) at (1,-1) [circle, draw, fill=black!50, inner sep=0pt, minimum width=4pt] {};
    \node (B4) at (1.5,-1) [circle, draw, fill=black!50, inner sep=0pt, minimum width=4pt]{};
    \node (C4) at (2,-1) [circle, draw, fill=black!50, inner sep=0pt, minimum width=4pt]{};
    \node (D4) at (2.5,-1) [circle, draw, fill=black!50, inner sep=0pt, minimum width=4pt]{};
    \node (E4) at (3,-1) [circle, draw, fill=black!50, inner sep=0pt, minimum width=4pt]{};
    \node (F4) at (2.4,-1.4) [circle, draw, fill=black!50, inner sep=0pt, minimum width=4pt]{};
    \node (G4) at (2.8,-1.8) [circle, draw, fill=black!50, inner sep=0pt, minimum width=4pt]{};
    \node (H4) at (2,-0.5) [circle, draw, fill=black!50, inner sep=0pt, minimum width=4pt]{};

    \node (05) at (4.2,0.5) [] {$E_7^{(1,1)}$};
    \node (A5) at (4.7,0.5) [circle, draw, fill=black!50, inner sep=0pt, minimum width=4pt] {};
    \node (B5) at (5.2,0.5) [circle, draw, fill=black!50, inner sep=0pt, minimum width=4pt]{};
    \node (C5) at (5.7,0.5) [circle, draw, fill=black!50, inner sep=0pt, minimum width=4pt]{};
    \node (D5) at (6.2,0.5) [circle, draw, fill=black!50, inner sep=0pt, minimum width=4pt]{};
    \node (E5) at (6.7,0.5) [circle, draw, fill=black!50, inner sep=0pt, minimum width=4pt]{};
    \node (H5) at (6.6,0.1) [circle, draw, fill=black!50, inner sep=0pt, minimum width=4pt]{};
    \node (F5) at (7.2,0.5) [circle, draw, fill=black!50, inner sep=0pt, minimum width=4pt]{};
    \node (G5) at (7.7,0.5) [circle, draw, fill=black!50, inner sep=0pt, minimum width=4pt]{};
    \node (I5) at (6.2,1) [circle, draw, fill=black!50, inner sep=0pt, minimum width=4pt]{};

    \node (06) at (4.2,-1) [] {$E_8^{(1,1)}$};
    \node (A6) at (4.7,-1) [circle, draw, fill=black!50, inner sep=0pt, minimum width=4pt] {};
    \node (B6) at (5.2,-1) [circle, draw, fill=black!50, inner sep=0pt, minimum width=4pt]{};
    \node (C6) at (5.7,-1) [circle, draw, fill=black!50, inner sep=0pt, minimum width=4pt]{};
    \node (D6) at (6.2,-1) [circle, draw, fill=black!50, inner sep=0pt, minimum width=4pt]{};
    \node (E6) at (6.7,-1) [circle, draw, fill=black!50, inner sep=0pt, minimum width=4pt]{};
    \node (F6) at (7.2,-1) [circle, draw, fill=black!50, inner sep=0pt, minimum width=4pt]{};
    \node (G6) at (7.7,-1) [circle, draw, fill=black!50, inner sep=0pt, minimum width=4pt]{};
    \node (H6) at (8.2,-1) [circle, draw, fill=black!50, inner sep=0pt, minimum width=4pt]{};
    \node (I6) at (5.7,-0.5) [circle, draw, fill=black!50, inner sep=0pt, minimum width=4pt]{};
    \node (J6) at (6.1,-1.4) [circle, draw, fill=black!50, inner sep=0pt, minimum width=4pt]{};
    

    \draw[-] (A3) to (C3);
    \draw[-] (C3) to (B3);
    \draw[-] (C3) to (E3);
    \draw[-] (C3) to (F3);
    \draw[-] (G3) to (A3);
    \draw[-] (G3) to (B3);
    \draw[-] (G3) to (E3);
    \draw[-] (G3) to (F3);
    \draw[dashed] ([xshift=0.5]C3.north) to ([xshift=0.5]G3.south);
    \draw[dashed] ([xshift=-0.5]C3.north) to ([xshift=-0.5]G3.south);

    \draw[-] (A4) to (B4);
    \draw[-] (B4) to (C4);
    \draw[-] (C4) to (D4);
    \draw[-] (D4) to (E4);
    \draw[-] (C4) to (F4);
    \draw[-] (F4) to (G4);
    \draw[dashed] ([xshift=0.5]C4.north) to ([xshift=0.5]H4.south);
    \draw[dashed] ([xshift=-0.5]C4.north) to ([xshift=-0.5]H4.south);
    \draw[-] (B4) to (H4);
    \draw[-] (D4) to (H4);
    \draw[-] (F4) to (H4);

    \draw[-] (A5) to (B5);
    \draw[-] (B5) to (C5);
    \draw[-] (C5) to (D5);
    \draw[-] (D5) to (E5);
    \draw[-] (E5) to (F5);
    \draw[-] (F5) to (G5);
    \draw[-] (D5) to (H5);
    \draw[-] (I5) to (C5);
    \draw[-] (I5) to (E5);
    \draw[-] (I5) to (H5);
    \draw[dashed] ([xshift=0.5]D5.north) to ([xshift=0.5]I5.south);
    \draw[dashed] ([xshift=-0.5]D5.north) to ([xshift=-0.5]I5.south);

    \draw[-] (A6) to (B6);
    \draw[-] (B6) to (C6);
    \draw[-] (C6) to (D6);
    \draw[-] (D6) to (E6);
    \draw[-] (E6) to (F6);
    \draw[-] (F6) to (G6);
    \draw[-] (G6) to (H6);
    \draw[-] (I6) to (B6);
    \draw[-] (I6) to (D6);
    \draw[-] (I6) to (J6);
    \draw[-] (C6) to (J6);
    \draw[dashed] ([xshift=0.5]C6.north) to ([xshift=0.5]I6.south);
    \draw[dashed] ([xshift=-0.5]C6.north) to ([xshift=-0.5]I6.south);

  \end{tikzpicture}
  \caption{Elliptic Dynkin diagrams for the tubular elliptic root systems} \label{fig:EllipticDynkin}
\end{figure}

\subsection{Generation of $L(\Phi)$}
Here we study the bases of the lattice $L(\Phi)$ and show that in a simply-laced elliptic root system of rank $n$ a set of $n$ roots generates the root lattice if and only if the corresponding reflections generate the corresponding elliptic Weyl group (see Lemma~\ref{lem:GenLatticeGroup}).

\begin{Lemma} \label{lem:LengthDetRoots}
Let $\Phi$ be an irreducible simply-laced elliptic root system of rank at least four and $\Phi^\prime \subseteq \Phi$ a subsystem of the same rank as $\Phi$. Then:
\begin{enumerate}
\item[(a)] $\Phi^\prime = \{ \alpha \in L(\Phi^\prime) \mid (\alpha \mid \alpha)=2 \}$;
\item[(b)] If in addition, $L(\Phi) = L(\Phi^\prime)$, then $\Phi = \Phi^\prime$.
\end{enumerate}
\end{Lemma}

\begin{proof}
One inclusion is clear since $\Phi^\prime$ is simply-laced. Therefore let $\alpha \in L(\Phi^\prime)$ with $(\alpha \mid \alpha)=2$. By Proposition \ref{RootSystem} there exist $a^\prime,b^\prime \in R$ such that
$$
\Phi^\prime = \Phi^\prime_{\text{fin}} \oplus \ZZ a^\prime \oplus \ZZ b^\prime.
$$
In particular, we have
$$
L(\Phi^\prime) = L(\Phi^\prime_{\text{fin}} ) \oplus \ZZ a^\prime \oplus \ZZ b^\prime.
$$
Hence $\alpha = \overline{\alpha}+k a^\prime+\ell b^\prime$ with $\overline{\alpha} \in L(\Phi^\prime_{\text{fin}})$. Therefore we have
$$
2 = (\alpha \mid \alpha) = (\overline{\alpha} \mid \overline{\alpha}).
$$
By \cite[Lemma 5.7]{BGRW} or \cite[Theorem 3.12]{HK16} the condition $(\overline{\alpha} \mid \overline{\alpha})=2$ is equivalent to $\overline{\alpha} \in \Phi^\prime_{\text{fin}} $. Hence $\alpha \in \Phi^\prime$, which proves $(a)$. For part (b) note that
$$
\Phi^\prime \stackrel{(a)}{=} \{ \alpha \in L(\Phi^\prime) \mid (\alpha \mid \alpha)=2 \} = \{ \alpha \in L(\Phi) \mid (\alpha \mid \alpha)=2 \} \stackrel{(a)}{=} \Phi.
$$
\end{proof}

\begin{Lemma} \label{lem:GenLatticeGroup}
Let $\Phi$ be an irreducible simply-laced elliptic root system of rank $n \geq 4$, $\Phi = \FRS \oplus \ZZ a \oplus \ZZ b$ and $\beta_{1},\ldots,\beta_{n}\in \Phi$. Then $W = \langle s_{\beta_{1}},\ldots,s_{\beta_{n}} \rangle$ if and only if $L(\{\beta_{1},\ldots,\beta_{n}\})=L(\Phi)$.
\end{Lemma}

\begin{proof}
Let $W = \langle s_{\beta_{1}},\ldots,s_{\beta_{n}} \rangle$. Since the irreducible simply laced elliptic root systems have elliptic Dynkin diagrams that contain spanning trees with simple edges (see \cite[Table 1]{Sai85}), Lemma~\ref{lem:gen_root_lattice} stays valid for these systems. The latter implies $L(\{\beta_{1},\ldots,\beta_{n}\})=L(\Phi)$.

Now assume that $L(\{\beta_{1},\ldots,\beta_{n}\})=L(\Phi)$. Put $Q:= \{ \beta_{1},\ldots,\beta_{n} \}$ and $\Phi':= W_Q(Q)$. Then $L(\Phi) = L(Q) \subseteq L(\Phi')  \subseteq L(\Phi)$ which shows $L(\Phi')  = L(\Phi)$. Thus $V = \spanr(L(\Phi)) = \spanr(L(\Phi^\prime)) = \spanr(\Phi^\prime) \subseteq V$. Therefore $V = \spanr(\Phi^\prime)$ and $\Phi^\prime$ is a root subsystem of $\Phi$. Hence Lemma \ref{lem:LengthDetRoots} (b) yields $\Phi^\prime = \Phi$. It remains to show that $W_Q = W_{\Phi^\prime}$, as this then yields $W_Q = W_{\Phi^\prime} = W_\Phi = W$. As $Q \subseteq \Phi^\prime$ we have $W_Q \subseteq W_{\Phi^\prime}$. Let $\alpha \in \Phi^\prime = W_Q(Q)$, that is $\alpha = w(\beta)$ for some $w \in W_Q$ and $\beta \in Q$. Then $s_{\alpha} = s_{w(\beta)} = w s_{\beta} w^{-1} \in W_Q,$ so $W_{\Phi^\prime} \subseteq W_Q$ as well.
\end{proof}

\subsection{The elliptic Weyl groups}\label{subsec:EllipticWeylGroup}

The aim of this section is to provide a tool, the Eichler-Siegel map, for the study of the elliptic Weyl groups, and to recall from \cite{Sai85} the structure of an elliptic Weyl group for an irreducible simply-laced elliptic root system $\Phi$.

Recall that $V = \VFIN \oplus \RR a \oplus \RR b$. Let $v = v_f + ra +sb$ for $v \in V$ where $r,s \in \RR$ and $v_f \in  \VFIN $. We will denote $v_f$ by $\overline{v}$.
It follows that  every $\alpha$ in $\Phi$ splits uniquely into $\alpha = \alpha_f + \alpha_R$ where $\alpha_f = \overline{\alpha} \in \FRS$ and $\alpha_R \in L(\Phi) \cap R =  \ZZ a \oplus \ZZ b$ by Lemma~\ref{RadicalLattice}.

As in \cite[Section (1.14)]{Sai85} we define a semi-group structure $\circ$ on $V \otimes_{\ZZ} (V/ R)$ by
$$\varphi_1 \circ \varphi_2:= \varphi_1 + \varphi_2 - (\varphi_1 \mid \varphi_2),$$
where $(\varphi_1 \mid \varphi_2) := \sum_{i_1,i_2} f_{i_1}^1 \otimes (g_{i_1}^1 \mid f_{i_2}^2) g_{i_2}^2$ for $\varphi_j = \sum_{i_j} f_{i_j}^j \otimes g_{i_j}^j$, for $j = 1,2$. The map
$$E: V \otimes_{\ZZ} (V/ R) \rightarrow \END(V),~ \sum_i f_i \otimes \overline{g}_i \mapsto \left(v \mapsto v- \sum_i  (g_i \mid v)f_i \right)$$
is called \defn{Eichler-Siegel map} for $V$ with respect to $(- \mid -)$. We collect some properties of this map.

\begin{Lemma}[{\cite[Sections (1.14) and (1.15)]{Sai85}}]\label{Eichler}
The following hold.
\begin{itemize}
    \item[(a)] $E$ is injective;
    \item[(b)] $E$ is a homomorphism of semi-groups that is $E(\varphi \circ \psi) = E(\varphi) E(\psi)$;
    \item[(c)] the subspace $R \otimes_{\ZZ} (V / R)$ of $V \otimes_{\ZZ} V/R$ is closed under $\circ$ and the semi-group structure coincides with the additive structure of the vector space on this subspace;
    \item[(d)]  Let $\alpha \in V$ non-isotropic. Then  $s_{\alpha} = E(\alpha \otimes \overline{\alpha})$;
    \item[(e)] the inverse of $E$ on $W$ is well-defined:
          $$E^{-1}: W \rightarrow V \otimes_{\ZZ} (V/ R);$$
    \item[(f)]  $E^{-1}(W) \subseteq L(\Phi) \otimes_{\ZZ} (L(\Phi) / ( L(\Phi) \cap R))$;
    \item[(g)] Let $\alpha \in \Phi$ and $\alpha = \overline{\alpha} + r$ where $r \in R$. Then  
    $$s_{\alpha} = E((\overline{\alpha} +r) \otimes \overline{\alpha}) = E(\overline{\alpha}\otimes \overline{\alpha})E(r \otimes \overline{\alpha}) = s_{\overline{\alpha}} E(r \otimes \overline{\alpha}).$$
  \end{itemize}
\end{Lemma}

Using Lemma~\ref{Eichler} (g)  we can write down the reflection $s_\alpha$ explicitly.

\begin{Lemma}\label{Reflection}
  Let $\alpha$ be a root in $\Phi$ and
  $\alpha = \alpha_f + \alpha_R$ where $\alpha_f$ is in $ \VFIN$ and $\alpha_R$  in $R$. Then $s_\alpha (v) = v - (\alpha_f \mid v)\alpha_f - (\alpha_f \mid v)\alpha_R$ for all $v \in V$.
\end{Lemma}

\medskip
\begin{Lemma}\label{lem:Translation}
Let $\alpha_{f} \in  \FRS$, $r \in R$, $x \in \{a,b\}$ and $v \in V$. Then the following hold.
\begin{itemize}
\item[(a)] $s_{\alpha_{f}} s_{\alpha_{f} +r} = E(r \otimes \overline{\alpha_{f}})$;
\item[(b)] $w E(r \otimes  \overline{\alpha_{f}}) w^{-1} = E(r \otimes
\overline{w(\alpha_{f})})$ for all $w \in W$;
\item[(c)] $E(x \otimes - v_f) s_{\alpha_{f} +kx} E(x \otimes v_f) = 
s_{\alpha_{f} +(k - (\alpha_f \mid v))x}$ for every $k \in \ZZ$.
\end{itemize}
\end{Lemma}

\begin{proof}
Assertion (a) is an immediate consequence of Lemma~\ref{Eichler} (h). Part (b) can be deduced from part (a): We have
$$
E(r \otimes \overline{w(\alpha_{f})}) \stackrel{(a)}{=} s_{w(\alpha_{f})}s_{w(\alpha_{f}) + r} = s_{w(\alpha_{f})}s_{w(\alpha_{f} + r)} =
w s_{\alpha_{f}}s_{\alpha_{f} + r} w^{-1} \stackrel{(a)}{=}
w E(r \otimes  \overline{\alpha_{f}}) w^{-1}.
$$
Assertion (c) follows with Lemma~\ref{Eichler} (b) and (c), and as by (a) we have 
$$E(x \otimes - v_f) s_{\alpha_{f} +kx} E(x \otimes v_f) = s_{\alpha_f} 
E(x \otimes - s_{\alpha_f}(v_f)) E(kx \otimes \alpha_f)E(x \otimes v_f).$$
\end{proof}

\begin{Lemma} \label{le:LatticeFactor}
Let $\Phi$ be an irreducible simply-laced elliptic root system. Then
$$T_R := E^{-1}(W) \cap \left(R \otimes_{\ZZ} (V/ R) \right) = (L(\Phi) \cap R) \otimes_{\ZZ} L({\FRS})= (\ZZ a + \ZZ b)\otimes_{\ZZ} L({\FRS})$$
is a full lattice in $R \otimes_{\ZZ} L({\FRS})$.
\end{Lemma}

\begin{proof} We know $(L(\Phi) \cap R) \otimes_{\ZZ} L({\FRS}) = (\ZZ a + \ZZ b)\otimes_{\ZZ} L({\FRS})$ by Lemma~\ref{RadicalLattice}, and by Proposition~\ref{RootSystem} we have $\alpha + r \in \Phi$ for every $\alpha \in \FRS$ and $r \in L(\Phi) \cap R$. Therefore, Lemma~\ref{lem:Translation} (a) yields $(L(\Phi) \cap R) \otimes_{\ZZ} L({\FRS}) \subseteq T_R$. The equality follows with Lemma~\ref{Eichler} (f).
\end{proof}

\medskip
\noindent By \cite[Section 1.15]{Sai85} and Lemma \ref{le:LatticeFactor}, we obtain the following description of $W$.

\begin{theorem} \label{cor:SplitEllipticWeylGroup}
  Let $\Phi$ be a simply-laced elliptic root system of rank $n+2$. Then
  $$W =  W_{\FRS}\ltimes  E(T_R) \cong 
  W_{\FRS} \ltimes ((\ZZ a + \ZZ b)\otimes_{\ZZ} L({\FRS}))$$
  is the split extension of the finite Coxeter group $W_{\FRS}$ of rank $n$ by a free abelian group of rank $2(n-1)$, and the action of $W_{\FRS}$ on $E(T_R)$ is given by 
  Lemma~\ref{lem:Translation} (b).
\end{theorem}

\begin{remark}\label{rem:GroupOperation}
The group operation in $W_{\FRS}\ltimes (\spanz(a,b) \otimes_{\ZZ} L(\FRS))$ is as follows 
\begin{align*}
{} &  w_1 E(a \otimes x_1 + b \otimes x_2) \cdot w_2 E(a \otimes y_1 + b \otimes y_2)\\
{} & =
(w_1w_2) (w_2^{-1}E(a \otimes x_1 + b \otimes x_2) w_2) E(a \otimes y_1 + b \otimes y_2)\\
 {} & = 
(w_1w_2) E(a \otimes w_2^{-1}(x_1) + b \otimes w_2^{-1}(x_2)) E(a \otimes y_1 + b \otimes y_2)\\
{} & =
(w_1w_2) E(a \otimes (w_2^{-1}(x_1)+y_1) + b \otimes (w_2^{-1}(x_2)+y_2)),
\end{align*}
where we applied part (g) of Lemma \ref{Eichler} to obtain the last equality.
\end{remark}

\subsection{The Coxeter transformation}\label{subsec:CoxTrafo}

Let $\Phi$ be a tubular elliptic root system with elliptic root basis $\Gamma= \Gamma(\Phi) = \Gamma_{\text{aff}}\cup \Gamma_{\max^*}$. We fix a Coxeter transformation (see Definition \ref{def:basic}~(f)) $c \in W=W_{\Phi}$ of the form  
 $$c:=  c(\Gamma) = \prod_{\alpha \in \Gamma \setminus (\Gamma_{\max} \cup \Gamma_{\max}^*)} s_{\alpha} \cdot \prod_{\alpha \in \Gamma_{\max}} s_{\alpha}s_{\alpha^*},$$
where the first product is taken in arbitrary order.

\medskip
\noindent
Let $L:= L(\Phi)$.  Notice that by choosing the basis $\Gamma$ the elements $a, b \in R$ such that $L(\{a,b\}) = L \cap R$ are uniquely determined, as $\widetilde{\alpha} = \sum_{i = 1}^n m_i \alpha_i$ and $b = \widetilde{\alpha} + \alpha_0$. We get different Coxeter transformations in $W$ if we change the order of the reflections in the product 
$$\prod_{\alpha \in \Gamma \setminus (\Gamma_{\max} \cup \Gamma_{\max}^*)} s_\alpha.$$  
Are Coxeter transformations conjugate in $W$ or in $O(L)$, where $O(L):= \{ \varphi \in O(V) \mid \varphi(L) = L \}$? Note, as $W$ leaves the bilinear form $(-\mid -)$ invariant,  $W$ is a subgroup of $O(L)$.

\begin{Proposition} \label{prop:CoxTrafoProps}
  Let $c \in W$ be a Coxeter transformation. Then the following hold.

  \noindent
  \begin{itemize}
    \item[(a)] All the Coxeter transformations with respect to a fixed $\Gamma$ are conjugated in $W$;
    \item[(b)] in $O(L)$ the sign change $a \mapsto -a$ maps the conjugacy class of $c$ in $W$ onto the conjugacy class of $c^{-1}$ in $W$;
    \item[(c)]  the elements 
    $c(\Gamma)$, where $\Gamma$ runs through all the elliptic root basis, are conjugated in  $O(L)$;
    \item[(d)] a Coxeter transformation is semi-simple of finite order $\ell_{\max}+1$ where $\ell_{\max}$ is the maximal length of a component of the diagram $\ARSS \setminus{\{\alpha \in \ARSS \mid m_\alpha = m_{\max}\}}$.
  \end{itemize}
\end{Proposition}

\begin{proof}
Except for part $(c)$ this is precisely \cite[Section 9, Lemma A]{Sai85}. Regarding part (c), if $a_1, a_2, b_1, b_2 \in R\setminus{\{0\}}$ such that $L(\{a_k,b_k\}) = L \cap R$ for $k = 1,2$, then the map $\varphi$ that is the identity on $\Gamma_{\text{fin}}:= \{\alpha_1, \ldots , \alpha_n\}$ and that sends $a_1$ and $b_1$ onto $a_2$ and $b_2$, respectively, is an isometry of $(V,(-,-))$ mapping the elliptic basis $\Gamma_1$  related to $a_1$ and $b_1$ onto the elliptic basis $\Gamma_2$ related to $a_2$ and $b_2$. Therefore $\varphi$ is in $O(L)$ and maps a Coxeter transformation $c(\Gamma_1)$ onto a Coxeter transformation $c(\Gamma_2)$ in $W$. Now (c) follows with (a).
\end{proof}

Notice that Proposition~\ref{prop:CoxTrafoProps}~(a) follows also from \cite[Lemma~4.1]{BWY21}.

In a Coxeter group each reflection factorization of a Coxeter is generating, which is a consequence of \cite[Theorem~1.3]{BDSW14}. Coxeter transformations in tubular elliptic root systems behave differently as we show next.

\begin{example}\label{Example:Non-generating}
  Consider a tubular elliptic root system of type $E_6^{(1,1)}$ and let $\{ \alpha_1, \ldots, \alpha_6 \}$ be a simple system for $\FRS$ of type $E_6$. The elliptic root basis in Example \ref{ex:rootbasis} yields a factorization of the Coxeter transformation $c$ which is generating. Consider the roots
   \begin{align*}
  	 \beta_1 & = \alpha_1 +a , ~\beta_2 = \alpha_2 +2a, ~\beta_3 = \alpha_3 + a, ~\beta_4 = \alpha_4 -3a,        \\
  	\beta_5 & = \alpha_5 +2a, ~\beta_6 = \alpha_6 -a, ~\beta_7 = \widetilde{\alpha} + a-b, ~\beta_8 = \alpha_4+a,
  \end{align*}
  then $c = s_{\beta_1}s_{\beta_2}s_{\beta_3}s_{\beta_5}s_{\beta_6}s_{\beta_7}s_{\beta_4} s_{\beta_8}$. An easy calculation yields that this is not a generating factorization.
\end{example}

\begin{remark} \label{rem:problems}
\begin{itemize}
\item[(a)] Although in the tubular cases not each reduced factorization of the Coxeter transformation is generating, this a priori does not imply that $[\idop, c]^{\operatorname{gen}} \subsetneq [\idop, c]$. That is, for each element $w$ with $w \leq c$ there might be at least one generating factorization for $c$ such that $w$ is a prefix of this factorization. The latter is always true in the wild and domestic cases as we show in \cite{BWY21}.

\item[(b)] We just noted that we have to consider generating factorizations in tubular elliptic Weyl groups. We want to emphasize that this makes a difference compared to Coxeter groups. In Coxeter groups, each reflection factorization of a Coxeter element is generating. This fact is due to the Hurwitz transitivity \cite[Theorem~1.3]{BDSW14}. To our knowledge, no conceptual reason is known for this fact. Also note that in finite and affine Coxeter groups, if one reduced reflection factorization for an arbitrary element is generating, then each reflection factorization is generating (see \cite[Theorem 1.2]{BGRW} and \cite[Theorem 1.1]{PW17}). And this fact does not hold for Coxeter transformations in tubular elliptic Weyl groups.
\end{itemize}
\end{remark}

\medskip
We are now able to give a proof of Proposition~\ref{lem:redCoxElt} in the tubular case, that is, we consider a Coxeter transformation $c$ in $W=W_{\Phi}$, where $\Phi$ is a tubular elliptic root system. We show that $c$ cannot be written as a product of less than $|\Gamma|$ reflections, where $\Gamma$ is an elliptic root basis of $\Phi$. 

In the following we  consider the tubular elliptic root systems, which are those of types $D_4^{(1,1)}$ and $E_n^{(1,1)}$ ($n=6,7,8$).

We need some more preparation.

\begin{remark}
In every tubular elliptic root system there exists by Proposition~\ref{RootSystem} two isomorphic affine root subsystems that can be obtained by applying the projections $p_{U}$ and $p_{U'}$ to $\Phi$ where $U=\spanr(a)$ and $U'=\spanr(b)$. The two projections yield isomorphic affine Coxeter groups.

Denote by $s_{i}$ the reflection corresponding to the simple root $\alpha_{i}$ as defined in Example~\ref{ex:rootbasis} and by $s_{t^{*}}$ the reflection corresponding to $\alpha_{t^{*}}$ for $t=2$ in the case $D_{4}$ and $t=4$ otherwise. In the following two proofs we identify the reflection $s_{p_{U}(\alpha_{0})}$ with $s_{0}$ and $s_{p_{U'}(\alpha_{t^{*}})}$ with $s_{t^{*}}$. The reflections $s_{p_{U'}(\alpha_{i})}$ and $s_{p_{U}(\alpha_{i})}$ will be identified with $s_{i}$ for $i\neq 0,t$. The reflections $s_{p_{U}(\alpha_{t^{*}})}$ and $s_{p_{U'}(\alpha_{0})}$  are denoted by $\overline{s_{t^{*}}}$ resp. $\overline{s_{0}}$ in the corresponding affine Weyl groups.

By using \cite[Section 2]{MacD72} we freely switch between the linear and the affine realisation and denote the translation part in the affine expression by $\TR(-)$. The same correspondence is given explicitly in \cite[Proposition 2]{DL11}.
\end{remark}

\begin{Lemma}[{\cite[Lemma~2.1]{BDSW14}}] \label{lem:clCox_red}
Let $(\widetilde{W},S)$ be a Coxeter system with $S=\lbrace s_{1},\ldots,s_{n}\rbrace$, then the product $s_{i_{1}} \dots s_{i_{m}}$ with $s_{i_{j}}\neq s_{i_{k}}$ for $j\neq k$ is reduced in terms of the generating set $\bigcup_{w\in W} w S w^{-1}$.
\end{Lemma}

Let $c_{1}: = p_{U}(c) = \overline{s_{0}}s_{1} \ldots \widehat{s_{t}} \ldots s_{n-1} s_{n}s_{t}s_{t^{*}}$ and $c_{2}:= p_{U'}(c) =s_{0}s_{1}\dots \widehat{s_{t}}\dots s_{n-1} s_{n}$ be the projections of $c=s_{0}s_{1}s_{3}s_{4}\ldots s_{n}s_{t}s_{t^{*}}$ to the underlying affine Coxeter groups.

\begin{Lemma}\label{lem:red_Proj}
The length of $c_{1}$ as well as of $c_{2}$ is $n$ in their corresponding affine Coxeter groups with respect to the respective set of reflections.
\end{Lemma}

\begin{proof}
The element $c_{2}=s_{0}s_{1}\dots \widehat{s_{t}} \ldots s_{n-1} s_{n}$ is a parabolic Coxeter element in an affine Weyl group and by Lemma \ref{lem:clCox_red} its factorization is reduced.\\
Consider the element $c_{1}=\overline{s_{0}}s_{1} \ldots \widehat{s_{t}} \ldots s_{n-1} s_{n}s_{t}s_{t^{*}}$ in the canonical underlying affine Weyl group of $W_{\Phi}$. By Carter's Lemma \cite[Lemma 3]{Car72} and \cite[Plate I]{Bou02} the factorization of $\overline{c}:=\overline{s_{0}}s_{1}\ldots \widehat{s_{t}} \ldots s_{n-1}s_{n}$ is reduced in the finite Weyl group $W_{\FRS}$ with corresponding root system type $\FRS$ and the space Move$(\overline{c})$, defined in \cite[Definition 1.13]{LMPS19}, is equal to $\spanr(\FRS)$. Therefore the element 
\[c_{1}=\overline{s_{0}}s_{1}\dots \widehat{s_{t}} \ldots s_{n-1} s_{n}s_{t}s_{t^{*}}=\TR(\overline{s_{0}}s_{1}\dots \widehat{s_{t}} \ldots s_{n-1} s_{n}(-\alpha_{t})) \overline{s_{0}}s_{1}\ldots \widehat{s_{t}} \ldots s_{n-1}s_{n}\]
is elliptic (see \cite{LMPS19}). Then \cite[Proposition 1.34]{LMPS19} yields the assertion.
\end{proof}

\begin{proof}[\textbf{Proof of Proposition~\ref{lem:redCoxElt} (the tubular case)}]
We show that the factorization of the fixed Coxeter element $c =s_{0}s_{1}\dots \widehat{s_{t}}\dots s_{n-1} s_{n} s_{t}s_{t^{*}}$ is reduced. Therefore assume that $\ell_{T}(c)<n+2$. By Lemma \ref{lem:red_Proj} it follows that the length of $c$ is at least $n$. Since the parities of the lengths of $c$ in $W$ and of $p_U(c_{2})$ in the corresponding affine Weyl group are equal, we get $\ell_T(c) = n$.

Let $c=s_{\beta_{1},\ell_{1},k_{1}} \dots s_{\beta_{n},\ell_{n},k_{n}}$ be a $T$-reduced factorization, where the reflections corresponding to the roots $\beta_{i}+\ell_{i}a+k_{i}b$ are denoted by $s_{\beta_{i},\ell_{i},k_{i}}$ with $\beta_{i} \in \FRS$ and $\ell_{i},k_{i}\in \ZZ$ for $1\leq i \leq n$. After applying the projections $p_{U'}$ and $p_{U}$ to $c$ it holds, respectively,
\begin{align*}
c_{1}=s_{\beta_{1},\ell_{1}} \dots s_{\beta_{n},\ell_{n}}, \qquad c_{2}=s_{\beta_{1},k_{1}} \dots s_{\beta_{n},k_{n}}.
\end{align*}
By Lemma \ref{lem:red_Proj} the element $c_{2}$ is a parabolic Coxeter element and by \cite[Theorem 1.3]{BDSW14} the Hurwitz action is transitive on the reduced factorizations of $c_{2}$, i.e. there exists a braid $\sigma \in B_{n}$ with $\sigma(s_{\beta_{1},k_{1}}, \ldots, s_{\beta_{n},k_{n}})=(s_{0},\ldots,\widehat{s_{t}},\ldots, s_{n-1} ,s_{n})$. Applying the same braid to the factorization $c_{1}=s_{\beta_{1},\ell_{1}} \dots s_{\beta_{n},\ell_{n}}$ we get
$$
c_{1}=s_{\alpha_{0},q_{0}} s_{\alpha_{1},q_{1}} \dots \widehat{s_{\alpha_{t},q_{t}}} \dots s_{\alpha_{n-1},q_{n-1}}s_{\alpha_{n},q_{n}}
$$
for $q_{i} \in \ZZ$. By the previous argumentation and the proof of Lemma \ref{lem:red_Proj} we have
\begin{align*}
\TR(\overline{s_{0}}s_{1}\cdots \widehat{s_{t}} \cdots  s_{n}(-\alpha_{t})) \overline{s_{0}}s_{1}\cdots \widehat{s_{t}} \cdots s_{n} = c_{1} = s_{\alpha_{0},q_{0}} s_{\alpha_{1},q_{1}} \cdots \widehat{s_{\alpha_{t},q_{t}}} \dots  s_{\alpha_{n},q_{n}},
\end{align*}
where the right hand side is considered as the affine realisation of $c_{1}$. Then it holds
\begin{align*}
\overline{s_{0}}s_{1} \cdots \widehat{s_{t}}\cdots s_{n}(-\alpha_{t})=c_{1}(0)
= x-\alpha_t
\end{align*}
for some $x \in L_1:=L( \alpha_{0},\alpha_{1},\dots,\widehat{\alpha_{t}},\dots,\alpha_{n-1},\alpha_{n})$. In particular $\alpha_{t}\in L_1$.

Figure \ref{fig:Dynkin} describes the unique coefficients that are needed to express the highest root of a finite crystallographic root system in terms of the simple roots $\alpha_{1},\dots, \alpha_{n}$.

In our notation $\alpha_{t}$ corresponds to the vertex that is labeled by a red dot in the extended diagrams of $D_{4}$ and $E_{i}$ $(i=6,7,8)$. The corresponding coefficients are $2$, $3$, $4$ and $6$, respectively. The unique $\ZZ$-linear combination of $\alpha_{0}$ in terms of $\alpha_{1},\alpha_{2},\dots,\alpha_{n}$ contains a coefficient which is not divisible by $2$, $3$, $4$ and $6$, respectively. This implies that $\alpha_{t} \not\in L_1$, which is a contradiction to our observation above. The latter implies that $\ell_T(w) = n+2$. By Lemma~\ref{prop:CoxTrafoProps} (a) the conjugacy class of $c$ does not depend on the order of the first $n$ factors of  the factorization of a Coxeter transformation. Since the length function is invariant under conjugacy the assertion follows.
\end{proof}

\subsection{The elliptic root systems for $\COH(\XX)$}\label{subsec:TubEllRootSystem}
In this section we describe the root quadruple $(K_{0}(\XX),\chi^{s},\Delta_{\text{re}}(\XX),c)$ attached to the category of coherent sheaves over a weighted projective line $\XX$ of tubular type. By definition a weighted projective line $\XX$ is tubular if the so-called Euler characterstic is zero (see \cite[Section 6, 10.5]{HTT07} ). This is exactly the case if the radical of the symmetrized Euler form attached to the category $\COH(\XX)$ has rank two (see \cite[Proposition 18.8]{Len99}). Thus \cite[Section 6, Proposition 10.12]{HTT07} yields that $\XX$ is a projective line with weight sequence $(2,2,2,2),~(3,3,3),~(2,2,4)$ or $(2,3,6)$. Therefore for $\XX$ tubular the set $\Delta_{\text{re}}(\XX)$ can be identified with the extended root system described in Section \ref{sec:root_star_quiver}, where the weight sequences are precisley the previous four sequences. 

\begin{corollary} \label{prop:TubEllRoot}
Let $\XX$ be a weighted projective line of tubular type, and let $(K_0(\XX), \chi^{s}, \Delta_{\text{re}}(\XX), c)$ be the associated extended root system. Then the following holds.
\begin{itemize}
\item[(a)]  $\Delta_{\text{re}}(\XX)$ is a tubular elliptic root system in $K_0(\XX)~\otimes_{\ZZ}~\RR$ with respect to the bilinear form $\chi^{s}$;
\item[(b)] the type of $\Delta_{\text{re}}(\XX)$ is $D_4^{(1,1)}, E_6^{(1,1)},~E_7^{(1,1)}$ or $E_8^{(1,1)}$, respectively, (see the diagrams in Figure~\ref{fig:EllipticDynkin});
\item[(c)] $c$ is a Coxeter transformation in $W_{\Delta_{\text{re}}(\XX)}$.
\end{itemize}
\end{corollary}

\begin{proof}
Consider the following weight sequences that characterize the tubular types  \[p\in \lbrace (2,2,2,2),(3,3,3),(2,2,4),(2,3,6)\rbrace\] and a normalized  $\lambda$. By Propositions~\ref{lem:root_system} and  \ref{thm:SaitoClassification} the root system corresponding to $\COH(\XX)$ with $\XX$ of tubular type is of type $D^{(1,1)}_{4},E^{(1,1)}_{6},E^{(1,1)}_{7}$ or $E^{(1,1)}_{8}$, respectivley.
\end{proof}

\begin{remark}
  Barot and Geiss \cite{BG12} studied the cluster algebras given by a matrix whose associated quiver is given by one of the quivers in Figure \ref{fig:EllipticDynkinQuiver} (note that these quivers are mutation finite). They categorified the corresponding cluster algebras and described the (real) Schur roots for $\mathbb{X}$ of tubular type.
\end{remark}

 \begin{figure}
    \centering
    \begin{tikzpicture}[scale=1.9]

      \node (03) at (0.5,0.5) [] {$D_4^{(1,1)}$};
      \node (G33) at (2,1.15) [] {\tiny{$1^*$}};
       \node (G333) at (2,1) [circle, fill=white, inner sep=0pt, minimum width=9pt]{};
      \node (G3) at (2,1) [circle, draw, fill=black!50, inner sep=0pt, minimum width=4pt]{};
      \node (A33) at (1.2,0.5) [] {\tiny{$(1,1)$}};
      \node (A3) at (1.5,0.5) [circle, draw, fill=black!50, inner sep=0pt, minimum width=4pt]{};
      \node (B33) at (2.8,0.5) [] {\tiny{$(4,1)$}};
      \node (B3) at (2.5,0.5) [circle, draw, fill=black!50, inner sep=0pt, minimum width=4pt]{};
      \node (C3) at (2,0.5) [circle, draw, fill=black!50, inner sep=0pt, minimum width=4pt]{};
      \node (C33) at (2,0.65) [] {\tiny{$1$}};
      \node (E33) at (1.3,0.1) [] {\tiny{$(2,1)$}};
      \node (E3) at (1.6, 0.1) [circle, draw, fill=black!50, inner sep=0pt, minimum width=4pt]{};
      \node (F33) at (2.7,0.1) [] {\tiny{$(3,1)$}};
      \node (F3) at (2.4, 0.1) [circle, draw, fill=black!50, inner sep=0pt, minimum width=4pt]{};

      \node (04) at (0.5,-1) [] {$E_6^{(1,1)}$};
      \node (A44) at (1,-1.15) [] {\tiny{$(1,2)$}};
      \node (A4) at (1,-1) [circle, draw, fill=black!50, inner sep=0pt, minimum width=4pt] {};
      \node (B44) at (1.5,-1.15) [] {\tiny{$(1,1)$}};
      \node (B4) at (1.5,-1) [circle, draw, fill=black!50, inner sep=0pt, minimum width=4pt]{};
      \node (C44) at (2,-0.85) [] {\tiny{$1$}};
      \node (C4) at (2,-1) [circle, draw, fill=black!50, inner sep=0pt, minimum width=4pt]{};
      \node (D44) at (2.5,-1.15) [] {\tiny{$(3,1)$}};
      \node (D4) at (2.5,-1) [circle, draw, fill=black!50, inner sep=0pt, minimum width=4pt]{};
      \node (E44) at (3,-1.15) [] {\tiny{$(3,2)$}};
      \node (E4) at (3,-1) [circle, draw, fill=black!50, inner sep=0pt, minimum width=4pt]{};
      \node (F44) at (2.7,-1.4) [] {\tiny{$(2,1)$}};
      \node (F4) at (2.4,-1.4) [circle, draw, fill=black!50, inner sep=0pt, minimum width=4pt]{};
      \node (G44) at (3.1,-1.8) [] {\tiny{$(2,2)$}};
      \node (G4) at (2.8,-1.8) [circle, draw, fill=black!50, inner sep=0pt, minimum width=4pt]{};
      \node (H44) at (2,-0.35) [] {\tiny{$1^*$}};
      \node (H444) at (2,-0.5) [circle, fill=white, inner sep=0pt, minimum width=9pt]{};
      \node (H4) at (2,-0.5) [circle, draw, fill=black!50, inner sep=0pt, minimum width=4pt]{};

      \node (05) at (4,0.5) [] {$E_7^{(1,1)}$};
      \node (A5) at (4.5,0.5) [circle, draw, fill=black!50, inner sep=0pt, minimum width=4pt] {};
      \node (A55) at (4.5,0.65) [] {\tiny{$(1,3)$}};
      \node (B5) at (5,0.5) [circle, draw, fill=black!50, inner sep=0pt, minimum width=4pt]{};
      \node (B55) at (5,0.65) [] {\tiny{$(1,2)$}};
      \node (C5) at (5.5,0.5) [circle, draw, fill=black!50, inner sep=0pt, minimum width=4pt]{};
      \node (C55) at (5.5,0.65) [] {\tiny{$(1,1)$}};
      \node (D55) at (6,0.65) [] {\tiny{$1$}};
      \node (D5) at (6,0.5) [circle, draw, fill=black!50, inner sep=0pt, minimum width=4pt]{};

      \node (E5) at (6.5,0.5) [circle, draw, fill=black!50, inner sep=0pt, minimum width=4pt]{};
      \node (E55) at (6.5,0.65) [] {\tiny{$(3,1)$}};
      \node (H55) at (6.7,0.2) [] {\tiny{$(2,1)$}};
      \node (H5) at (6.4,0.2) [circle, draw, fill=black!50, inner sep=0pt, minimum width=4pt]{};
      \node (F5) at (7,0.5) [circle, draw, fill=black!50, inner sep=0pt, minimum width=4pt]{};
      \node (F55) at (7,0.65) [] {\tiny{$(3,2)$}};
      \node (G5) at (7.5,0.5) [circle, draw, fill=black!50, inner sep=0pt, minimum width=4pt]{};
      \node (G55) at (7.5,0.65) [] {\tiny{$(3,3)$}};
      \node (I55) at (6,1.15) [] {\tiny{$1*$}};
      \node (I555) at (6,1) [circle, fill=white, inner sep=0pt, minimum width=9pt]{};
      \node (I5) at (6,1) [circle, draw, fill=black!50, inner sep=0pt, minimum width=4pt]{};

      \node (06) at (4,-1) [] {$E_8^{(1,1)}$};
      \node (A66) at (4.5,-1.15) [] {\tiny{$(1,2)$}};
      \node (A6) at (4.5,-1) [circle, draw, fill=black!50, inner sep=0pt, minimum width=4pt] {};
      \node (B66) at (5,-1.15) [] {\tiny{$(1,1)$}};
      \node (B6) at (5,-1) [circle, draw, fill=black!50, inner sep=0pt, minimum width=4pt]{};
      \node (C66) at (5.5,-0.85) [] {\tiny{$1$}};
      \node (C6) at (5.5,-1) [circle, draw, fill=black!50, inner sep=0pt, minimum width=4pt]{};
      \node (D66) at (6,-1.15) [] {\tiny{$(3,1)$}};
      \node (D6) at (6,-1) [circle, draw, fill=black!50, inner sep=0pt, minimum width=4pt]{};
      \node (E66) at (6.5,-1.15) [] {\tiny{$(3,2)$}};
      \node (E6) at (6.5,-1) [circle, draw, fill=black!50, inner sep=0pt, minimum width=4pt]{};
      \node (F66) at (7,-1.15) [] {\tiny{$(3,3)$}};
      \node (F6) at (7,-1) [circle, draw, fill=black!50, inner sep=0pt, minimum width=4pt]{};
      \node (G66) at (7.5,-1.15) [] {\tiny{$(3,4)$}};
      \node (G6) at (7.5,-1) [circle, draw, fill=black!50, inner sep=0pt, minimum width=4pt]{};
      \node (H66) at (8,-1.15) [] {\tiny{$(3,5)$}};
      \node (H6) at (8,-1) [circle, draw, fill=black!50, inner sep=0pt, minimum width=4pt]{};
      \node (I666) at (5.5,-0.5) [circle, fill=white, inner sep=0pt, minimum width=9pt]{};
      \node (I66) at (5.5,-0.35) [] {\tiny{$1$}};
      \node (I6) at (5.5,-0.5) [circle, draw, fill=black!50, inner sep=0pt, minimum width=4pt]{};
      \node (J66) at (6.2,-1.4) [] {\tiny{$(2,1)$}};
      \node (J6) at (5.9,-1.4) [circle, draw, fill=black!50, inner sep=0pt, minimum width=4pt]{};

      \draw[<-] (A3) to (C3);
      \draw[->] (C3) to (B3);
      \draw[->] (C3) to (E3);
      \draw[->] (C3) to (F3);
      \draw[<-] (G3) to (A3);
      \draw[<-] (G3) to (B3);
      \draw[<-] (G333) to (E3);
      \draw[<-] (G333) to (F3);

      \draw[<-] (A4) to (B4);
      \draw[<-] (B4) to (C4);
      \draw[->] (C4) to (D4);
      \draw[->] (D4) to (E4);
      \draw[->] (C4) to (F4);
      \draw[->] (F4) to (G4);
      \draw[->] (B4) to (H4);
      \draw[->] (D4) to (H4);
      \draw[->] (F4) to (H444);

      \draw[<-] (A5) to (B5);
      \draw[<-] (B5) to (C5);
      \draw[->] (D5) to (C5);
      \draw[->] (D5) to (E5);
      \draw[->] (E5) to (F5);
      \draw[->] (F5) to (G5);
      \draw[->] (D5) to (H5);
      \draw[<-] (I5) to (C5);
      \draw[<-] (I5) to (E5);
      \draw[<-] (I555) to (H5);

      \draw[<-] (A6) to (B6);
      \draw[<-] (B6) to (C6);
      \draw[->] (C6) to (D6);
      \draw[->] (D6) to (E6);
      \draw[->] (E6) to (F6);
      \draw[->] (F6) to (G6);
      \draw[->] (G6) to (H6);
      \draw[<-] (I6) to (B6);
      \draw[<-] (I6) to (D6);
      \draw[<-] (I666) to (J6);
      \draw[->] (C6) to (J6);

    \end{tikzpicture}
    \caption{Elliptic Dynkin quivers of tubular type $\phantom{000000000000000}$} \label{fig:EllipticDynkinQuiver}
  \end{figure}

\section{Hurwitz transitivity in tubular elliptic Weyl groups} \label{sec:HurwitzElliptic}

In this section we prove Theorem~\ref{thm:MainElliptic}. Therefore, we assume that $W = W_\Phi$ is a tubular elliptic Weyl group, that is $\Phi$ is of type  $X_n^{(1,1)}$ where  $X_n^{(1,1)}$ is $D^{(1,1)}_{4},E^{(1,1)}_{6},E^{(1,1)}_{7}$ or $E^{(1,1)}_{8}$ (see Corollary~\ref{prop:TubEllRoot}). According to Example~\ref{ex:rootbasis} we fix the Coxeter transformation
\[c=s_{1}s_{3}s_{4}s_{0}s_{2}s_{2^{*}}\]
in the case $\Phi= D^{(1,1)}_{4}$ and 
\[c=s_{1}\cdots \widehat{s_{4}} \cdots s_{n}s_{0}s_{4}s_{4^{*}}\]
in the cases $\Phi=E^{(1,1)}_{n}$ $(n=6,7,8)$, where $s_i = s_{\alpha_i}$ for $i \geq 0$ and $s_{j^{*}}= s_{\alpha_j^*}$. Let $(t_{1},\ldots,t_{n+2})\in \Redd(c)$ be a factorization of $c$, so  $W = \langle t_{1},\ldots,t_{n+2} \rangle$, and for $1\leq i \leq n+2$ let $$k_{i},\ell_{i}\in \ZZ~\mbox{and}~\beta_{i}\in \PFRS:=\text{span}_{\ZZ_{\geq 0}}(\alpha_{1},\ldots,\alpha_{n})\cap \FRS$$
such that $t_{i}=s_{\beta_{i}+k_{i}a+\ell_{i}b}.$

The strategy of the proof is first to show that we may assume up to Hurwitz action that
$$(\beta_1, \ldots , \beta_{n+2}) = (\alpha_1, \ldots, \widehat{\alpha_t}, \ldots , \alpha_n, - \widetilde{\alpha}, \alpha_t, \alpha_t), $$ and that $\ell_i = 0$ for $i \neq n$  where $t=2$ for $\Phi=D^{(1,1)}_{4}$ and $t=4$ for $\Phi=E^{(1,1)}_{n}$ (see \cite[Corollary~1.4]{LR16}). 

The rest of the proof is then a study of the radical parts of the roots related to the reflections $t_i$ for $1 \leq i \leq n+2$. Recall that $U$ is the subspace of $V$, which is generated by $a$ (see Section~\ref{subsec:Basis}). In this section we set $\overline{V}:=p_U(V) = V/U$. The action of $W$ on $\overline{V}$ induces the canonical group homomorphism $W \rightarrow \Aut(\overline{V})$. Let $\overline{W}$ be the image of this map. Furthermore, we put
$$\overline{\Phi} := \ARS = p_U(\Phi) = \FRS + \ZZ b.$$

We denote by $\overline{v}$ and by $\overline{w}$ the images of $v$ in $\overline{V}$ and $w$ in $\overline{W}$, respectively. The reflection $s_\alpha$ acts as $s_{\overline{\alpha}}$ on $\overline{V}$  for $\alpha \in \Phi$. Therefore, $\alpha_i$ is a representant of the coset $\overline{\alpha}_i$ for $0 \leq i \leq n$, and $s_i$ is a representant of $\overline{s}_i$ for $0 \leq i \leq n$, and $\overline{s}_{t^*} =\overline{s}_t$. Thus $\overline{W} = \langle \overline{s}_i, \overline{s}_{t^*}~|~0 \leq i \leq n \rangle$ is isomorphic to the affine Coxeter group $W_{\text{aff}}:=\langle s_{0},s_{1},\ldots,s_{n}\rangle$ with simple system $S_{\text{aff}}:=\lbrace s_{0},s_{1},\ldots,s_{n} \rbrace$ and root system $\ARS$. In particular our notation implies  $\overline{t}_i = s_{\beta_{i}+\ell_{i}b}$ for $1\leq i \leq n+2$.

It follows that 
$$\overline{c}=\overline{t_{1}}\cdots \overline{t_{n+2}}=s_{1}\cdots \widehat{s_{t}}\cdots s_{n}s_{0}.$$ 
Thus  $\overline{c}$ is a standard parabolic Coxeter element in $(\overline{W},\overline{S})$. Therefore $\ell_{\overline{S}}(\overline{c})=n=\ell_{\overline{T}}(\overline{c})$, where $\overline{T}$ the set of reflections in $\overline{W}$, and $\ell_{\overline{S}}$ and $\ell_{\overline{T}}$ are the length functions attached to the generating sets $\overline{S}$ and $\overline{T}$ of $\overline{W}$, respectively (see \cite[Lemma 2.1]{BDSW14}).

By \cite[Lemma 2.3]{WY19} there exists a braid $\tau_{1}\in \mathcal{B}_{n+2}$ such that 
\[\tau_{1}(\overline{t_{1}},\ldots,\overline{t_{n+2}})=(\overline{r_{1}},\ldots,\overline{r_{n}},\overline{r_{n+1}},\overline{r_{n+1}}),\]
where $\overline{r_{i}}\in \overline{T}$ for  $1\leq i \leq n+1$. Then $\overline{c} = \overline{r_{1}}\cdots\overline{r_{n}}$ is a reduced reflection factorization of the standard parabolic Coxeter element $\overline{c}$ in $\overline{W}$. By \cite[Theorem 1.3]{BDSW14} there exists a braid $\tau_{2}\in \mathcal{B}_{n} $, which can be read as an element of $ \mathcal{B}_{n+2}$ such that 
\[\tau_{2}(\overline{r_{1}},\ldots,\overline{r_{n}},\overline{r_{n+1}},\overline{r_{n+1}})=(s_{1},\ldots,\widehat{s_{t}},\ldots, s_{n}, s_{0},\overline{r_{n+1}},\overline{r_{n+1}}).\]

Let $\beta\in \overline{\Phi}^+$ and $\ell\in \ZZ$ such that $\overline{r_{n+1}}=s_{\beta+\ell b}$, and recall that $s_0 = s_{\alpha_0} = s_{-\widetilde{\alpha}+b}= s_{\widetilde{\alpha}-b}$, where  $\widetilde{\alpha}=\sum_{i=1}^{n}m_{i}\alpha_{i}$ is the highest root in $\FRS \subseteq \overline{\Phi}$. In particular, there are $k_1, \ldots ,k_n, k', k'',\widetilde{k} \in \ZZ$ such that
\begin{align} \label{equ:Tau12}
\tau_{2}\tau_{1}(t_{1},\ldots,t_{n+2})=(s_{\alpha_{1}+k_{1}a},\ldots,\widehat{s_{\alpha_{t}+k_{t}a}},\ldots,
s_{\alpha_{n}+k_{n}a},s_{\widetilde{\alpha}-\widetilde{k}a-b},s_{\beta+ k'a+\ell b},s_{\beta+ k'' a+\ell b}).
\end{align}
Next we investigate the roots appearing in  (\ref{equ:Tau12}).

\begin{Lemma} \label{lem:lambda_t}
Let  $c = u_1 \ldots u_{n+2}$, where $u_i \in T$,  is a factorization of $c$ as given in  (\ref{equ:Tau12}). Further let $\lambda_{j} \in \ZZ_{\geq 0}$ such that $\beta=\sum_{i=1}^{n}\lambda_{i}\alpha_{i}\in \PFRS$. 
Then the following hold
\begin{itemize}
\item[(a)]  $(\lambda_{t}, \ell)  \in \lbrace (1, 0) ,(m_{t}-1, -1) \rbrace$.
\item[(b)] $|k'' - k'| = 1$.
\end{itemize}
\end{Lemma}

\begin{proof}
The reflections $t_1, \ldots , t_{n+2}$ generate $W$ by assumption and the Hurwitz action preserves this property. In particular, the tuple obtained by applying the braid $\tau_2\tau_1$ to the tuple $(t_1, \ldots , t_{n+2})$ generates $W$. Therefore equation (\ref{equ:Tau12}) and Lemma~\ref{lem:GenLatticeGroup} yield
\begin{align*}
L(\Phi)&=\spanz(\alpha_{1}+k_{1}a,\ldots,\widehat{\alpha_{t}+k_{t}a},\ldots,\alpha_{n}+k_{n}a,\widetilde{\alpha}-\widetilde{k}a-b,\beta+ k'a+\ell b,\beta+k'' a+\ell b) \\
&=\spanz(\alpha_{1}+k_{1}a,\ldots,\widehat{\alpha_{t}+k_{t}a},\ldots,\alpha_{n}+k_{n}a,-\widetilde{\alpha}+\widetilde{k}a +b,\beta+k' a+ \ell b,(k''- k')a).
\end{align*}

Note that  $\alpha_1, \ldots , \hat{\alpha_t}, \ldots, \alpha_n, \tilde{\alpha}$ generate a sublattice of $L(\FRS)$, which does not contain $\alpha_t$,  as  the following  holds
by \cite{Bou02}:
$$\begin{array}{|l|l|l|l|l|}
	\hline
	\FRS & D_4 & E_6 & E_7 & E_8\\ \hline
	m_t & 2 & 3& 4& 6\\
	\hline
\end{array}$$
Therefore, $\alpha_t \in L(\Phi)$ implies that $\lambda_t  >0$.

We first show (a), that is $(\lambda_{t}, \ell)\in \lbrace (1, 0) ,(m_{t}-1, -1) \rbrace$. We calculate in $L(\overline{\Phi})$.  There are $\mu_i, \mu, \mu' \in \ZZ$ with $1 \leq i \leq n$, such that 
\[\alpha_{t}=\sum_{i=1,~i\neq t}^{n}\mu_{i}\alpha_{i}+\mu(-\widetilde{\alpha}+b)+\mu'(\beta+\ell b),\]
and therefore
$$0=\mu + \mu' \ell~\mbox{and}~ 1=-m_{t}\mu + \lambda_{t}\mu',$$
which yields 
$$1 = m_t \mu' \ell + \lambda_{t}\mu' = \mu'(m_t \ell +\lambda_t).$$
Hence $\mu', m_t \ell +\lambda_t \in \lbrace \pm 1 \rbrace$. As $1\leq \lambda_{t} \leq m_{t}$ (see \cite{Bou02}), we get $(\ell, \lambda_t) = (0,1)$ if $m_t \ell +\lambda_t = 1$ and $(\ell, \lambda_t) = (-1, m_{t}-1)$ if $m_t \ell +\lambda_t = -1$, which proves (a).

Now we prove assertion (b). Let  $\mu_i, \mu, \mu', x \in \ZZ$ with $1 \leq i \leq n$, such that 
$$a = \sum_{i=1,~i\neq t}^{n}
\mu_{i}(\alpha_{i} + k_i  a) + \mu(-\widetilde{\alpha} + \widetilde{k}a + b)+\mu'(\beta+ k' a+ \ell b) + x (k''-k')a.$$
If $\ell = 0$, then $\mu = 0$.  Since $\lambda_t >0$, it follows   $\mu' = 0$, as well. As $\alpha_1, \ldots, \alpha_n, a$ are linearly independent we get as a consequence  $\mu_i = 0$ for all $i$ and  $|x||k'' - k'| = 1$, which is   assertion (b).
\end{proof}

\medskip
\begin{Lemma}\label{lem:beta=alphat}
The reflection $s_{\beta}$ is  conjugated to $s_{t}$ by an element in $H:=\langle s_{\widetilde{\alpha}}, s_{1},\ldots,\widehat{s_{t}},\ldots,s_{n} \rangle$.
\end{Lemma}

\begin{proof}
By Lemma~\ref{lem:lambda_t} we have  $\beta=\sum_{i=1}^{n}\lambda_{i}\alpha_{i}$ where $\lambda_{t}$ is in $ \lbrace 1, m_t-1 \rbrace$.

Assume $\lambda_t > 1$. Then $\lambda_t = m_t - 1$ and $\FRS$ is of type $E_n$  for some $n \in \{6,7,8\}$. According to \cite[Plates V- VI]{Bou02} the longest root $\widetilde{\alpha}$ is perpendicular to $\alpha_1, \ldots , \alpha_n$ beside one  $\alpha_j$. Moreover there all the roots $\beta$ are listed that have the property $\lambda_t = m_t-1$. It is easily checked that for all these roots $\lambda_j= 1$. Therefore it follows that $s_{\widetilde{\alpha}}(\beta) = \sum_{i=1}^{n}\mu_{i}\alpha_{i}$ with $\mu_i \in \ZZ$ and $\mu_t = -1$.
 
As $s_\gamma = s_{-\gamma}$ for every root $\gamma \in \Phi$, we may assume $\lambda_t = 1$. Then \cite[Lemma~5.2 (b)]{BWY21} (see also \cite{Bou02})  yields that $\lambda_{r}\in \lbrace 0,1\rbrace$ for all $1\leq r \leq n$, which implies that $s_{\beta}$ is conjugated to $s_{t}$ by an element in $H$, see  also \cite[Corollary~5.3]{BWY21}
\end{proof}

\begin{remark} We exhibit this for an example: For the root $\beta=\alpha_{1}+\alpha_{3}+\alpha_{4}+\alpha_{5}+\alpha_{6}$ in a root system of type $E_7$, we have $w(\beta)=\alpha_{t}$, where $w:=s_{3}s_{1}s_{5}s_{6}\in H$.
\end{remark}

\medskip
By Lemma~\ref{lem:beta=alphat} and Lemma \ref{lem:BraidConjugation} we obtain that there exists a braid $\tau_{3}\in \mathcal{B}_{n+2}$ such that 
\begin{align*} 
\tau_{3}&(s_{\alpha_{1}+k_{1}a},\ldots,\widehat{s_{\alpha_{t}+k_{t}a}},\ldots, s_{\alpha_{n}+k_{n}a},s_{\widetilde{\alpha}-\widetilde{k}a -b},s_{\beta+ k'a +\ell b},s_{\beta+ k''a+ \ell b})=\\ 
&(s_{\alpha_{1}+k'_{1}a},\ldots,\widehat{s_{\alpha_{t}+k'_{t}a}},\ldots, s_{\alpha_{n}+k'_{n}a},s_{\widetilde{\alpha}-\widetilde{k}'a-b},s_{\alpha_{t}
	+\overline{k'}a+\overline{\ell }b } ,s_{\alpha_{t}  + \overline{k''}a+\overline{\ell}b}),
\end{align*}
where $k'_1, \ldots k'_n, \overline{k'},\widetilde{k}', \overline{k''},\overline{\ell }\in \ZZ$.

By applying Lemma~\ref{lem:lambda_t} (a) to the new factorization of $c$ we conclude that  $|\overline{k''} - \overline{k'}| = 1$. Part (b) of the same lemma gives that we have
\begin{itemize}
\item either $\overline{\ell} = 0$, 
\item or $\overline{\ell} = -1$ and $1 = \lambda_t = m_t -1$ and $\FRS$ is of type $D_4$ (as in this case $m_t = 2$).
\end{itemize}

Next we introduce a notation, which is inspired by Kluitmann \cite[Section 3.1]{Klu87}, but slightly different. Recall that we have  
by Theorem \ref{cor:SplitEllipticWeylGroup} 
$$
W\cong W_{\FRS}\ltimes (\spanz(a,b) \otimes_{\ZZ} L(\FRS)).
$$
We will write the element $(w,a\otimes x_{1} +b \otimes x_{2})\in  W_{\FRS}\ltimes (\spanz(a,b) \otimes_{\ZZ} L(\FRS))$ in vector notation $\begin{bmatrix} w \\ x_{1} \\ x_{2}\end{bmatrix}$, where $ w \in W_{\FRS}, x_1, x_2 \in L(\FRS)$. Then the reflection with respect to the root $\alpha = \alpha_f + ka + \ell b$, where $\alpha_f \in \Phi_{\fin}$ and $k,\ell \in \ZZ$, is in this notation
$$\begin{bmatrix}  s_{\alpha_f} \\ k \alpha_f \\ \ell \alpha_f
\end{bmatrix},$$
see Lemma~\ref{Eichler} (h). By Remark~\ref{rem:GroupOperation}  the group operation is  given by 
\begin{align} \label{equ:DefMulti}
\begin{bmatrix} w_1 \\ x_{1} \\ x_{2}\end{bmatrix} \cdot \begin{bmatrix} w_2 \\ y_{1} \\ y_{2}\end{bmatrix} = 
\begin{bmatrix} w_1 w_2 \\ w_2^{-1}(x_1)+y_1 \\ w_2^{-1}(x_2)+y_2\end{bmatrix},
\end{align}
while conjugation is given by
\begin{align} \label{equ:DefConj}
\begin{bmatrix} w_2 \\ y_{1} \\ y_{2}\end{bmatrix} \cdot \begin{bmatrix} w_1 \\ x_{1} \\ x_{2}\end{bmatrix} \cdot \begin{bmatrix} w_2 \\ y_{1} \\ y_{2}\end{bmatrix}^{-1} = 
\begin{bmatrix} w_2 w_1 w_2^{-1} \\ w_2(x_1)+ (\idop - w_2w_1^{-1}w_2^{-1})(y_1) \\ w_2(x_2)+(\idop-w_2w_1^{-1}w_2^{-1})(y_2) \end{bmatrix},
\end{align}

\begin{Lemma}\label{lem:Differenz_a}
Let $k_1,k_2 \in \ZZ$ such that $|k_1-k_2| = 1$. Then there is a Hurwitz move $\tau$ such that 
\begin{align*}
\tau \left( \begin{bmatrix} s_{t} \\ k_1 \alpha_{t} \\ \ell \alpha_t\end{bmatrix} , 
\begin{bmatrix} s_{t} \\ k_2 \alpha_{t} \\ \ell \alpha_t\end{bmatrix} \right) =
\left( \begin{bmatrix} s_{t} \\ 0 \\ \ell \alpha_t\end{bmatrix} , 
\begin{bmatrix} s_{t} \\ (k_2 - k_1) \alpha_{t} \\ \ell \alpha_t\end{bmatrix} \right).
\end{align*}
\end{Lemma}

\begin{proof} A direct calculation shows for $m\geq 1$ that 
\begin{align} \label{equ:HAmove1}
\sigma_1^m \left( \begin{bmatrix} s_{t} \\ 0 \\ \ell \alpha_t\end{bmatrix} , 
\begin{bmatrix} s_{t} \\ \alpha_{t} \\ \ell \alpha_t\end{bmatrix} \right) & =
\left( \begin{bmatrix} s_{t} \\ m \alpha_t \\ \ell \alpha_t\end{bmatrix} , 
\begin{bmatrix} s_{t} \\ (m+1) \alpha_{t} \\ \ell \alpha_t\end{bmatrix} \right)\\ \label{equ:HAmove2}
\sigma_1^{-m} \left( \begin{bmatrix} s_{t} \\ 0 \\ \ell \alpha_t\end{bmatrix} , 
\begin{bmatrix} s_{t} \\ -\alpha_{t} \\ \ell \alpha_t\end{bmatrix} \right) & =
\left( \begin{bmatrix} s_{t} \\ m \alpha_t \\ \ell \alpha_t\end{bmatrix} , 
\begin{bmatrix} s_{t} \\ (m-1) \alpha_{t} \\ \ell \alpha_t\end{bmatrix} \right)
\end{align}
If $k_2-k_1= 1$, that is $k_2= k_1+1$, the assertion follows from (\ref{equ:HAmove1}) by choosing $m = k_1$. If $k_2-k_1= -1$, that is $k_2= k_1-1$, the assertion follows from (\ref{equ:HAmove2}) by choosing $m = k_1$.
\end{proof}

\begin{proof}[\textbf{Proof of Theorem \ref{thm:MainElliptic}}]
By Lemmas~\ref{lem:Differenz_a} and \ref{lem:lambda_t} we may assume that 
\begin{align*}
(t_1, \ldots, t_{n+2}) = (s_{\alpha_{1}+k'_{1}a},\ldots,\widehat{s_{\alpha_{t}+k'_{t}a}},\ldots, s_{\alpha_{n}+k'_{n}a},s_{\widetilde{\alpha}-\widetilde{k}'a-b},s_{\alpha_{t}
	+\overline{\ell }b } ,s_{\alpha_{t}  + \varepsilon a+\overline{\ell}b})
\end{align*} 
for some $\varepsilon \in \{\pm 1\}$. Assume first that $\overline{\ell } = 0$. Let $\Phi=E^{(1,1)}_{6}$. We get 
\begin{align*}
c&=
\begin{bmatrix}
s_{1}s_{2}s_{3}s_{5}s_{6}s_{\widetilde{\alpha}} \\\alpha_{4} \\ -\widetilde{\alpha}
\end{bmatrix}\\
&=\begin{bmatrix} s_{1} \\ k_{1}'\alpha_{1} \\ 0\end{bmatrix} \cdot
\begin{bmatrix} s_{2} \\ k_{2}'\alpha_{2} \\0 \end{bmatrix} \cdot 
\begin{bmatrix} s_{3} \\ k_{3}'\alpha_{3} \\0 \end{bmatrix} \cdot 
\begin{bmatrix} s_{5} \\ k_{5}'\alpha_{5} \\0 \end{bmatrix} \cdot 
\begin{bmatrix} s_{6} \\ k_{6}'\alpha_{6} \\0 \end{bmatrix} \cdot 
\begin{bmatrix} s_{\widetilde{\alpha}} \\ -\widetilde{k}'\widetilde{\alpha} \\ -\widetilde{\alpha}\end{bmatrix} \cdot
\begin{bmatrix} s_{4} \\ 0 \\ 0\end{bmatrix} \cdot 
\begin{bmatrix} s_{4} \\ \varepsilon \alpha_{4} \\ 0\end{bmatrix} \\
&=\begin{bmatrix} s_{1}s_{2}s_{3}s_{5}s_{6}s_{\widetilde{\alpha}} \\ k_{1}'\alpha_{1}+k_{2}'\alpha_{2}+(k_{1}'+k_{3}')\alpha_{3}+k_{5}'\alpha_{5}+(k_{5}'+k_{6}')\alpha_{6}-(k_{2}'+\widetilde{k}')\widetilde{\alpha}+ \varepsilon \alpha_{4} \\ -\widetilde{\alpha}\end{bmatrix}.
\end{align*}

Note that $\widetilde{\alpha}=\alpha_{1}+2\alpha_{2}+2\alpha_{3}+3\alpha_{4}+2\alpha_{5}+\alpha_{6}$. By comparison of coefficients we obtain 
$$
k_1'=k_2'+\widetilde{k}', \quad k_2'=k_1'+k_3'=k_5'=k_5'+k_6'=2k_1'.
$$
We see that $k_6'=0$ as well as $k_1'=k_3'$. Now, comparing the coefficient of $\alpha_4$, we see that 
$$\varepsilon -3k_1' =1.$$
If $\varepsilon = -1$, then $3k_1'=-2$ and hence $k_1'\notin \mathbb{Z}$, which is not possible. This shows  that $\varepsilon = 1$. Then $k_1'=0$ and hence $k_2'=k_3'=k_5'=\widetilde{k}'=0$, and the assertion of the theorem holds in this case.

The proofs of the cases $E_{7}^{(1,1)}$ and $E_{8}^{(1,1)}$ are analogously to the case $\Phi=E^{(1,1)}_{6}$ and are therefore omitted.

Next assume  that  $\Phi=D_{4}^{(1,1)}$. Then  we get
\begin{align*}
c=
\begin{bmatrix}
s_{1}s_{3}s_{4}s_{-\widetilde{\alpha}} \\ \alpha_{2} \\ -\widetilde{\alpha}
\end{bmatrix}
&=\begin{bmatrix} s_{1} \\ k_{1}'\alpha_{1} \\ 0\end{bmatrix} \cdot
\begin{bmatrix} s_{3} \\ k_{3}'\alpha_{3} \\0 \end{bmatrix} \cdot 
\begin{bmatrix} s_{4} \\ k_{4}'\alpha_{4} \\0 \end{bmatrix} \cdot 
\begin{bmatrix} s_{\widetilde{\alpha}} \\ -\widetilde{k}'\widetilde{\alpha} \\ -\widetilde{\alpha}\end{bmatrix} \cdot
\begin{bmatrix} s_{2} \\ 0 \\ \ell \alpha_2\end{bmatrix} \cdot 
\begin{bmatrix} s_{2} \\ \varepsilon \alpha_{2} \\ \ell \alpha_2\end{bmatrix} \\
&=\begin{bmatrix} s_{1}s_{3}s_{4}s_{-\widetilde{\alpha}} \\ k_{1}'\alpha_{1}+k_{3}'\alpha_{3}+k_{4}'\alpha_{4}-\widetilde{k}'\widetilde{\alpha}+\varepsilon \alpha_{2} \\ -\widetilde{\alpha}\end{bmatrix}.
\end{align*}
for some $\varepsilon \in \{\pm 1\}$. Since $\widetilde{\alpha}=\alpha_{1}+2\alpha_{2}+\alpha_{3}+\alpha_{4}$ it holds
\[1= -2\widetilde{k}' + \varepsilon \text{ and } \widetilde{k}'=k_{1}'=k_{3}'=k_{4}'.\]

If $\varepsilon = 1$ it follows directly $0=\widetilde{k}'=k_{1}'=k_{3}'=k_{4}'$.

If  $\varepsilon= -1$, then we get $-1=\widetilde{k}'=k_{1}'=k_{3}'=k_{4}'$.

Hence in this case there are at most two orbits represented by the factorizations
\begin{align} \label{fac1}
    \left(\begin{bmatrix} s_{1} \\ 0 \\ 0\end{bmatrix} ,
\begin{bmatrix} s_{3} \\ 0 \\ 0 \end{bmatrix} ,
\begin{bmatrix} s_{4} \\ 0 \\0 \end{bmatrix},
\begin{bmatrix} s_{\widetilde{\alpha}} \\ 0 \\ -\widetilde{\alpha}\end{bmatrix} ,
\begin{bmatrix} s_{2} \\ 0 \\ 0\end{bmatrix} ,
\begin{bmatrix} s_{2} \\ \alpha_2 \\ 0\end{bmatrix}\right)
    \end{align}
    and
    \begin{align} \label{fac2}
    \left( \begin{bmatrix} s_{1}  \\ -\alpha_1 \\ 0\end{bmatrix} ,
\begin{bmatrix} s_{3} \\ -\alpha_3 \\0 \end{bmatrix} , 
\begin{bmatrix} s_{4}  \\-\alpha_4 \\ 0\end{bmatrix} ,
\begin{bmatrix} s_{\widetilde{\alpha}} \\ \widetilde{\alpha} \\ -\widetilde{\alpha}\end{bmatrix} ,
\begin{bmatrix} s_{2} \\ 0 \\ 0\end{bmatrix} ,
\begin{bmatrix} s_{2} \\ -\alpha_2 \\ 0\end{bmatrix} \right),
    \end{align}
respectively. It is a consequence of  \cite{BW24} that these two factorizations  are in two different Hurwitz orbits.

Now assume $\overline{\ell} \neq 0$. We conclude with Lemma~\ref{lem:lambda_t} that $\Phi = D_4^{(1,1)}$ and $\overline{\ell} = -1$. As before we get 
\[1= -2\widetilde{k}' + \varepsilon \text{ and } \widetilde{k}'=k_{1}'=k_{3}'=k_{4}'.\]
Therefore, as before, if $\varepsilon = 1$ it follows directly $0=\widetilde{k}'=k_{1}'=k_{3}'=k_{4}'$, and if $\varepsilon= -1$, then we get $-1=\widetilde{k}'=k_{1}'=k_{3}'=k_{4}'$. But now if $\varepsilon = 1$ the factorization is in the Hurwitz orbit of the factorization (\ref{fac2}) and  if $\varepsilon = -1$ the factorization is in the Hurwitz orbit of the factorization \ref{fac1}.
\end{proof}

\begin{remark} \label{rmk:ObsMulti}
\begin{itemize}
\item[(a)]     
For type $D_4^{(1,1)}$  we denote by $R_1$ and $R_2$ the Hurwitz orbits of $\mathcal{ B}_6$ on $\Redd(c)$ that contain the facorizations (\ref{fac1}) and (\ref{fac2}), respectively.
\item[(b)] 
	Observe using (\ref{equ:DefMulti}) and the fact that $p(\Phi)$ is crystallographic yields that
	$$\begin{bmatrix}
		s_{\beta_{1}} \\  k_1 \beta_{1} \\  \ell_1 \beta_{1}
	\end{bmatrix} \cdots 
	\begin{bmatrix} 
		s_{\beta_{n}} \\  k_{n} \beta_{n} \\  \ell_{n} \beta_{n}
	\end{bmatrix}
	\begin{bmatrix} 
		s_{\beta} \\  k' \beta \\  \ell' \beta 
	\end{bmatrix}
	\begin{bmatrix} 
		s_{\beta} \\  k'' \beta \\  \ell'' \beta 
	\end{bmatrix} =
	\begin{bmatrix}
		s_{\beta_{1}} \cdots s_{\beta_n} \\  x \\  y
	\end{bmatrix}
	\begin{bmatrix} 
		\idop \\  (k''-k') \beta \\  (\ell''-\ell') \beta 
	\end{bmatrix} =
	\begin{bmatrix}
		s_{\beta_{1}} \cdots s_{\beta_n} \\  x+(k''-k')\beta \\  y+(\ell''-\ell')\beta
	\end{bmatrix},
	$$
	where $x,y \in L(\{\beta_1, \ldots, \beta_n\})$.
 \end{itemize}
\end{remark}
\medskip

\section{Rigidity of the poset $[1,c]^{\mathrm{gen}}$}\label{sec:rigid}

In this section $(W,S)$ is a tubular elliptic Weyl system of rank $n+2$ with  root system $\Phi$, set of reflections $T$ and Coxeter transformation $c \in W$, and we will continue the notation introduced in Section~\ref{sec:Elliptic} and in the beginning of Section~\ref{sec:HurwitzElliptic}.

\begin{Definition}
Let  $t =(t_{1},\ldots,t_{n +2})$ and $r = (r_{1},\ldots,r_{n+2}) \in \Redd(c)$ be two reduced generating reflection factorizations of $c$. Then the pair $(r,t)$ is called rigid if $t_{1}\cdots t_{h}=r_{1}\cdots r_{h}$ for some $1 \leq h \leq n$ implies that  $(t_{1},\ldots,t_{h},r_{h+1},\ldots,r_{n +2}) \in \Redd(c)$ as well. The poset  $[1,c]^{\mathrm{gen}}$ is called \defn{rigid} if  the pair $(r,t)$ is rigid for all $t,r \in \Redd(c)$.
\end{Definition}

Note that the poset $[1,c]^{\mathrm{gen}}$ is rigid, if and only if for every path from $1$ to $c$  the reflections that  label the edges of the path generate $W$. The aim of this section is to prove the following theorem.

\begin{theorem}\label{prop:GenPrefix}
Let $W$ be a tubular elliptic Weyl group and $c \in W$ a Coxeter transformation. Then $[\idop, c]^{\mathrm{gen}}$ is rigid.
\end{theorem}

For the convenience of the reader we like to recall the following notions. A \defn{parabolic subgroup} of a Coxeter system $(W,S)$ is a  conjugate of  subgroup of $W$ which is generated by a subset of $S$. Moreoover, a \defn{parabolic quasi-Coxeter element} $w$ is an element in $W$ that possess a reduced reflection factorization $w = t_1 \cdots t_m$ such that the reflections $t_1, \ldots , t_m$ generate a parabolic subgroup of $W$. If the parabolic subgroup equals $W$, then $w$ is simply called \defn{quasi-Coxeter element} (see \cite{BGRW}).

\subsection{The setup}

Let $(t_{1},\ldots,t_{n+2}),(r_{1},\ldots,r_{n+2}) \in \Redd(c)$ such that $u:= t_{1}\cdots t_{h}=r_{1}\cdots r_{h}$ for some $1 \leq h \leq n+2$. Further we set $v := t_{h+1}\cdots t_{n+2}=r_{h+1}\cdots r_{n+2}$. Thus  $u$ and $v$ are the prefix and \defn{suffix} of $c$, which we want to investigate.

Let $\Phi$ be the tubular elliptic root system associated to $W$. By Section \ref{sec:Elliptic} we have that $\phi:= \Phi_{\rm{fin}} = p_{R}(\Phi)$ is a finite root system of type $D_4, E_6, E_7$ or $E_8$, respectively. Let us also abbreviate by $p$ the corresponding projection map 
$$p:= p_R :W=W_{\Phi} \rightarrow W_{\phi},~ s_{\alpha} \mapsto s_{p(\alpha)}.$$ 

We start by analyzing the setting.

\begin{Lemma}\label{lem:Lift}
If $p(u) = p(t_1) \cdots p(t_h)$ is a reduced reflection factorization in $W_{\phi}$, then $(t_{1},\ldots ,t_{h})$  is the unique preimage  of  $(p(t_1), \ldots ,p(t_h))$   in $W^h$ such that $u = t_{1}\cdots t_{h}$.
\end{Lemma}

\begin{proof}
Let $\beta = p(\beta) + ka + \ell b \in \Phi$, where $k,\ell \in \ZZ$. Then the related reflection is in vector notation
$\displaystyle \begin{bmatrix} s_{p(\beta)}\\ k  p(\beta)\\ \ell p(\beta) \end{bmatrix}$ (see Section \ref{sec:HurwitzElliptic}), and  for $\beta_i = p(\beta_i) + k_ia + \ell_ib \in \Phi$, where
$i = 1,\ldots, h$, the product is (see Remark~\ref{rem:GroupOperation} or (\ref{equ:DefMulti})) 
$$
\begin{bmatrix} s_{p(\beta_1)}\\ k_1  p(\beta_1)\\ \ell_1 p(\beta_2)\end{bmatrix}
\begin{bmatrix} s_{p(\beta_2)}\\ k_2  p(\beta_2)\\ \ell_2 p(\beta_2)\\
\end{bmatrix} \cdots 
\begin{bmatrix} s_{p(\beta_h)}\\ k_h  p(\beta_h)\\ \ell_h p(\beta_h)\\
\end{bmatrix} =
\begin{bmatrix} s_{p(\beta_1)} s_{p(\beta_2)} \cdots s_{p(\beta_h)}\\ x\\y \end{bmatrix},$$
where $x,y \in L(p(\beta_1), \ldots , p(\beta_h))$. As $p(u) = p(t_1) \cdots p(t_h)$ is reduced by assumption, the roots related to $ p(t_1), \ldots ,p(t_h)$ are linearly independent, see \cite[Lemma 3]{Car72}. Therefore, they form a basis of the lattice that they span, and the multiplication above shows that the factorization of $p(u)$ determines uniquely a preimage of the factorization of $p(u)$ in $W^h$.
\end{proof}

\begin{corollary}\label{cor:ReducedQuasiCox}
If $p(u)$ is a parabolic quasi-Coxeter element and $p(u)=p(t_1) \cdots p(t_h)$ is reduced in $W_\phi$, then  $(t_1, \dots , t_h,r_{h+1},\ldots,r_{n+2}) \in \Redd(c)$. The same statement holds if If $p(v)$ is a parabolic quasi-Coxeter element and $p(v)=p(t_{h+1}) \cdots p(t_{n+2})$ is reduced in $W_\phi$.
\end{corollary}
\begin{proof}
By \cite[Theorem 1.1]{BGRW} there exists a braid $\sigma \in {\mathcal B}_h$ such that
$$\sigma (p(r_1), \ldots ,p(r_h)) = (p(t_1), \ldots ,p(t_h)).$$
Then Lemma~\ref{lem:Lift} implies $\sigma (r_{1},\ldots, r_{h})=(t_{1},\ldots, t_{h})$, as well as that $ (r_1, \ldots, r_{n+2})$ and \linebreak 
$ (t_1, \ldots t_h, r_{h+1}, \ldots r_{n+2})$ 
are in the same Hurwitz orbit. The first part of the assertion follows with Theorem~\ref{thm:MainElliptic}. Analogously, the assertion follows if $p(v)$ is a parabolic quasi-Coxeter element and $p(v)=p(t_{h+1}) \cdots p(t_{n+2})$ is reduced.
\end{proof} 

We can also treat another special case directly.

\begin{Lemma}\label{lem:p(v)reduced}
	If $p(u)=p(t_1) \cdots p(t_h)$  and $p(v)=p(t_{h+1}) \cdots p(t_{n+2})$ are reduced, then Theorem~\ref{prop:GenPrefix} holds.
\end{Lemma}
\begin{proof}
If $p(u)$ or $p(v)$ is a parabolic quasi-Coxeter element, then we can argue as before. Therefore let us assume that $p(u)$ and $p(v)$ are not  parabolic quasi-Coxeter elements. This implies that the reflection length of $p(u)$ as well as of $p(v)$ in $W_{\phi}$ is greater $3$, as in a finite simply laced Coxeter system of rank at most $3$ every element is a parabolic Coxeter element. Therefore $\phi$ is of type $E_{6}$, $E_{7}$ or $E_{8}$. In \cite[Tables 9-11]{Car72} all the conjugacy classes of the elements in the respective Coxeter group $W_{\phi}$ are listed. We can read off the elements that are not parabolic quasi-Coxeter elements and use GAP \cite{GAP4} to determine all the pairs of conjugacy classes of elements $x, y \in W_{\phi}$ such that 
\begin{itemize}
\item $x$ and $y$ are not parabolic quasi-Coxeter elements, and
\item the conjugacy classes of $xy$ and of $p(c)$ are equal.
\end{itemize}
We obtain the following possibilities for $x$ and $y$, where the notation of  \cite[Tables 9-11]{Car72} is
used.

$$
\begin{array}{|l|l|l|}
\hline 
\mbox{type of}~\phi  & ~\mbox{type of}~x & ~\mbox{type of }~y\\\hline \hline
E_6   &  A_1^4           & A_1^4\\ \hline
E_7   & A_1^4            & A_1^5\\
      & A_1^4            & (A_3 \times  A_1^2)'\\ \hline
E_8   & A_3 \times A_1^2 & A_3 \times A_1^2\\
      & A_1^4            & A_2^3  \\
      & A_1^4            &  D_4 \times  A_1^2\\ \hline
\end{array}
$$

We used GAP \cite{GAP4} to check that in all the listed cases there are no  reduced reflection factorizations $(u_1, \ldots , u_h)$ of $x$ and $(u_{h+1}, \ldots , u_{n+2})$ of $y$, respectively, such that $(u_1, \ldots , u_{n+2})$ is generating. Thus, as $(t_1, \ldots, t_{n+2})$ is generating, $p(u)$ or $p(v)$ is a  parabolic quasi-Coxeter element. The assertion follows with Corollary~\ref{cor:ReducedQuasiCox}.
\end{proof}

The proof of Theorem~\ref{prop:GenPrefix} will be based on  the following lemma, which 
we will prove in the following sections.

\begin{Lemma}\label{lem:GenerationPrefix}
 Let  $t =(t_{1},\ldots,t_{n +2})$ and $r = (r_{1},\ldots,r_{n+2}) \in \Redd(c)$ and suppose that $u:= t_1 \cdots t_h = r_1 \cdots r_h $. If $p(u)=p(t_1) \cdots p(t_h)$ is  reduced, then $\langle t_1, \ldots t_h \rangle =\langle r_1, \ldots r_h \rangle$.
\end{Lemma}

Notice, if $p(u)=p(t_1) \cdots p(t_h)$ is not reduced, then $p(v)=p(t_{h+1}) \cdots p(t_{n+2})$ is a reduced factorization and vice versa.  Thus one of the two factorizations is reduced.
By using the Hurwitz action we can switch between prefixes and suffixes and will obtain
by applying Lemma~\ref{lem:GenerationPrefix} that 
 $\langle t_1, \ldots t_h \rangle = \langle r_1, \ldots r_h \rangle$  or that
$\langle t_{h+1}, \ldots t_{n+2} \rangle =
\langle r_{h+1}, \ldots r_{n+2} \rangle$. In both cases Theorem~\ref{prop:GenPrefix} concludes.

We summarize the assumptions that we may suppose to prove Lemma~\ref{lem:GenerationPrefix}  in the following.

\begin{hyp}\label{hyp:Reduction}
The factorization $p(u) = p(t_1) \cdots p(t_h)$ is reduced, while $p(v)=p(t_{h+1}) \cdots p(t_{n+2})$ is not reduced, and $p(u)$ is not a parabolic quasi-Coxeter element in $W_\phi$.
\end{hyp}

We can modify the factorizations $v = t_{h+1} \cdots t_{n+2} = r_{h+1} \cdots r_{n+2} $ conveniently, as by \cite[Corollary~1.4]{LR16} we can choose the roots $\beta_i, \gamma_i, \beta, \gamma \in \phi$ for $h+1 \leq i \leq n+2$ such that 
$$v= 
\begin{bmatrix}
s_{\beta_{h+1}} \\  k_{h+1} \beta_{h+1} \\  l_{h+1} \beta_{h+1}
\end{bmatrix} \cdots 
\begin{bmatrix} 
s_{\beta_{n}} \\  k_{n} \beta_{n} \\  l_{n} \beta_{n}
\end{bmatrix}
\begin{bmatrix} 
s_{\beta} \\  k' \beta \\  l' \beta 
\end{bmatrix}
\begin{bmatrix} 
s_{\beta} \\  k'' \beta \\  l'' \beta 
\end{bmatrix}\\
 = 
\begin{bmatrix} 
s_{\gamma_{h+1}} \\  \kappa_{h+1} \gamma_{h+1} \\  \iota_{h+1} \gamma_{h+1}
\end{bmatrix} \cdots 
\begin{bmatrix} 
s_{\gamma_{n}} \\  \kappa_{n} \gamma_{n} \\  \iota_{n} \gamma_{n}
\end{bmatrix}
\begin{bmatrix} 
s_{\gamma} \\  \kappa' \gamma \\  \iota' \gamma 
\end{bmatrix}
\begin{bmatrix} 
s_{\gamma} \\  \kappa'' \gamma \\ \iota'' \gamma 
\end{bmatrix},$$
Observe, if $k: = k'' - k'  = 0$, then $l:= l'' - l' \neq 0$ and vice versa, as  these factorizations are reduced in $W$ by our assumption.

Further for $ 1 \leq i \leq h$ let $k_i, l_i, \kappa_i, \iota_i \in \ZZ$ and $\beta_i, \gamma_i \in \phi$  such that we have in vector notation 
$$ t_i =  \begin{bmatrix}
s_{\beta_{i}} \\  k_{i} \beta_{i} \\  l_{i} \beta_{i}
\end{bmatrix} ~\mbox{and}~
r_i = \begin{bmatrix} 
s_{\gamma_{i}} \\  \kappa_{i} \gamma_{i} \\  \iota_{i} \gamma_{i}
\end{bmatrix}.
$$
\medskip

The strategy of the proof of Theorem~\ref{prop:GenPrefix} is as follows. The major step is to show Lemma~\ref{lem:GenerationPrefix} under the assumptions of  Hypothesis~\ref{hyp:Reduction}. Therefore in Section~\ref{sec:EqualSuffixes} we show that up to Hurwitz action on the orbit, which contains $(r_{h+1}, \ldots , r_{n+2})$, we have $r_{n+1}r_{n+2} = t_{n+1}t_{n+2}$. After that, in Section~\ref{sec:ProofRigid},  we prove via contradiction that $\beta_1, \ldots , \beta_n$ and $\gamma_{1}, \ldots ,\gamma_{n}$ generate the same root sublattice. In order to obtain this result we will heavily use the orbits of the Hurwitz action on the set of reduced factorizations of $p(c)$. For convenience of the reader we list them  in Appendix \ref{sec:3A2}-\ref{sec:subE8} and we provide further information on the orbits in Section~\ref{obs:parabolic}. It  will follow from \cite[Proposition 5.1]{BGRW} that $W':= \langle s_{\beta_1}, \ldots , s_{\beta_n} \rangle$ equals $\langle s_{\gamma_{1}}, \ldots ,s_{\gamma_{n}} \rangle$. As $p(c)$ is a quasi-Coxeter element in $W'$, and $p(u)$ a parabolic quasi-Coxeter element in $W'$, since it is a prefix of $p(c)$, application of \cite[Theorem~1.1]{BGRW} in $W'$ will yield that $(s_{\beta_1}, \ldots , s_{\beta_h})$ and $(s_{\gamma_{1}}, \ldots ,s_{\gamma_{h}})$ are in the same Hurwitz orbit, and that therefore $\langle s_{\beta_1}, \ldots , s_{\beta_h} \rangle = \langle s_{\gamma_{1}}, \ldots ,s_{\gamma_{h}} \rangle $ which will prove Lemma~\ref{lem:GenerationPrefix}. As a consequence, Theorem~\ref{prop:GenPrefix} will follow.
\medskip

Next we present  how we will often manipulate the factorization of an element of $W$ by the Hurwitz action.

\begin{Lemma} \label{lem:ConjHO}
	For each $w \in \langle s_{\delta_1}, \ldots, s_{\delta_n} \rangle$, where $\delta_1, \ldots , \delta_n \in \phi$, and all $k_i.  l_i, k, l , \kappa, \iota \in \ZZ$ there exist
	$k_i', l_i',k',l', \kappa', \iota' \in \ZZ$ and a braid $\sigma \in \mathcal{B}_{n+2}$ such that 
	$$
	\sigma \left( \begin{bmatrix} 
		s_{\delta_{1}} \\  k_1 \delta_1 \\  l_1 \delta_1
	\end{bmatrix}, \ldots ,
	\begin{bmatrix} 
		s_{\delta_{n}} \\  k_n \delta_n \\  l_n \delta_n
	\end{bmatrix},
	\begin{bmatrix} 
		s_{\delta} \\  k \delta \\  l \delta 
	\end{bmatrix},
	\begin{bmatrix} 
		s_{\delta} \\  \kappa \delta \\  \iota \delta 
	\end{bmatrix} \right) =
	\left( \begin{bmatrix} 
		s_{\delta_{1}} \\  k_1' \delta_1 \\  l_1' \delta_1
	\end{bmatrix}, \ldots ,
	\begin{bmatrix} 
		s_{\delta_{n}} \\  k_n' \delta_n \\  l_n' \delta_n
	\end{bmatrix},
	\begin{bmatrix} 
		ws_{\delta}w^{-1} \\  k' w(\delta) \\  l' w(\delta)
	\end{bmatrix},
	\begin{bmatrix} 
		ws_{\delta}w^{-1} \\  \kappa' w(\delta) \\  \iota' w(\delta)
	\end{bmatrix} \right)
	$$
\end{Lemma}

\begin{proof}
By \cite[Lemma 2.5]{WY19} there exists a braid $\sigma \in \mathcal{B}_{n+2}$ such that
$$
\sigma(s_{\delta_1},\ldots, s_{\delta_n}, s_{\delta},s_{\delta}) = (s_{\delta_1},\ldots, s_{\delta_n}, ws_{\delta}w^{-1},ws_{\delta}w^{-1}).
$$
The braid $\sigma$ is also the braid we were looking for.
\end{proof}

\subsection{Reduction to equal suffixes}\label{sec:EqualSuffixes}
Our next goal is to show the following:

\begin{Lemma} \label{lem:HOspezial}
Under the assumption of Hypothesis~\ref{hyp:Reduction} 
there is a factorization $(\rho_{h+1}, \ldots , \rho_{n+2})$ in the Hurwitz orbit of $(r_{h+1}, \ldots , r_{n+2})$
such that $\rho_{n+1}\rho_{n+2} = t_{n+1} t_{n+2}$.\\
\end{Lemma}

\begin{remark} \label{rmk:parabolic}
There are further reductions that we can make in some cases. 
\begin{enumerate}
\item[(a)] Recall that by  Hypothesis~\ref{hyp:Reduction} $p(u)$ is not a parabolic quasi-Coxeter element in $W_{\phi}$. Hence its reflection length is greater than $3$ (see \cite[Tables~4, 9,10,11]{Car72}). Therefore, if $\phi$ is of type $D_4$, we get that $p(u)$ and $p(v)$ are of reflection length $4$ and $0$, respectively. In particular Lemma~\ref{lem:HOspezial} holds for type $D_4$. If $\phi$ is of type $E_6$, $E_7$ or $E_8$, respectively, then we get that the reflection length of $p(v)$ is bounded  from above by  $2, 3,$ or $4$, respectively.

\item[(b)] If $p(v)$ is a parabolic quasi-Coxeter element in $W_{\phi}$, we may assume by applying the Hurwitz action that $\beta_i = \gamma_i$ for $h+1 \leq i \leq n$ (see \cite[Theorem 1.1]{BGRW}). Further by (b) we obtain that 
$$k \beta  \in L(\psi),~~\mbox{where}~\psi:= \langle \beta_{h+1}, \ldots , \beta_{n}, \gamma \rangle_{\text{RS}}\subseteq \phi.$$
Suppose  further that $P = W_{\psi}$ is  a parabolic subgroup of $W_{\phi}$, and let  $U = \Fix(P)$. Then the elements in  $\psi$ are precisely the roots that are  perpendicular to $U$. Thus $ \beta$ is perpendicular to $U$, as well, which implies  $\beta \in \psi$.

Further notice that the assumption that $p(v)$ is a parabolic quasi-Coxeter element yields that $Q:= \langle s_{\beta_{h+1}} , \dots ,s_{\beta_{n}} \rangle$ is a parabolic subgroup of $P$. Hence, if $\gamma$ is not in $L(\{\beta_{h+1} , \ldots , \beta_n\})$, then $Q$ is a parabolic subgroup of rank one less than the rank of $P$, and the roots $\beta_{h+1} , \ldots , \beta_n, \gamma$ are linear independent.

\item[(c)] Let $\phi$ be of type $E_n$ where $6\leq n \leq 8$. Then all the elements  of reflection length at most $3$ in $W_\phi$ are parabolic quasi-Coxeter elements (see part (a)). This implies that the reflection subgroups, which can be generated by at most three reflections, of the parabolic subgroups of $W_\phi$  are all parabolic subgroups of $W_\phi$. Hence, if $n-h+ 1 \leq 3$, i.e. $n-h \leq 2$, then $\langle Q, s_{\beta}\rangle$ is a parabolic subgroup of $P$; so equals $P$.

The only reflection subgroup of $W_\phi$ of rank $4$, which is not parabolic, is of type $A_1^4$ (see \cite[Tables~ 9,10,11]{Car72}). Therefore, if $\Psi$ is of rank $4$, but does not contain a proper root subsystem of type $A_1^4$, then we can conclude
again that $\langle Q, s_{\beta}\rangle$ equals $P$.
\end{enumerate}
\end{remark}

\begin{proof}[Proof of Lemma \ref{lem:HOspezial}]
If $n- h= 0$, then the assertion immediately holds. Therefore the assertion holds for type $D_4$. If $\Phi$ is of type $E_6$, then $p(c)$ is of type $A_2^3$ and every prefix of $p(c)$ is a parabolic quasi-Coxeter element (see \cite[Table~9]{Car72}). Hence we are left to consider the cases $E_7$ and $E_8$.

We assume from now on that $n-h \geq 1$. We explain our method of proof by considering $n-h = 1$, which is $h = n-1$. Then $v = t_n t_{n+1} t_{n+2}$ and $\psi = \langle  \beta_{n}, \gamma \rangle_{\text{RS}} $. Recall that $p(c)$ is not a quasi-Coxeter element, and that therefore $p(t_1), \ldots , p(t_{n})$ generate a proper subgroup of  $W_{\phi}$. Since $p(t_1), \ldots , p(t_{n+2})$  generate $W$, the ambient root $\gamma$ of $p(r_{n+1}) = p(r_{n+2})$ is not contained in the lattice $L(\beta_1 , \ldots \beta_n)$ (see \cite[Theorem~1.1]{BaW}), and in particular $\gamma \neq \pm \beta_n$. It follows that $\psi$ is either of type $A_1^2$ or $A_2$. In both cases $W_\psi$ is a parabolic subgroup, and $Q:= \langle s_{\beta_n} \rangle$  is a parabolic subgroup of $W_\psi$ of rank one less than the rank of $W_\psi$. And $Q$ does neither contain $s_\beta$ nor $s_\gamma$. Therefore $P = \langle Q, s_{\beta}\rangle$, as well, and  there is $w \in Q$ such that $w(\gamma) = \beta$ (see \cite[Theorem~1.6]{BaW}).  Thus we can apply Lemma~\ref{lem:ConjHO} and obtain $\beta = \gamma$.  Then, as $\beta_n$ and $\beta$ are linearly independent, it follows $k_n = \kappa_n$ and  $l_n = \iota_n$. This shows the assertion in the case $n-h = 1$.

The same argument can be used whenever $p(v)$ is a parabolic quasi-Coxeter element and $W_{\psi}$ a parabolic subgroup of $W$ of rank at most $4$ (see Remark~\ref{rmk:parabolic} (b) and (c)). As all the  reflection subgroups of $W$ of rank at most $3$ are parabolic (see Remark~\ref{rmk:parabolic} (c)). Therefore, the above argument also holds for $n-h =2$.

In the  cases $n - h  \geq 2$, we used the GAP-function \textit{ClassMultiplicationCoefficient} to determine the conjugacy classes $\mathcal{C}_1$ and $\mathcal{C}_2$ which contain the elements $x$ and $y$ respectively, such that $xy$ is conjugate to $p(c)$. As we assume that $p(u)$ is not a parabolic quasi-Coxeter element, we only considered those $x$ that are not parabolic quasi-Coxeter. We found out that then $y$ is always a parabolic quasi-Coxeter element (see Table~\ref{tab:Types} and \cite[Tables~10,11]{Car72}). 

Therefore we may assume $\beta_i = \gamma_i$ for $h+1 \leq i \leq n$, and further we have that $Q:= \langle s_{\beta_{h+1}}, \ldots, s_{\beta_{n}} \rangle$ is a parabolic subgroup of $W_\psi$. If $W_\Psi$ is also itself a parabolic subgroup of rank at most $3$, then we can finish the proof as we did in the case $n-h = 1$ (see Remark~\ref{rmk:parabolic} (b) and (c)).

In Table \ref{tab:Types} we list all the possible types for $p(u)$ and $p(v)$, which we found in our GAP-calculation. We also list the type of $\psi$ in precisely those cases where $W_\psi$ may not be a parabolic subgroup. It turns out that in all cases,  the rank of $\psi$ is $n-h+1$.

\begin{table}
\centering
\begin{tabular}{|l|l|l|l|l|l|}
\hline
type of $\phi$ & $h$ &  $n-h $& type of $p(u)$ & type of $p(v)$ &type of $\psi$\\ \hline \hline
$E_7$ & $4$ & $3$ &$A_1^4$  & $A_1^3$  & $A_1^4$ \\ \hline
$E_7$ & $5$ & $2$ &$A_1^5 ,~ (A_3 \times A_1^2)''$ & $A_1^2$ & -- \\ \hline
$E_8$ & $4$ &$4 $ &$A_1^4$ & $A_2 \times A_1^2$ & $(A_3 \times A_1^2)^\prime$\\ \hline
$E_8$ & $4$ &$4 $ &$A_1^4$ &$~A_2^2$ &--  \\ \hline
$E_8$ & $5$ & $3$ & $A_1^5$ & $A_1^3$ & $A_1^4$ \\ \hline
$E_8$ & $5$ & $3$ & $A_1^5$ & $A_2 \times A_1$ & -- \\ \hline
$E_8$ & $6$ &$2 $& $A_2 \times A_1^4,~ A_3 \times A_1^3$ & $A_1^2$ & -- \\ \hline
$E_8$ & $6$ & $2$& $(A_5 \times A_1)',~A_2 \times A_1^4$ & $A_2$& -- \\ \hline
\end{tabular}
\caption{Types of $p(u)$ and $p(v)$.}
\label{tab:Types}
\end{table}

To prove the lemma it remains to consider $n-h \in \{3,4\}$ and the cases where $W_\psi$ is not a parabolic subgroup of $W$, that is the cases in Table~\ref{tab:Types} with an entry in the last column. 

Assume first $n-h = 3$.
If $\phi$ is of type $E_7$ or $E_8$ and  $\psi$ of type $A_1^4$, then there are two root subsystems of $\phi$ of this type: for one $W_\psi$ is parabolic and for the other one not. We checked using GAP that the roots $\beta_1, \ldots, \beta_n, \gamma$ generate the root lattice $L(\phi)$ if and only if $ W_\psi$ is parabolic. 
Moreover, by consulting Table~\ref{tab:Types} we see that all the reflection subgroups of $W_\Psi$ are parabolic subgroups by Remark~\ref{rmk:parabolic}, yielding again the assertion.

Last we consider $n-h = 4$.
Then $\phi$ is of type $E_8$ and $p(v)$ of type $A_2 \times A_2$ or $(A_3 \times A_1^2)^\prime$. In the first case the reflection subgroups of rank $5$ of $W_\phi$ that contain $A_2 \times A_2$ are of type $A_2^2 \times A_1, A_3 \times A_2$ or $A_5$, which are all parabolic subgroups. Moreover, all its reflection subgroups are also parabolic subgroups. Thus we can argue as before. Now assume that $\psi$ is of type $(A_3 \times A_1^2)^\prime$.
Then $p(v)$ is of type $A_2 \times A_1^2$ and $p(u)$ of type $A_1^4$. Assume without loss of generality that $\beta_5$ and $\beta_6$ commute with $\beta_7$ and $\beta_8$ and that the latter form an $A_2$-system. Let $\widetilde{\psi}$ be the smallest root subsystem containing $\beta_5, \ldots , \beta_8, \beta$. If $W_{\widetilde{\psi}}$ is a parabolic subgroup, we proceed by changing the roles of $\psi$ and $\widetilde{\psi}$. Therefore we may assume that $\psi$ and $\widetilde{\psi}$ are both not parabolic, that is both are of type $(A_3\times A_1^2)'$ and intersect in a system of type $A_3\times A_1^2$. As $\beta, \gamma \not\in L(\beta_5, \ldots, \beta_8)$, it follows that $\beta$ and $\gamma$ are perpendicular to $\beta_5$ and $\beta_6$. Thus $k \beta \in L(\gamma, \beta_7, \beta_8)$, and as the roots $\gamma, \beta_7, \beta_8$ form an $A_3$-system, which is parabolic in $\phi$, we obtain as above $\beta = \gamma$. This yields the assertion, as $\beta_5, \ldots, \beta_9, \beta$ are again linear independent.
\end{proof}

\medskip
The next step is to derive from $t_{n+1}t_{n+2} = r_{n+1} r_{n+2}$ that the factorizations $(p(t_1), \ldots,  p(t_n))$ and $(p(r_1), \ldots, p(r_n))$ are  in the same Hurwirtz orbit. Therefore, we need more explicit information on the Hurwitz orbits on $\Red_T(p(c))$.

\subsection{The Hurwitz orbits on $\Red_T(p(c))$}\label{obs:parabolic}
In this section we recall the description of the Hurwitz orbits on  $\Red_T(p(c))$ given in \cite{Klu87}. 

Denote the number of Hurwitz orbits for the set of reduced reflection factorizations for $p(c)$ by $N$. Let $i \in \{1, \ldots, N\}$ and let $R_i$ be a set of roots corresponding to a reduced reflection factorization of $p(c)$ from the $i$th orbit. Define $W_i:=W_{R_i}$, and note that this definition is independent of the chosen set $R_i$ in the $i$th orbit since the Hurwitz action preserves the generated group. In particular, the roots related to each orbit form a root subsystem $\phi_i:=W_i(R_i)$ of $\phi$ for $1 \leq i \leq N$. 

We start by observing some general facts on the lattices $L_i:= L(\phi_i)$ as on the subgroups $W_i$ of $p(W)$ generated by the reflections with respect to the set of roots $\phi_i$, for $1 \leq i \leq N$.

\begin{Lemma}\label{lem:DiffOrbits}
\begin{itemize}
\item[(a)] The subgroups $W_i$
are pairwise different proper subgroups of $p(W)$, and they  are not parabolic subgroups of $p(W)$, for $1 \leq i \leq N$.
\item[(b)] The lattices $L_i $ are pairwise different proper sublattices of $L(p(\phi))$, for $1 \leq i \leq N$.
\end{itemize}
\end{Lemma}
\begin{proof}
 The element $p(c)$ is not quasi-Coxeter in $p(W$), but in $W_i$ for $1 \leq i < j \leq N$. Therefore $W_i$ is not a parabolic subgroup of $p(W)$, but the Hurwitz action on the set of reduced factorizations of $p(c)$ in $W_i$ is transitive (see \cite[Corollary~1.3]{BGRW}). Since $W_i$ and $W_j$ correspond to different Hurwitz orbits on the set of reduced factorisations in $p(W)$ for $1 \leq i < j \leq N$, it is $W_i \neq W_j$, which is (a).
 
By \cite[Lemma 5.7]{BGRW} we have $\phi_i = \{r \in L_i \mid (r \mid r) = 2\}$. Therefore, $L_i = L_j$ implies $\phi_i = \phi_j$, which yields $W_i = W_j$, and $i = j$ by (a), for $1 \leq i \leq j \leq N$, which  is (b).
\end{proof}

For the relevant types $\phi$ the number $N$, and the types of the root systems $\phi_i$ are given in Table \ref{tab:Tab_p(c)}.

\begin{table}[h]
\centering
\begin{tabular}{|l|l|l|l|}
\hline
type of $\phi$ & types of $p(c)$  & $N$ & types of $\phi_i$\\
\hline \hline
$D_4$ & $A_1^4$ & $3$ & $A_1^4$\\ \hline
$E_6$ & $A_2^3$ & $4$ & $A_2^3$\\ \hline
$E_7$ & $A_3^2\times A_1,~D_4(a_1) \times A_1^3$ & $7$ & $A_3^2\times A_1,~D_4 \times A_1^3$ \\ \hline
$E_8$ & $A_5 \times A_2 \times A_1$ & $12$ & $A_5 \times A_2 \times A_1$ \\
\hline 
\end{tabular}
    \caption{Types of $p(c)$ and $\phi_i$}
    \label{tab:Tab_p(c)}
\end{table}

If $\phi$ is of type $D_4$ or $E_6$, we have a partition $\phi = \phi_1 \cup \ldots \cup \phi_N$, where $N$ and the type of the $\phi_i$ are given by the table above. If $\phi$ is of type $E_7$, the set of reduced reflection factorizations of $p(c)$ consists of $6$ Hurwitz orbits of types $A_3^2 \times A_1$ while one of them is of type $D_4(a_1) \times  A_1^3$ (see Section \ref{sec:subE7}). If $\phi$ is of type $E_8$, the set of reduced reflection factorizations of $p(c)$ consists of $12$ Hurwitz orbits of type $A_5 \times A_2 \times A_1$.

The following statement is a consequence of \cite[Kapitel I, 2.3, 3.4, 4.6]{Klu87}, as $C_{p(W)}(p(c))$ acts on the set of Hurwitz orbits in the action of the braid group on the set of reduced factorizations of $p(c)$.

\begin{Lemma}\label{Lem:TransOnHur}
If $\phi$ is of type different from $E_7$, then $C_{p(W)}(p(c))$ is transitive on the set of subsystems $\phi_i$ for $1 \leq i \leq N$. If $\phi$ is of type $E_7$, then $C_{p(W)}(p(c))$ is transitive on the set of subsystems $\phi_i$ that are of type $A_3^2 \times A_1$.
\end{Lemma}

\begin{Lemma}\label{Lem:Numberm}
Let $i \in \{1, \ldots , N\}$ and $\beta \in \phi$ such that $L(\phi_i , \beta) = L(\phi)$, and let $m_t$ be the coefficient of the red vertex in Figure~\ref{fig:Dynkin}. If $z$ is an integer such that $z \beta \in L_i$, then $z$ is divisible by $m_t$.
\end{Lemma}
\begin{proof}
Assume that $\phi_i$ is not of type $D_4 \times A_1^3$. Then we may assume  that $i = 1$ and that $\phi_1$ is generated by the set of roots $(\{\alpha_1, \ldots , \alpha_n\}\setminus{\{\alpha_t\}}) \cup \{\widetilde{\alpha}\}$ due to Lemma~\ref{Lem:TransOnHur} (see the Appendix). As $m_t$ is the coefficient of $\alpha_t$ in the expansion of $\widetilde{\alpha}$ in the basis $\Pi := \{\alpha_1, \ldots \alpha_n\}$ it follows, if $z$ is an integer such that $z \alpha_t \in L_1$, then $z$ is divisible by $m_t$. As $\alpha_t$ is in $L:= L(\phi)$, it is a $\Z$-linear combination of $(\Pi\setminus{\{\alpha_t\}}) \cup \{\tilde{\alpha}, \beta\}$. Thus if  $z\beta \in L_1$, then it follows $z \alpha_t \in L_1$, as well, which yields that $z$ is divisible by $m_t$, the assertion.  

Now assume that $\phi$ is of type $E_7$ (hence $t=4$) and that $\phi_i$  is of type $D_4(a_1)\times A_1^3$. Then one particular choice for the roots $\beta_1, \ldots, \beta_7$ is given in Section \ref{sec:subE7}. If we assume that $2$ or $3$ is the smallest natural number such that $2 \alpha_4$ or $3 \alpha_4$ is in $L_7:=L(\beta_1, \ldots, \beta_7)$, then it is a straightforward calculation (by using these roots) to see that $\alpha_j \in L_7$, for $j \neq 4$. By \cite[Table 3.3]{BH19} we have that $L/L_7 \cong \mathbb{Z}_2^2$ and we arrive at a contradiction. Again, $4$ is the smallest natural number such that $4 \alpha_4 \in L_7$.
Therefore, $4 \delta \in L_7$ for all $\delta \in \phi$. This implies that $L( \beta_1, \ldots,  \beta_7, 2\delta )$ is a proper sublattice of $L$ for all $\delta \in \phi$: Assume to the contrary that $L( \beta_1, \ldots,  \beta_7, 2\delta )=L$ for some $\delta$. Hence there are $\lambda_1, \ldots , \lambda_8 \in \mathbb{Z}$ such that
$$
	\alpha_4  = \sum_{i=1}^7 \lambda_i \beta_i + \lambda_8 2 \delta, \quad~\mbox{which implies}
	 ~ 2\alpha_4  = \sum_{i=1}^7 \lambda_i 2\beta_i + \lambda_8 4 \delta.
$$
Since $4 \delta \in L_7$, we obtain $2 \alpha_4 \in L_7$, contradicting what we have observed above.
\end{proof}

\begin{corollary}\label{coro:Propmt}
The following holds.
\begin{itemize}
    \item[(a)]  For all $\delta \in \phi$ and all $1 \leq i \leq N$ we have $m_t \delta \in L_i$.
    \item[(b)] Let $r$ be a proper divisor of $m_t$ and $\delta \in \phi \setminus{L_i} $. Then $L(\phi_i \cup \{r \delta \})$ is a proper sublattice of $L(\phi)$.
\end{itemize}
\end{corollary}
\begin{proof} 
Assertions (a) and  (b) follow by applying again Lemma~\ref{Lem:TransOnHur} and by using the same argument as in the last paragraph of the proof of 
Lemma~\ref{Lem:Numberm}. 
\end{proof}

\subsection{Proofs of the rigidity of $[\idop, c]^{\mathrm{gen}}$ and of Theorem \ref{cor:MainTheorem}}\label{sec:ProofRigid}

As before, let  the roots related to $t_1, \ldots, t_n$, and $t_{n+1}, t_{n+2}$ be 
$$\beta_i + k_i a + l_i b, ~\mbox{for}~1 \leq i \leq n, ~\mbox{and}~\beta + k'a + l' b,~ \beta + k''a + l''b,$$  and those related to $r_1, \ldots,r_n$, and $r_{n+1}, r_{n+2}$ be 
$$\gamma_i + \kappa_i a + \iota_i b, ~\mbox{for}~1 \leq i \leq n, ~\mbox{and}~\gamma + \kappa' a + \iota' b, ~\gamma + \kappa'' a + \iota'' b.$$
Then $\beta_1, \ldots \beta_n$ and $\gamma_1, \ldots ,\gamma_n$ are in subsystems $\phi_s$ and $\phi_t$ for some  $s,t \in \{1, \ldots , N\}$, respectively. We choose the roots such that they are positive roots in $\phi_s$ and $\phi_t$, respectively.

Further due to Lemma \ref{lem:HOspezial}  we can assume that $t_{n+1}t_{n+2}=r_{n+1}r_{n+2}$.
 We derive from it that 
 \[\beta = \gamma~ \text{as well as }~ k = \kappa ~ \text{and}~ l = \iota, ~\text{ where }~k:=k''-k', \kappa:= \kappa''-\kappa', 
 l:=l''-l'~\text{and}~\iota := \iota''-\iota'.\]
 
 \begin{proof}[\textbf{Proof of Theorem~\ref{prop:GenPrefix}}]
 To show the assertion we will conclude from $(k,l) = (\kappa, \iota)$ that $s = t$ in a case by case study. If we know that $s = t$, then $W_s  = W_t$, and  $p(c) = p(t_1) \cdots p(t_n)= p(r_1) \cdots p(r_n)$ are reduced factorizations of the quasi-Coxeter element $p(c)$ in $W_s$. Further $p(u)$ is a parabolic quasi-Coxeter element in $W_s$ by \cite[Corollary~6.11]{BGRW}, and the assertion follows as in the proof of Corollary~\ref{cor:ReducedQuasiCox}.
 \\
Now we start the case by case study to show $s = t$.

\subsubsection*{\bm{$D_4^{(1,1)}$}}
We obtain in vector notation (see Remark \ref{rmk:ObsMulti} as well as the beginning of Section \ref{sec:HurwitzElliptic} for the choice of Coxeter transformation)

\begin{align*}
c=
\begin{bmatrix}
s_{1}s_{3}s_{4}s_{\widetilde{\alpha}} \\ \alpha_{2} \\  -\widetilde{\alpha}
\end{bmatrix}
&=\begin{bmatrix} p(t_1)p(t_2)p(t_3)p(t_4) \\ k_{1}\beta_{1}+ k_{2}\beta_{2} +k_{3}\beta_{3} +k_{4}\beta_{4}
+(k''-k')\beta \\  l_{1}\beta_{1}+ l_{2}\beta_{2} -l_{3}\beta_{3} +l_{4}\beta_{4}
+(l''- l')\beta 
\end{bmatrix}.
\end{align*}
Therefore, 
\begin{align*}
\alpha_2 & =k_{1}\beta_{1} + k_{2}\beta_{2} +k_{3}\beta_{3} +k_{4}\beta_{4}
+ k\beta ,~\mbox{and}~\\ 
-\widetilde{\alpha} & = l_{1}\beta_{1} + l_{2}\beta_{2} +l_{3}\beta_{3} +l_{4}\beta_{4}
+l\beta.
\end{align*}

We conclude that $x + k\beta=  \alpha_2$ and $y + l\beta =  - \widetilde{\alpha} $ for some $x,y \in L(\phi_s)$. As $\phi_s \cup \{\beta\}$ generates $L(\Phi)$, it is $\beta \not\in \Phi_s$. 
Observe that $k \beta \in L(\Phi_s)$ if and only if $\alpha_2 \in L(\phi_s)$;  and $l \beta \in L(\phi_s)$ if and only if $\widetilde{\alpha} \in L(\Phi_s)$.  Analogously the same holds for $\gamma = \beta$, $\alpha_2$, $\widetilde{\alpha}$, and $\phi_t$.

Further it is easy to see that $2 \delta \in L(\phi_i)$ for all $\delta \in \phi$ and for $i = 1,2,3$. Thus $k = \kappa$ is even if and only if $\alpha_2 \in L(\phi_s)$ if and only if $\alpha_2 \in L(\phi_t)$, and the statement holds for $l = \iota$ and $\widetilde{\alpha}$ analogously. As the $\phi_i$ are disjoint and as only one of $\phi_1, \phi_2, \phi_3$ neither contains $\alpha_2$ nor $\widetilde{\alpha}$, we get $s = t$.
\bigskip

\subsubsection*{\bm{$E_n^{(1,1)}$} \bm{$(n \in \{6,7,8\})$}}
Here we obtain in vector notation 
\begin{align*}
c=
\begin{bmatrix}
p(c) \\ \alpha_{4} \\-
\widetilde{\alpha}
\end{bmatrix}=
\begin{bmatrix} s_{\beta_1}\cdots s_{\beta_n} \\ 
x_s + k \beta \\ y_s + l \beta\end{bmatrix}=
\begin{bmatrix} s_{\gamma_1}\cdots s_{\gamma_n} \\ 
x_t + \kappa \beta \\ y_t + \iota \beta\end{bmatrix}.
\end{align*}
for some $x_s,y_s \in L(\phi_s)$, $x_t,y_t \in L(\phi_t)$. Thus we obtain
\begin{align} \label{equ:E71111}
x_s + k\beta =  \alpha_4 \quad \text{and} \quad y_s + l\beta =  -\widetilde{\alpha},
\end{align}
as well as
\begin{align*}
x_t + \kappa \beta =  \alpha_4 \quad \text{and} \quad y_t + \iota \beta & =  -\widetilde{\alpha}
\end{align*}
As we saw before $\beta \not\in L_s, L_t$.
Further notice, that for $z = s,t$, it holds 
$$
 i  \alpha_4 - j \widetilde{\alpha}   \stackrel{(\ref{equ:E71111})}{=} i 
(x_z + k \beta)  - j(y_z + l \beta)  = 
\underbrace{ix_z-jy_z}_{\in L_z} + (ik - jl)\beta,$$
which implies  $i  \alpha_4 - j \widetilde{\alpha}  \in L_z$. 
By Corollary~\ref{coro:Propmt} $m_t \delta \in L_z$ for every $\delta \in \phi$, it is $(ik - jl)\beta \in L_z$ if and only if $ik \equiv jl$ mod $m_t$. 
As $k = \kappa$ and $l= \iota$, we obtain that 
$i  \alpha_4 - j \widetilde{\alpha}$ is in $L_s$ if and only if $i  \alpha_4 - j \widetilde{\alpha}$ is in $L_t$.

In the case $n = 6$, we argue as in the $D_4^{(1,1)}$-case: the four systems $\phi_i$ are disjoint, and one precisely contains $\alpha_4$, another one $ \widetilde{\alpha}$, a third 
$\widetilde{\alpha}-\alpha_4$ and the last none of them.cThis shows that $s = t$.

In the cases $n = 7, 8$  we determine using GAP 
the tuples $(i,j)$ for $0 \leq i,j \leq m_t-1$ such that $i  \alpha_4 - j \widetilde{\alpha}  \in L_z$ for each lattice $L_z$ with $1 \leq z \leq N$. For each tuple $(i,j)$ we then calculate the pairs $(k,l)$ modulo $m_t$, such that  $ik \equiv jl$ mod $m_t$, see Table \ref{tab:Tuples78}. We get that two such tuples $(k,l)$ are different if they are related to different orbits. This implies, as $(k,l) = (\kappa, \iota)$ that $s = t$, and the assertion follows.
\end{proof}

\begin{table}[h]
\centering
\begin{tabular}{ll}
\begin{tabular}{|l|l|}
\hline
$z$  & $(k,l)$\\
\hline \hline
1 & (1,2), (3,2) \\
2 &  (1,0), (3,0) \\
3 & (2,3), (2,1) \\
4 & (0,3), (0,1) \\
5 & (1,3), (3,1) \\
6 & (1,1), (3,3) \\
7 & (2,2), (2,0), (0,2), (0,0)\\
\hline 
\end{tabular}
&
\begin{tabular}{|l|l||l|l|}
\hline
$z$  & $(k,l)$ & $z$ & $(k,l)$\\
\hline \hline
1 & (1,0), (5,0) & 7 & (2,5), (4,1)\\
2 &  (1,2), (5,4) & 8 & (0,5), (0,1)\\
3 & (1,4), (5,2) & 9 & (1,3), (5,3) \\
4 & (3,4), (3,2) & 10 & (1,5), (5,1)\\
5 & (2,3), (4,3) &  11 & (1,1), (5,5)\\
6 & (2,1), (4,5) & 12 & (3,5), (3,1)\\
\hline 
\end{tabular}
\end{tabular}
\caption{Tuples $(k,l)$ modulo $m_t$ for $n=7$ (left) and $n=8$ (right)}
\label{tab:Tuples78}
\end{table}

\begin{Definition}\label{Def:SubposetD}
For $c$ a Coxeter transformation in a tubular elliptic Weyl group of type $D_4^{(1,1)}$, we define the subposet $[\idop, c]_D^{\text{gen}}$ of $[\idop, c]^{\text{gen}}$ to be the set of elements in $[\idop, c]^{\text{gen}}$ that are a prefix of an element in the Hurwitz orbit $R_1$ (see Remark~\ref{rmk:ObsMulti} (a)).
\end{Definition}

\begin{proof}[\textbf{Proof of Theorem \ref{cor:MainTheorem}}]
We have to show that conditions (a) and (b) of Theorem \ref{conj:WeightProjElliptic0} are fulfilled. Condition (a) holds by Theorem \ref{thm:MainElliptic}, while condition (b) holds by Theorem \ref{prop:GenPrefix}.
\end{proof}

\begin{remark}
    If $A$ is a finite dimensional hereditary $k$-algebra  of finite representation type, then by \cite[Remark 6.8]{Kra12}  each thick subcategory of $\MOD(A)$ is generated by an exceptional sequence, while not every thick subcategory of $\COH(\XX)$ is generated by an exceptional sequence (see \cite{Koe11} or \cite{Kra12}).
If $\COH(\XX)$ is of domestic type then it is derived equivalent to the category of finitely generated modules over a finite dimensional hereditary tame $k$-algebra. Therefore, K\"ohler's PhD thesis \cite{Koe11} provides a complete combinatorial description of the thick subcategories of $\COH(\XX)$ in the domestic type. In the corresponding module category she proved that the only thick subcategories are those that are generated by exceptional sequences and those that are the thick subcategories of the full subcategory consisting of the regular objects.
\end{remark}

\newpage
\appendix
\section*{Appendix}
\renewcommand{\thesection}{A} 

We use the following abbreviation to notate roots: For example, if we have a root system with simple roots $\alpha_1, \ldots, \alpha_6$, then we denote the root $\alpha_1+\alpha_2+2\alpha_3+2\alpha_4+\alpha_5+\alpha_6$ by $12334456$. In addition, we denote the highest root $\widetilde{\alpha}$ by $0$

\subsection{Root subsystems of type $A_2^3$ in $E_6$} \label{sec:3A2}
\begin{itemize}
\item $\{ 1,5\} \cup \{ 2,6\} \cup \{3,0 \}$
\item $\{ 4,2345\} \cup \{ 3456,1234\} \cup \{1345,2456 \}$
\item $\{ 245,123445\} \cup \{ 234,234456\} \cup \{345,123456 \}$
\item $\{ 34,12345\} \cup \{ 45,23456\} \cup \{24,1234456 \}$
\end{itemize}

\subsection{Root subsystems of type $A_3^2\times A_1$ and $D_4(a_1) \times A_1^3$ in $E_7$} \label{sec:subE7}

The six subsystems of type $A_3^2\times A_1$:
\begin{itemize}
\item $\{ 234456, 1233445567, 23445567\} \cup 
\{ 123445, 12344567, 1234455667\} \cup \{ 2 \}$
\item $\{ 5,6,7\} \cup \{ 1,3,0 \} \cup \{ 2\}$ 
\item $\{ 24,345,123456\} \cup \{ 456,234567,134\} \cup \{ 123445567\}$
\item $\{ 4,2345,13456\} \cup \{ 2456,34567,1234\} \cup \{ 123445567\}$
\item $\{ 45,23456,134567\} \cup \{ 24567,34,12345\} \cup \{ 1234456\}$
\item $\{ 245,3456,1234567\} \cup \{ 4567,234,1345\} \cup \{ 1234456\}$
\end{itemize}
The subsystem of type $D_4(a_1) \times A_1^3$:
\begin{itemize}
\item $\{ 2, 1234456, 123445567\} \cup \{ 56, 23445, 13, 234455667\}$
\end{itemize}

\subsection{Root subsystems of type $A_5 \times A_2 \times A_1$ in $E_8$} \label{sec:subE8}
\begin{itemize}
\item $\{ 5,6,7,8,0\} \cup \{ 1,3 \} \cup \{ 2\}$
\item $\{ 234456, 1233445567, 12344556678, 1234455667780, 1233445678\} \cup \{ 12344567, 234455678\} \cup \{ 2\}$
\item $\{ 12334456, 123445567, 234455667, 12334455667780, 123445678\} \cup \{ 2344567, 12334455678\} \cup \{ 2\}$
\item $\{ 1234456, 123445567, 123344556678, 1234455667780, 23445678\} \cup \{ 123344567, 1234455678\} \cup \{ 2\}$
\item $\{ 122334445567, 12334445556678, 123344455678, 12334445566778, 123344455667\} \cup \{ 1, 3\} \cup \{ 12233444556678\}$
\item $\{ 34, 12345, 456, 234567, 1345678\} \cup \{ 12344567, 234455678\} \cup \{ 12233444556678\}$
\item $\{ 134, 245, 3456, 1234567, 45678\} \cup \{ 2344567, 12334455678\} \cup \{ 12233444556678\}$
\item $\{ 4, 2345, 13456, 24567, 345678\} \cup \{ 123344567, 1234455678\} \cup \{ 12233444556678\}$
\item $\{ 12334445567, 122334445556678, 1223344455678, 123344455566780, 1223344455667\} \cup \{ 1, 3\} \cup \{ 1233444556678\}$
\item $\{ 234, 1345, 2456, 34567, 12345678\} \cup \{ 12344567, 234455678\} \cup \{ 1233444556678\}$
\item $\{ 1234, 45, 23456, 134567, 245678\} \cup \{ 2344567, 12334455678\} \cup \{ 1233444556678\}$
\item $\{ 24, 345, 123456, 4567, 2345678\} \cup \{ 123344567, 1234455678\} \cup \{ 1233444556678\}$
\end{itemize}

\newpage
\nocite{*}\textbf{}

\bibliography{mybibDerived}

\providecommand{\bysame}{\leavevmode\hbox to3em{\hrulefill}\thinspace}
\providecommand{\MR}{\relax\ifhmode\unskip\space\fi MR }
\providecommand{\MRhref}[2]{%
  \href{http://www.ams.org/mathscinet-getitem?mr=#1}{#2}
}
\providecommand{\href}[2]{#2}
\begin{thebibliography}{10}

\bibitem{BA97}
Bruce~N. Allison, Saeid Azam, Stephen Berman, Yun Gao, and Arturo Pianzola,
  \emph{Extended affine {L}ie algebras and their root systems}, Mem. Amer.
  Math. Soc. \textbf{126} (1997), no.~603, x+122. \MR{1376741}

\bibitem{HTT07}
Lidia Angeleri~H\"{u}gel, Dieter Happel, and Henning Krause, \emph{Basic
  results of classical tilting theory}, Handbook of tilting theory, London
  Math. Soc. Lecture Note Ser., vol. 332, Cambridge Univ. Press, Cambridge,
  2007, pp.~9--13. \MR{2384605}

\bibitem{Au55}
Maurice Auslander, \emph{On the dimension of modules and algebras. {III}.
  {G}lobal dimension}, Nagoya Math. J. \textbf{9} (1955), 67--77. \MR{74406}

\bibitem{ARS97}
Maurice Auslander, Idun Reiten, and Sverre~O. Smal\o, \emph{Representation
  theory of {A}rtin algebras}, Cambridge Studies in Advanced Mathematics,
  vol.~36, Cambridge University Press, Cambridge, 1997, Corrected reprint of
  the 1995 original. \MR{1476671}

\bibitem{BH19}
Sven Balnojan and Claus Hertling, \emph{Reduced and nonreduced presentations of
  {W}eyl group elements}, J. Lie Theory \textbf{29} (2019), no.~2, 559--599.
  \MR{3942566}

\bibitem{BG12}
Michael Barot and Christof Geiss, \emph{Tubular cluster algebras {I}:
  categorification}, Math. Z. \textbf{271} (2012), no.~3-4, 1091--1115.
  \MR{2945599}

\bibitem{BDSW14}
Barbara Baumeister, Matthew Dyer, Christian Stump, and Patrick Wegener, \emph{A
  note on the transitive {H}urwitz action on decompositions of parabolic
  {C}oxeter elements}, Proc. Amer. Math. Soc. Ser. B \textbf{1} (2014),
  149--154. \MR{3294251}

\bibitem{BGRW}
Barbara Baumeister, Thomas Gobet, Kieran Roberts, and Patrick Wegener, \emph{On
  the {H}urwitz action in finite {C}oxeter groups}, J. Group Theory \textbf{20}
  (2017), no.~1, 103--131. \MR{3592608}

\bibitem{BMN25}
Barbara Baumeister, Jon McCammond, and Georges Neaime, \emph{The extended
  affine artin group of type ${D}_4$ and its dual presentation}, in
  preparation.

\bibitem{BaW}
Barbara Baumeister and Patrick Wegener, \emph{A note on {W}eyl groups and root
  lattices}, Arch. Math. (Basel) \textbf{111} (2018), no.~5, 469--477.
  \MR{3859428}

\bibitem{BW24}
\bysame, \emph{Elliptic {W}eyl {G}roups and {H}urwitz {T}ransitivity},  (2024).

\bibitem{BWY24}
Barbara Baumeister, Patrick Wegener, and Sophiane Yahiatene, \emph{Extended
  {W}eyl groups, {H}urwitz transitivity and weighted projective lines {I}:
  Generalities and the tubular case}, arxiv (23).

\bibitem{BWY21}
\bysame, \emph{Extended {W}eyl groups, {H}urwitz transitivity and weighted
  projective lines {II}: The wild case}, arxiv (23).

\bibitem{Bes03}
David Bessis, \emph{The dual braid monoid}, Ann. Sci. \'{E}cole Norm. Sup. (4)
  \textbf{36} (2003), no.~5, 647--683. \MR{2032983}

\bibitem{Bo89}
Alexey~I. Bondal, \emph{Representations of associative algebras and coherent
  sheaves}, Izv. Akad. Nauk SSSR Ser. Mat. \textbf{53} (1989), no.~1, 25--44.
  \MR{992977}

\bibitem{Bon83}
Klaus Bongartz, \emph{Algebras and quadratic forms}, J. London Math. Soc. (2)
  \textbf{28} (1983), no.~3, 461--469. \MR{724715}

\bibitem{Bou02}
Nicolas Bourbaki, \emph{Lie groups and {L}ie algebras. {C}hapters 4--6},
  Elements of Mathematics (Berlin), Springer-Verlag, Berlin, 2002, Translated
  from the 1968 French original by Andrew Pressley. \MR{1890629}

\bibitem{Br01}
Thomas Brady, \emph{A partial order on the symmetric group and new
  {$K(\pi,1)$}'s for the braid groups}, Adv. Math. \textbf{161} (2001), no.~1,
  20--40. \MR{1857934}

\bibitem{Br07}
Kristian Br\"{u}ning, \emph{Thick subcategories of the derived category of a
  hereditary algebra}, Homology Homotopy Appl. \textbf{9} (2007), no.~2,
  165--176. \MR{2366948}

\bibitem{Car72}
Roger~W. Carter, \emph{Conjugacy classes in the {W}eyl group}, Compositio Math.
  \textbf{25} (1972), 1--59. \MR{318337}

\bibitem{CK09}
Xiao-Wu {Chen} and Henning {Krause}, \emph{{Introduction to coherent sheaves on
  weighted projective lines}}, arXiv e-prints (2009), arXiv:0911.4473.

\bibitem{CB92}
William Crawley-Boevey, \emph{Exceptional sequences of representations of
  quivers}, Proceedings of the {S}ixth {I}nternational {C}onference on
  {R}epresentations of {A}lgebras ({O}ttawa, {ON}, 1992), Carleton-Ottawa Math.
  Lecture Note Ser., vol.~14, Carleton Univ., Ottawa, ON, 1992, p.~7.
  \MR{1206935}

\bibitem{Di09}
Nikolay~D. Dichev, \emph{Thick subcategories for quiver representations}, Ph.D.
  thesis, Universit\"at Paderborn, Germany, 2009.

\bibitem{DPR}
J.~Matthew Douglass, G\"{o}tz Pfeiffer, and Gerhard R\"{o}hrle, \emph{On
  reflection subgroups of finite {C}oxeter groups}, Comm. Algebra \textbf{41}
  (2013), no.~7, 2574--2592. \MR{3169410}

\bibitem{DL11}
M.~J. Dyer and G.~I. Lehrer, \emph{Reflection subgroups of finite and affine
  {W}eyl groups}, Trans. Amer. Math. Soc. \textbf{363} (2011), no.~11,
  5971--6005. \MR{2817417}

\bibitem{GAP4}
The GAP~Group, \emph{{GAP -- Groups, Algorithms, and Programming, Version
  4.11.0}}, 2020.

\bibitem{GL87}
Werner Geigle and Helmut Lenzing, \emph{A class of weighted projective curves
  arising in representation theory of finite-dimensional algebras},
  Singularities, representation of algebras, and vector bundles ({L}ambrecht,
  1985), Lecture Notes in Math., vol. 1273, Springer, Berlin, 1987,
  pp.~265--297. \MR{915180}

\bibitem{GSZ01}
Edward~L. Green, {\O}yvind Solberg, and Dan Zacharia, \emph{Minimal projective
  resolutions}, Trans. Amer. Math. Soc. \textbf{353} (2001), no.~7, 2915--2939.
  \MR{1828479}

\bibitem{Hap01}
Dieter Happel, \emph{A characterization of hereditary categories with tilting
  object}, Invent. Math. \textbf{144} (2001), no.~2, 381--398. \MR{1827736}

\bibitem{HK16}
Andrew Hubery and Henning Krause, \emph{A categorification of non-crossing
  partitions}, J. Eur. Math. Soc. (JEMS) \textbf{18} (2016), no.~10,
  2273--2313. \MR{3551191}

\bibitem{Hum90}
James~E. Humphreys, \emph{Reflection groups and {C}oxeter groups}, Cambridge
  Studies in Advanced Mathematics, vol.~29, Cambridge University Press,
  Cambridge, 1990. \MR{1066460}

\bibitem{IS10}
Kiyoshi Igusa and Ralf Schiffler, \emph{Exceptional sequences and clusters}, J.
  Algebra \textbf{323} (2010), no.~8, 2183--2202. \MR{2596373}

\bibitem{IPT15}
Colin Ingalls, Charles Paquette, and Hugh Thomas, \emph{Semi-stable
  subcategories for {E}uclidean quivers}, Proc. Lond. Math. Soc. (3)
  \textbf{110} (2015), no.~4, 805--840. \MR{3335288}

\bibitem{IT09}
Colin Ingalls and Hugh Thomas, \emph{Noncrossing partitions and representations
  of quivers}, Compos. Math. \textbf{145} (2009), no.~6, 1533--1562.
  \MR{2575093}

\bibitem{Kac83}
Victor~G. Kac, \emph{Infinite-dimensional {L}ie algebras}, Progress in
  Mathematics, vol.~44, Birkh\"{a}user Boston, Inc., Boston, MA, 1983, An
  introduction. \MR{739850}

\bibitem{Klu87}
Paul Kluitmann, \emph{Ausgezeichnete {B}asen erweiterter affiner
  {W}urzelgitter}, Bonner Mathematische Schriften, vol. 185, Universit\"{a}t
  Bonn, Mathematisches Institut, Bonn, 1987, Dissertation, Rheinische
  Friedrich-Wilhelms-Universit\"{a}t, Bonn, 1986. \MR{930668}

\bibitem{Koe11}
Claudia K{\"o}hler, \emph{A combinatorial classification of thick subcategories
  of derived and stable categories}, Ph.D. thesis, Universit\"at Bielefeld,
  Germany, 2011.

\bibitem{Kra12}
Henning Krause, \emph{Report on locally finite triangulated categories}, J.
  K-Theory \textbf{9} (2012), no.~3, 421--458. \MR{2955969}

\bibitem{KM02}
Dirk Kussin and Hagen Meltzer, \emph{The braid group action for exceptional
  curves}, Arch. Math. (Basel) \textbf{79} (2002), no.~5, 335--344.
  \MR{1951302}

\bibitem{Len99}
Helmut Lenzing, \emph{Coxeter transformations associated with
  finite-dimensional algebras}, Computational methods for representations of
  groups and algebras ({E}ssen, 1997), Progr. Math., vol. 173, Birkh\"{a}user,
  Basel, 1999, pp.~287--308. \MR{1714618}

\bibitem{LMPS19}
Joel~Brewster Lewis, Jon McCammond, T.~Kyle Petersen, and Petra Schwer,
  \emph{Computing reflection length in an affine {C}oxeter group}, Trans. Amer.
  Math. Soc. \textbf{371} (2019), no.~6, 4097--4127. \MR{3917218}

\bibitem{LR16}
Joel~Brewster Lewis and Victor Reiner, \emph{Circuits and {H}urwitz action in
  finite root systems}, New York J. Math. \textbf{22} (2016), 1457--1486.
  \MR{3603073}

\bibitem{MacD72}
Ian~G. Macdonald, \emph{Affine root systems and {D}edekind's {$\eta
  $}-function}, Invent. Math. \textbf{15} (1972), 91--143. \MR{357528}

\bibitem{MP11}
Jon McCammond and T.~Kyle Petersen, \emph{Bounding reflection length in an
  affine {C}oxeter group}, J. Algebraic Combin. \textbf{34} (2011), no.~4,
  711--719. \MR{2842917}

\bibitem{Mel04}
Hagen Meltzer, \emph{Exceptional vector bundles, tilting sheaves and tilting
  complexes for weighted projective lines}, Mem. Amer. Math. Soc. \textbf{171}
  (2004), no.~808, viii+139. \MR{2074151}

\bibitem{PS20}
Giovanni Paolini and Mario Salvetti, \emph{Proof of the {$K(\pi,1)$} conjecture
  for affine {A}rtin groups}, Invent. Math. \textbf{224} (2021), no.~2,
  487--572. \MR{4243019}

\bibitem{Rei98}
Idun Reiten, \emph{Tilting theory and quasitilted algebras}, Proceedings of the
  {I}nternational {C}ongress of {M}athematicians, {V}ol. {II} ({B}erlin, 1998),
  no. Extra Vol. II, 1998, pp.~109--120. \MR{1648061}

\bibitem{RI84}
Claus~Michael Ringel, \emph{Tame algebras and integral quadratic forms},
  Lecture Notes in Mathematics, vol. 1099, Springer-Verlag, Berlin, 1984.
  \MR{774589}

\bibitem{RI94}
\bysame, \emph{The braid group action on the set of exceptional sequences of a
  hereditary {A}rtin algebra}, Abelian group theory and related topics
  ({O}berwolfach, 1993), Contemp. Math., vol. 171, Amer. Math. Soc.,
  Providence, RI, 1994, pp.~339--352. \MR{1293154}

\bibitem{Sai85}
Kyoji Saito, \emph{Extended affine root systems. {I}. {C}oxeter
  transformations}, Publ. Res. Inst. Math. Sci. \textbf{21} (1985), no.~1,
  75--179. \MR{780892}

\bibitem{ST97}
Kyoji Saito and Tadayoshi Takebayashi, \emph{Extended affine root systems.
  {III}. {E}lliptic {W}eyl groups}, Publ. Res. Inst. Math. Sci. \textbf{33}
  (1997), no.~2, 301--329. \MR{1442503}

\bibitem{SY00}
Kyoji Saito and Daigo Yoshii, \emph{Extended affine root system. {IV}.
  {S}imply-laced elliptic {L}ie algebras}, Publ. Res. Inst. Math. Sci.
  \textbf{36} (2000), no.~3, 385--421. \MR{1781435}

\bibitem{STW16}
Yuuki Shiraishi, Atsushi Takahashi, and Kentaro Wada, \emph{On {W}eyl groups
  and {A}rtin groups associated to orbifold projective lines}, J. Algebra
  \textbf{453} (2016), 249--290. \MR{3465355}

\bibitem{Sun17}
Chao Sun, \emph{Bounded t-structures on the bounded derived category of
  coherent sheaves over a weighted projective line}, Algebr. Represent. Theory
  \textbf{23} (2020), no.~6, 2167--2235. \MR{4165014}

\bibitem{PW17}
Patrick Wegener, \emph{On the {H}urwitz action in affine {C}oxeter groups}, J.
  Pure Appl. Algebra \textbf{224} (2020), no.~7, 106308. \MR{4058242}

\bibitem{WY19}
Patrick Wegener and Sophiane Yahiatene, \emph{A note on non-reduced reflection
  factorizations of {C}oxeter elements}, Algebr. Comb. \textbf{3} (2020),
  no.~2, 465--469. \MR{4099003}

\end{thebibliography}
\bibliographystyle{amsplain}

\end{document}